\documentclass[a4paper]{article}
\usepackage{amsmath,amsthm,amsfonts,thmtools}
\usepackage{amssymb}
\usepackage[cmintegrals]{newtxmath}
\usepackage{graphicx}
\usepackage{listings}
\usepackage[font=footnotesize]{caption}
\usepackage[font=scriptsize]{subcaption}
\usepackage{euscript}
\usepackage{units}
\usepackage{enumerate}
\usepackage{enumitem}
\usepackage{xcolor}
\usepackage{todonotes}
\usepackage{booktabs}
\usepackage{tikz}
\usepackage{comment}
 
\DeclareMathAlphabet{\pazocal}{OMS}{zplm}{m}{n}

\usepackage[bookmarks=true,hidelinks]{hyperref}
\usepackage{bbm}
\usepackage{multirow}

\usepackage{mathtools}
\usepackage{subcaption} 
\usepackage{graphicx}
\usepackage{pgfplots}
\usepackage{float}
\usepackage{algpseudocode}
\usepackage{bbold}
\usepackage{cite}
\usepackage{microtype}
\usepackage{rotating}
\usepackage{enumitem}
\usepackage{bm}
\usepackage{hyperref}

\numberwithin{equation}{section}
\newtheorem{theorem}{Theorem}[section]

\newtheorem{lemma}[theorem]{Lemma}

\newtheorem{assumption}[theorem]{Assumption}
\newtheorem{algorithm}[theorem]{Algorithm}
\newtheorem{definition}[theorem]{Definition}

\numberwithin{equation}{section}

\theoremstyle{definition}

\newtheoremstyle{myremarkstyle}{}{}{}{}{\bfseries}{.}{ }{}
\theoremstyle{myremarkstyle}
\declaretheorem[name=Remark,qed=$\blacksquare$,numberlike=theorem]{remark}

\makeatletter
\newcommand*{\intavg}{%
  \mint@l{-}{}%
}
\newcommand*{\mint@l}[2]{%
  \@ifnextchar\limits{%
    \mint@l{#1}%
  }{%
    \@ifnextchar\nolimits{%
      \mint@l{#1}%
    }{%
      \@ifnextchar\displaylimits{%
        \mint@l{#1}%
      }{%
        \mint@s{#2}{#1}%
      }%
    }%
  }%
}
\newcommand*{\mint@s}[2]{%
  \@ifnextchar_{%
    \mint@sub{#1}{#2}%
  }{%
    \@ifnextchar^{%
      \mint@sup{#1}{#2}%
    }{%
      \mint@{#1}{#2}{}{}%
    }%
  }%
}
\def\mint@sub#1#2_#3{%
  \@ifnextchar^{%
    \mint@sub@sup{#1}{#2}{#3}%
  }{%
    \mint@{#1}{#2}{#3}{}%
  }%
}
\def\mint@sup#1#2^#3{%
  \@ifnextchar_{%
    \mint@sub@sup{#1}{#2}{#3}%
  }{%
    \mint@{#1}{#2}{}{#3}%
  }%
}
\def\mint@sub@sup#1#2#3^#4{%
  \mint@{#1}{#2}{#3}{#4}%
}
\def\mint@sup@sub#1#2#3_#4{%
  \mint@{#1}{#2}{#4}{#3}%
}
\newcommand*{\mint@}[4]{%
  \mathop{}%
  \mkern-\thinmuskip
  \mathchoice{%
    \mint@@{#1}{#2}{#3}{#4}%
        \displaystyle\textstyle\scriptstyle
  }{%
    \mint@@{#1}{#2}{#3}{#4}%
        \textstyle\scriptstyle\scriptstyle
  }{%
    \mint@@{#1}{#2}{#3}{#4}%
        \scriptstyle\scriptscriptstyle\scriptscriptstyle
  }{%
    \mint@@{#1}{#2}{#3}{#4}%
        \scriptscriptstyle\scriptscriptstyle\scriptscriptstyle
  }%
  \mkern-\thinmuskip
  \int#1%
  \ifx\\#3\\\else_{#3}\fi
  \ifx\\#4\\\else^{#4}\fi  
}
\newcommand*{\mint@@}[7]{%
  \begingroup
    \sbox0{$#5\int\m@th$}%
    \sbox2{$#5\int_{}\m@th$}%
    \dimen2=\wd0 %
    \let\mint@limits=#1\relax
    \ifx\mint@limits\relax
      \sbox4{$#5\int_{\kern1sp}^{\kern1sp}\m@th$}%
      \ifdim\wd4>\wd2 %
        \let\mint@limits=\nolimits
      \else
        \let\mint@limits=\limits
      \fi
    \fi
    \ifx\mint@limits\displaylimits
      \ifx#5\displaystyle
        \let\mint@limits=\limits
      \fi
    \fi
    \ifx\mint@limits\limits
      \sbox0{$#7#3\m@th$}%
      \sbox2{$#7#4\m@th$}%
      \ifdim\wd0>\dimen2 %
        \dimen2=\wd0 %
      \fi
      \ifdim\wd2>\dimen2 %
        \dimen2=\wd2 %
      \fi
    \fi
    \rlap{%
      $#5%
        \vcenter{%
          \hbox to\dimen2{%
            \hss
            $#6{#2}\m@th$%
            \hss
          }%
        }%
      $%
    }%
  \endgroup
}

\def\XXint#1#2#3{{\setbox0=\hbox{$#1{#2#3}{\int}$ }
		\vcenter{\hbox{$#2#3$ }}\kern-.6\wd0}}

\renewcommand{\div}{{\rm div}}
\renewcommand{\geq}{\geqslant}
\renewcommand{\leq}{\leqslant}

\renewcommand{\epsilon}{\varepsilon}
\renewcommand{\phi}{\varphi}

\newcommand{\nl}{{\mathbf n}}

\renewcommand{\div}{{\rm div}}

\newcommand{\R}{\mathbb{R}}
\newcommand{\N}{\mathbb{N}}

\newcommand{\bu}{{\bf u}}		% Phase space

\newcommand{\f}{\mathbf{f}}
\newcommand{\g}{\mathbf{g}}

\renewcommand{\g}{{\bf g}}

\newcommand{\map}{\EuScript{L}}
\newcommand{\bc}{\EuScript{B}}
\newcommand{\train}{\EuScript{S}}

\newcommand{\er}{\EuScript{E}}

\newcommand{\cost}{\EuScript{C}}

\newcommand{\df}{\EuScript{D}}
\newcommand{\dom}{\mathbb{D}}
\newcommand{\bv}{\mathbf{v}}

\newcommand{\res}{\EuScript{R}}

\newcommand{\bD}{\partial D}
\newcommand{\bdom}{\partial \dom}
\newcommand{\tr}{\EuScript{T}}

\begin{document}

\date{\today}

\title{Estimates on the generalization error of Physics Informed Neural Networks (PINNs) for approximating \\a class of inverse problems for PDEs}

\author{Siddhartha Mishra \thanks{Seminar for Applied Mathematics (SAM), D-Math \newline
  ETH Z\"urich, R\"amistrasse 101, 
  Z\"urich-8092, Switzerland} and
  Roberto Molinaro \thanks{Seminar for Applied Mathematics (SAM), D-Math \newline
  ETH Z\"urich, R\"amistrasse 101, 
  Z\"urich-8092, Switzerland.}}

\maketitle
\begin{abstract}
Physics informed neural networks (PINNs) have recently been very successfully applied for efficiently approximating inverse problems for PDEs. We focus on a particular class of inverse problems, the so-called data assimilation or unique continuation problems, and prove rigorous estimates on the generalization error of PINNs approximating them. An abstract framework is presented and conditional stability estimates for the underlying inverse problem are employed to derive the estimate on the PINN generalization error, providing rigorous justification for the use of PINNs in this context. The abstract framework is illustrated with examples of four prototypical linear PDEs. Numerical experiments, validating the proposed theory, are also presented.

\end{abstract}

\section{Introduction}
Partial Differential Equations (PDEs) are widely used in modeling of phenomena of interest in physics, biology and engineering. A wide variety of numerical methods, such as finite difference, finite element, finite volume and spectral methods, have been devised to numerically simulate solutions of different types of PDEs. Despite their widespread use, these methods may not suffice to solve many \emph{large scale} problems with PDEs such as uncertainty quantification (UQ), PDEs in very high dimensions, deterministic and Bayesian inverse problems and PDEs with multiple scales and modeling multiphysics. A key challenge in the simulation of these problems is the very high computational cost, even on state of the art high performance computing platforms. 

\emph{Deep learning}, based on \emph{deep neural networks} (DNNs) \cite{DLbook}, i.e. functions formed by compositions of affine transformations and scalar non-linear activation functions, has been very successful in tackling diverse tasks in science and engineering \cite{DLnat} such as at image and text classification, computer vision, text and speech recognition, autonomous systems and robotics, game intelligence and even protein folding \cite{Dfold}.

Deep learning is being increasingly used in the numerical approximations of PDEs. A brief and very incomplete survey of this rapidly growing body of literature is as follows, one approach in using deep neural networks for numerically approximating PDEs is based on explicit (or semi-implicit) representation formulas such as the Feynman-Kac formula for parabolic (and elliptic) PDEs, whose compositional structure is in turn utilized for approximation by DNNs. This approach is presented and analyzed for a variety of (parametric) elliptic, parabolic and linear transport PDEs in \cite{E1,HEJ1,Jent1,Pet1} and references therein.

Another strategy is to augment existing numerical methods by adding deep learning based modules to them, for instance by learning free parameters of numerical schemes from data \cite{SM1,DR1} and references therein. 

A third approach is to use deep neural networks to learn \emph{observables} or quantities of interest of the solutions of the underlying PDEs, from data. This approach has been described in \cite{LMR1,LMM1,MR1,LMPR1} in the context of uncertainty quantification and PDE constrained optimization and \cite{QUAT1} for model order reduction, among others. 

As deep neural networks possess the so-called \emph{universal approximation property} i.e any continuous function can be approximated by DNNs, see \cite{Bar1,Hor1,Cy1,YAR1} and references therein, it is natural to use them as ansatz spaces for the solutions of PDEs, in particular by collocating the PDE residual at suitably chosen training points. This approach was first proposed in \cite{Lag1,Lag2}. More recently, it has been revived and developed in significantly greater detail by the pioneering contributions of Karniadakis and collaborators, starting with \cite{KAR1,KAR2}. These authors have termed the underlying neural networks as \emph{Physics Informed Neural Networks} (PINNs) and we will continue to use this by now widely accepted terminology in this paper. There has been an explosive growth of papers that present PINNs for different applications and a very incomplete list of references includes \cite{KAR4,KAR5,KAR6,KAR7,KAR8, shukla, jag1, jag2}. 

PINNs have already emerged as an attractive framework for approximating solutions of PDEs. They are particularly successful for approximating the so-called \emph{inverse problems} involving PDEs, see \cite{KAR2,KAR4,KAR6,KAR8} and references therein. For such inverse problems, one does not necessarily have complete information on the inputs to the PDEs, such as initial data, boundary conditions and coefficients. Hence, the so-called \emph{forward problem} cannot be solved uniquely. However, we have access to (noisy) data for (observables) of the underlying solution. The aim is to use this data in order to determine the unknown inputs of the PDEs and consequently its solution. PINNs are very promising in solving inverse problems efficiently on account of their ability to integrate data and PDEs. This is seen in the robustness and accuracy of the results (see \cite{KAR2,KAR4,KAR6} and references therein). Another very attractive feature of PINNs is simplicity of their implementation. In particular, essentially the same code (with very minor modifications) can be used to approximate both forward and inverse problems. 

Why are PINNs so efficient at approximating solutions to both forward and inverse problems for PDEs? Very few rigorous results on the approximation error with PINNs are available, with two recent notable exceptions. In \cite{DAR1}, the authors investigate convergence of PINNs to solutions of the forward problem for linear elliptic and parabolic PDEs, on increasing the number of training samples. They prove convergence under certain assumptions on the training loss, based on H\"older regularized loss functions. 

On the other hand, in another recent paper \cite{MM1}, the authors estimate the generalization (approximation) error for PINNs approximating forward problems for a large class of linear and nonlinear PDEs in terms of the training error and the number of training samples. The key point in \cite{MM1} is the leveraging of \emph{stability estimates} for classical solutions of the underlying PDE, together with estimates on quadrature errors, to derive upper bounds on the generalization error of PINNs. 

The aforementioned articles provide rigorous justification for using PINNs to approximate forward problems for certain PDEs. However, to the best of the authors' knowledge, there is no rigorous estimate for the approximation error of PINNs for inverse problems. The main aim of this article is to \emph {derive rigorous bounds on the generalization errors of PINNs for approximating solutions to  a class of inverse problems for PDEs}. We will focus on the so-called \emph{unique continuation} or \emph{data assimilation} problems, where some inputs to the PDEs are unknown and need to be inferred from (possibly noisy) measurements of certain observables of the underlying solution. These are exactly the class of inverse problems, which were considered in recent papers \cite{KAR2,KAR4,KAR6} and references therein. 

Our strategy for deriving bounds on the generalization errors follows the recent paper \cite{MM1} and is based on \emph{conditional stability estimates} on the underlying solution of the inverse problem. Such stability estimates have been widely investigated in the theoretical PDE literature and can be derived for a variety of linear and non-linear PDEs, for instance from the well-known \emph{three-balls} inequalities \cite{ALE1} and references therein, and \emph{Carleman estimates} \cite{IBook} and references therein. We will formalize a very general framework for the unique continuation (data assimilation) inverse problem for an abstract PDE, with a wide variety of potential applications. Our derived generalization error estimate will show that the error due to approximating the underlying inverse problem with a trained PINN will be sufficiently low as long as 
\begin{itemize}
    \item The training error is low i.e, the PINN has been trained well. This error is computed and monitored during the training process. Hence, it is available \emph{a posteriori}.
    \item The number of training (collocation) points is sufficiently large. This number is determined by the error due to an underlying quadrature rule.
    \item The solution to the underlying inverse problem is \emph{conditionally stable} i.e, it is stable as long as the underlying solution and the approximating PINN are sufficiently regular. 
\end{itemize}
Thus, with the derived error estimate, we identify possible mechanisms by which the PINNs are able to approximate inverse problems for PDEs so well and provide a rigorous justification for their successful use. 

We will exemplify the proposed abstract framework by applying to four well-known model problems, namely the data assimilation problems for the Poisson, Heat, Wave and Stokes equations, covering the whole spectrum of linear elliptic, parabolic, hyperbolic and indefinite PDEs. For each of these examples, we derive bounds on the generalization error and also provide numerical experiments to validate the proposed theory.

The rest of the paper is organized as follows, in section \ref{sec:2}, we describe the PINNs algorithm for the unique continuation (data assimilation) problem for an abstract PDE and derive an estimate on the generalization error. In sections \ref{sec:3},\ref{sec:4},\ref{sec:5} and \ref{sec:6}, we consider the Poisson, Heat, Wave and Stokes equations, respectively. 
\section{An abstract framework for Physics informed Neural Networks approximating inverse problems for PDEs}
\label{sec:2}
\subsection{The underlying abstract PDE}
\label{sec:21}
Let $\dom \subset \R^{\bar{d}}$, for some $\bar{d} \geq 1$, be the underlying domain, with smooth ($C^1$) boundary, denoted by $\partial \dom$. We include space-time domains by setting 
$\dom = (0,T) \times D \subset \R^d$ with $d \geq 1$. In this case $\bar{d} = d +1$ and the \emph{space-time boundary} is $\bdom = \bD \times (0,T) \cup D \times \{t=0\}$, with $\bD$ denoting the smooth boundary of $D$. Let $X = L^{p_x}(\dom;\R^m)$, $Y = L^{p_y}(\dom;\R^m)$ and $Z = L^{p_z}(\dom^{\prime};\R^m)$, for some $\dom^{\prime} \subset \dom$, with $m \geq 1$ and $1 \leq p_x,p_y,p_z < \infty$ be the underlying spaces and $X^{\ast} \subset X$, $Y^{\ast} \subset Y$, $Z^{\ast} \subset Z$ be Banach spaces. Let $\bar{X} \subset L^{\bar{p}_x}(\bdom;\R^m)$ and 
$\bar{Y} \subset L^{\bar{p}_y}(\bdom;\R^m)$, with $1 \leq \bar{p}_x,\bar{p}_y < \infty$  be Banach spaces.
\subsubsection{The Forward problem}
The abstract PDE is represented by the following differential equation,
\begin{equation}
    \label{eq:pde}
    \df(\bu) = \f.
\end{equation}
Here, the \emph{differential operator} is a mapping, $\df: X^{\ast} \mapsto Y^{\ast}$ and the \emph{input} $\f \in Y^{\ast}$, such that 
\begin{equation}
\label{eq:assm1}
\begin{aligned}
&(H1): \quad \|\df(\bu)\|_{Y^{\ast}} < +\infty, \quad \forall~ \bu \in X^{\ast}, ~{\rm with}~\|\bu\|_{X^{\ast}} < +\infty. \\
&(H2):\quad \|\f\|_{Y^{\ast}} < +\infty. 
\end{aligned}
\end{equation}
The PDE \eqref{eq:pde} is equipped with the \emph{boundary conditions},
\begin{equation}
    \label{eq:bc}
    \bc\left(\tr\bu\right) = \f_b,
\end{equation}
with \emph{trace operator} $\tr:X^{\ast} \mapsto \bar{X}$ and a generic boundary operator $\bc:\bar{X} \mapsto \bar{Y}$, with each of these operators being bounded under the corresponding norms. We note that the generic boundary condition also includes initial conditions when $\dom = (0,T) \times D \subset \R^d$. We assume that $\f_b \in \bar{Y}$.

We further assume that the forward problem \eqref{eq:pde}, \eqref{eq:bc} is well-posed i.e, given $\f \in Y^{\ast}, \f_b \in \bar{Y}$, there exists an unique $\bu \in X^{\ast}$ such that \eqref{eq:pde} and \eqref{eq:bc} are satisfied. 
\subsubsection{The Inverse problem}
As mentioned in the introduction, we will focus on the following inverse problem in this article, we assume that the generic boundary conditions (which include initial conditions) $\f_b$ in \eqref{eq:bc} are not known. Clearly, the forward problem for the PDE \eqref{eq:pde} will be ill-posed in this case. In particular, no unique solution can be guaranteed. Furthermore, we assume that we have access to (noiseless) measurements of the underlying solution $\bu$ in a sub-domain $\dom^{\prime} \subset \dom$ i.e,
\begin{equation}
    \label{eq:dt}
    \map(\bu) = \g, \quad \forall x \in \dom^{\prime},
\end{equation}
with the \emph{restriction operator} $\map: X^{\ast} \mapsto Z^{\ast}$ and \emph{data} $\g \in Z^{\ast}$, with
\begin{equation}
\label{eq:assm2}
\begin{aligned}
&(H3): \quad \|\map(\bu)\|_{Z^{\ast}} < +\infty, \quad \forall~ \bu \in X^{\ast}, ~{\rm with}~\|\bu\|_{X^{\ast}} < +\infty. \\
&(H4):\quad \|\g\|_{Z^{\ast}} < +\infty. 
\end{aligned}
\end{equation}
We term $\dom^{\prime}$ as the \emph{observation domain}. Thus, in solving the inverse problem \eqref{eq:pde}, \eqref{eq:dt}, one determines the function $\bu \in X^{\ast}$ and consequently the boundary conditions $\tr(\bu)$ from the data \eqref{eq:dt}, given only on the observation domain. This inverse problem is often referred to in the literature as the \emph{unique continuation} or the \emph{data assimilation} problem,\cite{BO1} and references therein.  

We further assume that solutions to the inverse problem \eqref{eq:pde}, \eqref{eq:dt}, satisfy the following \emph{conditional stability estimate},
\begin{assumption}
Let $\hat{X} \subset X^{\ast} \subset X = L^{p_x}(\dom)$ be a Banach space. For any $\bu,\bv \in \hat{X}$, the differential operator $\df$ and restriction operator $\map$ satisfy,
\begin{equation}
    \label{eq:assm}
(H5):\quad    \|\bu - \bv\|_{L^{p_x}(E)} \leq C_{pd}\left(\|\bu\|_{\hat{X}}, \|\bv\|_{\hat{X}} \right) \left(\|\df(\bu) - \df(\bv)\|^{\tau_p}_{Y} + \|\map(\bu) - \map(\bv)\|^{\tau_d}_{Z} \right),
\end{equation}
\end{assumption}
for some $0 < \tau_p,\tau_d \leq 1$ and for any subset $\dom^{\prime} \subset E \subset \dom$. This bound \eqref{eq:assm} is termed a \emph{conditional stability estimate} as it presupposes that the underlying solutions have sufficiently regularity as $\hat{X} \subset X^{\ast} \subset X$. 
\begin{remark}
We can extend the hypothesis for the inverse problem \eqref{eq:pde},\eqref{eq:bc} in the following ways,
\begin{itemize}
    \item Allow the measurement set $\dom^{\prime}$ to intersect the boundary i.e, $\partial \dom^{\prime} \cap \bdom \neq \Phi$.
    \item Replace the bound \eqref{eq:assm} with the weaker bound,
    \begin{equation}
    \label{eq:assm3}
(H6):\quad    \|\bu - \bv\|_{L^{p_x}(E)} \leq C_{pd}\left(\|\bu\|_{\hat{X}}, \|\bv\|_{\hat{X}} \right) \omega \left(\|\df(\bu) - \df(\bv)\|_{Y} + \|\map(\bu) - \map(\bv)\|_{Z} \right),
\end{equation}
    with $\omega:\R \mapsto \R_+$ being a modulus of continuity. 
\end{itemize}

\end{remark}
We will provide several examples for this abstract Inverse problem satisfying the assumptions $H1-H6$ in the following sections. 
\subsubsection{Literature on the unique continuation (data assimilation) problem}
The unique continuation or data assimilation problem \eqref{eq:pde}, \eqref{eq:dt}, has been extensively studied by the theoretical PDE community for a long time, arguably starting with the influential work of Hadamard on the ill-posedness of the elliptic Cauchy problem. An extensive survey of the available literature on the unique continuation problem, particularly highlighting applications of the three balls inequality can be read from \cite{ALE1} and references therein. Similarly, the vast body of literature on use of Carleman estimates to study well-posedness of the unique continuation (data assimilation) problem for evolutionary PDEs has been surveyed in \cite{IBook}, with particularly emphasis on applications to controllability of PDEs. 

Algorithms to numerically approximate solutions of the unique continuation problems are mostly based on quasi-reversibility methods \cite{LL,Bour} and penalty methods \cite{BCD}, relying on Tikhonov regularization and its variants. More recently, regularized and stabilized non-confirming finite element methods have also been developed to deal with different examples of the data assimilation problem in \cite{B1,B2,BO1,BO2,BO3} and references therein. We will provide an alternative approximation framework, based on physics informed neural networks (PINNs) in this paper.

\subsection{Quadrature rules}
\label{sec:22}
In the following section, we will need to approximate integrals of functions. Hence, we need an abstract formulation for quadrature. To this end, we consider a mapping $h: \dom \mapsto \R^m$, such that $h \in \underline{Y} \subset Y^{\ast} \subset Y$. We are interested in approximating the integral,
$$
\overline{h}:= \int\limits_{\dom} h(y) dy,
$$
with $dy$ denoting the $\bar{d}$-dimensional Lebesgue measure. In order to approximate the above integral by a quadrature rule, we need to specify the quadrature points $y_{i} \in \dom$ for $1 \leq i \leq N$, for some $N \in \N$ as well as weights $w_i$, with $w_i \in \R_+$. Then the quadrature is defined by,
\begin{equation}
    \label{eq:quad}
    \overline{h}_N := \sum\limits_{i=1}^N w_i h(y_i).
\end{equation}
We further assume that the quadrature error is bounded as,
\begin{equation}
    \label{eq:qassm}
    \left|\overline{h} - \overline{h}_N\right| \leq C_{q}
    \left(\|h\|_{\underline{Y}},\bar{d} \right) N^{-\alpha},
\end{equation}
for some $\alpha > 0$. 

As long as the domain $\dom$ is in reasonably low dimension i.e $\bar{d} \leq 4$, we can use standard (composite) Gauss quadrature rules on an underlying grid. In this case, the quadrature points and weights depend on the underlying order of the quadrature rule \cite{SBbook} and the rate $\alpha$ depends on the regularity of the underlying integrand i.e, on the space $\underline{Y}$.

On the other hand, these grid based quadrature rules are not suitable for domains in high dimensions. For moderately high dimensions i.e $4 \leq \bar{d} \approx 20$, we can use \emph{low discrepancy sequences}, such as the Sobol and Halton sequences, as quadrature points \cite{CAF1}. As long as the integrand $h$ is of bounded \emph{Hardy-Krause variation} \cite{owen}, the error in \eqref{eq:qassm} converges at a rate $(\log(N))^{\bar{d}}N^{-1}$. One can also employ sparse grids and Smolyak quadrature rules \cite{sgbook} in this regime. 

For problems in very high dimensions $\bar{d} \gg 20$, Monte-Carlo quadrature is the numerical integration method of choice \cite{CAF1}. In this case, the quadrature points are randomly chosen, independent and identically distributed (with respect to a scaled Lebesgue measure). The estimate \eqref{eq:qassm} holds in the root mean square (RMS) sense and the rate of convergence is $\alpha = \frac{1}{2}$.

We will also need integrals of functions $h_d: \dom^{\prime} \subset \dom$ i.e,
$$
\overline{h_d}:= \int\limits_{\dom^{\prime}} h_d(y) dy.
$$
In principle, a different set of quadrature points $z_{j} \in \dom^{\prime}$ for $1 \leq j \leq N_d$, for some $N_d \in \N$ as well as weights $w^d_j$, can be chosen. Then the quadrature is defined by,
\begin{equation}
    \label{eq:dqd}
    \overline{h}_{d,N} := \sum\limits_{j=1}^{N_d} w^d_j h_d(z_j).
\end{equation}
Analogous to \eqref{eq:qassm}, the resulting error is bounded by,
\begin{equation}
    \label{eq:dqassm}
    \left|\overline{h}_d - \overline{h}_{d,N}\right| \leq C_{qd}
    \left(\|h\|_{\underline{Z}},\bar{d} \right) N_d^{-\alpha_d},
\end{equation}
for some $\alpha_d > 0$ and for $\underline{Z} \subset Z$. 
\subsection{PINNs}
\label{sec:23}
In this section, we will describe physics-informed neural networks (PINNs) for approximating solutions of the inverse problem \eqref{eq:pde}, \eqref{eq:dt} in the following steps.
\subsubsection{Neural Networks.}
Given an input $y \in \dom$, a feedforward neural network (also termed as a multi-layer perceptron), shown in figure \ref{fig:1}, transforms it to an output, through a layer of units (neurons) which compose of either affine-linear maps between units (in successive layers) or scalar non-linear activation functions within units \cite{DLbook}, resulting in the representation,
\begin{equation}
\label{eq:ann1}
\bu_{\theta}(y) = C_K \circ\sigma \circ C_{K-1}\ldots \ldots \ldots \circ\sigma \circ C_2 \circ \sigma \circ C_1(y).
\end{equation} 
Here, $\circ$ refers to the composition of functions and $\sigma$ is a scalar (non-linear) activation function. A large variety of activation functions have been considered in the machine learning literature \cite{DLbook}. Popular choices for the activation function $\sigma$ in \eqref{eq:ann1} include the sigmoid function, the hyperbolic tangent function and the \emph{ReLU} function.

For any $1 \leq k \leq K$, we define
\begin{equation}
\label{eq:C}
C_k z_k = W_k z_k + b_k, \quad {\rm for} ~ W_k \in \R^{d_{k+1} \times d_k}, z_k \in \R^{d_k}, b_k \in \R^{d_{k+1}}.
\end{equation}
For consistency of notation, we set $d_1 = \bar{d}$ and $d_K = m$. 

Thus in the terminology of machine learning (see also figure \ref{fig:1}), our neural network \eqref{eq:ann1} consists of an input layer, an output layer and $(K-1)$ hidden layers for some $1 < K \in \N$. The $k$-th hidden layer (with $d_k$ neurons) is given an input vector $z_k \in \R^{d_k}$ and transforms it first by an affine linear map $C_k$ \eqref{eq:C} and then by a nonlinear (component wise) activation $\sigma$. A straightforward addition shows that our network contains $\left(\bar{d} + m + \sum\limits_{k=2}^{K-1} d_k\right)$ neurons. 
We also denote, 
\begin{equation}
\label{eq:theta}
\theta = \{W_k, b_k\}, \theta_W = \{ W_k \}\quad \forall~ 1 \leq k \leq K,
\end{equation} 
to be the concatenated set of (tunable) weights for our network. It is straightforward to check that $\theta \in \Theta \subset \R^M$ with
\begin{equation}
\label{eq:ns}
M = \sum\limits_{k=1}^{K-1} (d_k +1) d_{k+1}.
\end{equation}
 \begin{figure}[htbp]
\centering
\includegraphics[width=8cm]{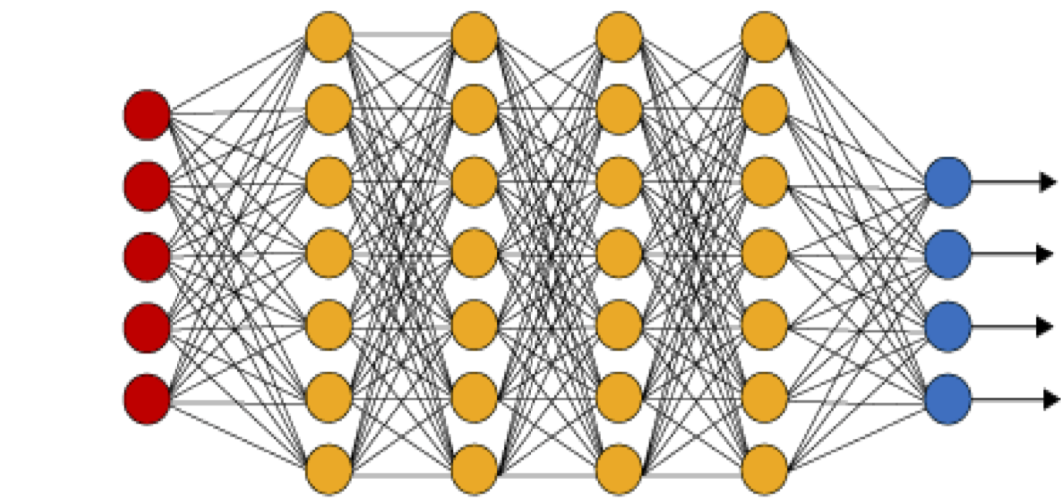}
\caption{An illustration of a (fully connected) deep neural network. The red neurons represent the inputs to the network and the blue neurons denote the output layer. They are
connected by hidden layers with yellow neurons. Each hidden unit (neuron) is connected by affine linear maps between units in different layers and then with nonlinear (scalar) activation functions within units.}
\label{fig:1}
\end{figure}
\subsubsection{Training PINNs: Loss functions and optimization}
The neural network $\bu_{\theta}$ \eqref{eq:ann1} depends on the tuning parameter $\theta \in \Theta$ of weights and biases. Within the standard paradigm of \emph{deep learning} \cite{DLbook}, one \emph{trains} the network by finding tuning parameters $\theta$ such that the loss (error, mismatch, regret) between the neural network and the underlying target is minimized. Here, our target is the solution $\bu \in X^{\ast}$ of the abstract inverse problem \eqref{eq:pde},\eqref{eq:dt} and we wish to find the tuning parameters $\theta$ such that the resulting neural network $\bu_{\theta}$ approximates $\bu$. 

To do so, we follow \cite{Lag1,KAR1,MM1} and define the following \emph{PDE residual}:
\begin{equation}
    \label{eq:res1}
    \res = \res_{\theta}:= \df\left(\bu_{\theta}\right) - \f. 
\end{equation}
By assumptions (H1),(H2) \eqref{eq:assm1}, we see that $\res \in Y^{\ast}$ and $\|\res\|_{Y^{\ast}} < +\infty$ for all $\theta \in \Theta$. Note that $\res(\bu) = \df(\bu) - \f \equiv 0$, for the solution $\bu$ of the PDE \eqref{eq:pde}. Hence, the term \emph{residual} is justified for \eqref{eq:res1}. 

As the inverse problem also involves a data term \eqref{eq:dt} on the observation domain, we need the following data residual:
\begin{equation}
    \label{eq:resd}
    \res_d = \res_{d,\theta}:= \map(\bu_{\theta}) - \g, \quad \forall x \in \dom^{\prime}.
\end{equation}
By assumptions (H3),(H4) \eqref{eq:assm2}, we see that $\res_d \in Z^{\ast}$ and $\|\res_d\|_{Z^{\ast}} < +\infty$ for all $\theta \in \Theta$. Note that $\res_d(\bu) = \map(\bu) - \g \equiv 0$ for all $x \in \dom^{\prime}$, for the solution $\bu$ of the inverse problem \eqref{eq:pde},\eqref{eq:dt}. 

The strategy of PINNs, following \cite{KAR2}, is to minimize the \emph{residuals} \eqref{eq:res1},\eqref{eq:resd}, simultaneously over the admissible set of tuning parameters $\theta \in \Theta$ i.e 
\begin{equation}
    \label{eq:pinn1}
    {\rm Find}~\theta^{\ast} \in \Theta:\quad \theta^{\ast} = {\rm arg}\min\limits_{\theta \in \Theta} \left(\|\res_{\theta}\|_{Y} + \|\res_{d,\theta}\|_{Z}\right).
\end{equation}
Realizing that $Y = L^{p_y}(\dom), Z=L^{p_z}(\dom^{\prime})$ for some $1 \leq p_x,p_y < \infty$, we can equivalently minimize,
\begin{equation}
    \label{eq:pinn2}
    {\rm Find}~\theta^{\ast} \in \Theta:\quad \theta^{\ast} = {\rm arg}\min\limits_{\theta \in \Theta} \left(\int\limits_{\dom} |\res_{\theta}(y)|^{p_y} dy +\int\limits_{\dom^{\prime}} |\res_{d,\theta}(z)|^{p_z} dz \right)  . 
\end{equation}
As it will not be possible to evaluate the integral in \eqref{eq:pinn2} exactly, we need to approximate it numerically by a quadrature rule. To this end, we use the quadrature rules \eqref{eq:quad},\eqref{eq:dqd} discussed earlier. We select the \emph{training sets} $\train_{int} = \{y_i\}$ with $y_i \in \dom$ for all $1 \leq i \leq N_{int}$  and $\train_d = \{z_j\}$ with $z_j \in \dom^{\prime}$ and $1 \leq j \leq N_d$ as the quadrature points for the quadrature rules \eqref{eq:quad}, \eqref{eq:dqd}, respectively.  We consider the following \emph{loss function}:
\begin{equation}
    \label{eq:lf1}
    J(\theta):= \sum\limits_{j=1}^{N_d} w^d_j |\res_{d,\theta}(z_j)|^{p_z}+ \lambda \sum\limits_{i=1}^{N_{int}} w_i |\res_{\theta}(y_i)|^{p_y}, 
\end{equation}
with the residuals $\res_d$, $\res$ defined in \eqref{eq:resd} and \eqref{eq:res1}. Here, $\lambda > 0$ is an additional hyperparameter that balances the roles of the data and PDE residuals.

It is common in machine learning \cite{DLbook} to regularize the minimization problem for the loss function i.e we seek to find,
\begin{equation}
\label{eq:lf2}
\theta^{\ast} = {\rm arg}\min\limits_{\theta \in \Theta} \left(J(\theta) + \lambda_{reg} J_{reg}(\theta) \right).
\end{equation}  
Here, $J_{reg}:\Theta \to \R$ is a \emph{weight regularization} (penalization) term. A popular choice is to set  $J_{reg}(\theta) = \|\theta_W\|^q_q$ for either $q=1$ (to induce sparsity) or $q=2$. The parameter $0 \leq \lambda_{reg} \ll 1$ balances the regularization term with the actual loss $J$ \eqref{eq:lf1}. 

The above minimization problem amounts to finding a minimum of a possibly non-convex function over a subset of $\R^M$ for possibly very large $M$. We will follow standard practice in machine learning and solving this minimization problem approximately by either (first-order) stochastic gradient descent methods such as ADAM \cite{adam} or even higher-order optimization methods such as different variants of the LBFGS algorithm \cite{lbfgs}. 

For notational simplicity, we denote the (approximate, local) minimum in \eqref{eq:lf2} as $\theta^{\ast}$ and the underlying deep neural network $\bu^{\ast}= \bu_{\theta^{\ast}}$ will be our physics-informed neural network (PINN) approximation for the solution $\bu$ of the PDE \eqref{eq:pde}.  

The proposed algorithm for computing this PINN is summarized below,
\begin{algorithm} 
\label{alg:PINN} {\bf Finding a physics informed neural network to approximate the solution of the inverse problem \eqref{eq:pde},\eqref{eq:dt}}. 
\begin{itemize}
\item [{\bf Inputs}:] Underlying domains $\dom^{\prime} \subset \dom$, differential operator $\df$ and input source term $\f$ for the PDE \eqref{eq:pde}, restriction operator $\map$ and data term $\g$ for the data \eqref{eq:dt}, quadrature points and weights for the quadrature rules \eqref{eq:quad} \eqref{eq:dqd}, non-convex gradient based optimization algorithms.
\item [{\bf Goal}:] Find PINN $\bu^{\ast}= \bu_{\theta^{\ast}}$ for approximating the inverse problem for the  \eqref{eq:pde} with data \eqref{eq:dt}. 
\item [{\bf Step $1$}:] Choose the training set $\train_{int} = \{y_i\}$ for $y_i \in \dom$, for all $1 \leq i \leq N_{int}$  such that $\{y_i\}$ are quadrature points for the underlying quadrature rule \eqref{eq:quad}. Choose the training set $\train_d = \{z_j\}$ for $z_j \in \dom^{\prime}$, for all $1 \leq j \leq N_d$  such that $\{z_j\}$ are quadrature points for the underlying quadrature rule \eqref{eq:dqd}.
\item [{\bf Step $2$}:] For an initial value of the weight vector $\overline{\theta} \in \Theta$, evaluate the neural network $\bu_{\overline{\theta}}$ \eqref{eq:ann1}, the PDE residual \eqref{eq:res1}, data residual \eqref{eq:resd}, the loss function \eqref{eq:lf2} and its gradients to initialize the underlying optimization
algorithm.
\item [{\bf Step $3$}:] Run the optimization algorithm till an approximate local minimum $\theta^{\ast}$ of \eqref{eq:lf2} is reached. The map $\bu^{\ast} = \bu_{\theta^{\ast}}$ is the desired PINN for approximating the solution $\bu$ of the inverse problem \eqref{eq:pde},\eqref{eq:dt}. 
\end{itemize}
\end{algorithm}
\subsection{An abstract estimate on the generalization error}
In this section, we will estimate the error due to the PINN (generated by algorithm \ref{alg:PINN}) in approximating the solution $\bu$ of the inverse problem for PDE \eqref{eq:pde} with data \eqref{eq:dt}. 

We focus on the so-called \emph{generalization error} of the PINN, which in machine learning terminology \cite{MLbook}, is often understood as the error of the neural network on \emph{unseen data}. Following the recent paper \cite{MM1}, we set $\dom^{\prime} \subset E \subset \dom$ and define the corresponding generalization error as,
 \begin{equation}
    \label{eq:egen}
    \er_G(E) = \er_{G} (E;\theta^{\ast},\train_{int},\train_d) := \|\bu-\bu^{\ast}\|_{L^{p_x}(E)}.
\end{equation}
As written out above, the generalization error clearly depends on the training sets $\train_{int},\train_d$, as well as on the parameters $\theta^{\ast}$ of the PINN, generated by algorithm \ref{alg:PINN}. However, we shall suppress this explicit dependence for notational convenience. 

As in the recent article \cite{MM1}, we estimate the generalization error \eqref{eq:egen} in terms of the so-called \emph{training error} $\er_T = \er_T(\theta^{\ast},\train_{int},\train_d)$ defined by,
\begin{equation}
    \label{eq:etrn}
    \er_T = \Big(\er_{d,T}^{p_z} + \lambda\er_{p,T}^{p_y}\Big)^{\frac{2}{p_y + p_z}},\quad \er_{d,T}:=\left(\sum\limits_{j=1}^{N_d} w^d_j |\res_{d,\theta}(z_j)|^{p_z}\right)^{\frac{1}{p_z}},\quad  \er_{p,T}:= \left(\sum\limits_{i=1}^{N_{int}} w_i |\res_{\theta}(y_i)|^{p_y}\right)^{\frac{1}{p_y}}. 
\end{equation}
Note that, after the training has concluded, the training error $\er_T$ can be readily computed from the loss functions \eqref{eq:lf1} or \eqref{eq:lf2}. The bound on generalization error in terms of training error is given by the following estimate,
\begin{theorem}
\label{thm:1}
Let $\bu \in \hat{X} \subset X^{\ast} \subset X$ be the solution of the inverse problem associated with PDE 
\eqref{eq:pde} and data \eqref{eq:dt}. Assume that the stability hypothesis \eqref{eq:assm} holds for any $\dom^{\prime} \subset E \subset \dom$. Let $\bu^{\ast} \in \hat{X} \subset X^{\ast}$ be a PINN generated by the algorithm \ref{alg:PINN}, based on the training sets $\train_{int}$ (of quadrature points corresponding to the quadrature rule \eqref{eq:quad}) and $\train_d$ (of quadrature points corresponding to the quadrature rule \eqref{eq:dqd}). Further assume that the residuals $\res_{\theta^{\ast}}$ \eqref{eq:res1},  and $\res_{d,\theta^{\ast}}$ \eqref{eq:resd}, be such that $|\res_{\theta^{\ast}}|^{p_y} \in \underline{Y},|\res_{d,\theta^{\ast}}|^{p_z} \in \underline{Z}$ and the quadrature errors satisfy \eqref{eq:qassm}, \eqref{eq:dqassm}. Then the following estimate on the generalization error \eqref{eq:egen} holds,
\begin{equation}
    \label{eq:egenb}
    \er_G \leq C_{pd} \left(\er_{p,T}^{\tau_p} + \er_{d,T}^{\tau_d}+ C_{q}^{\frac{\tau_p}{p_y}}N_{int}^{-\frac{\alpha \tau_p}{p_y}} + C_{qd}^{\frac{\tau_d}{p_z}}N_d^{-\frac{\alpha_d \tau_d}{p_z}} \right),
\end{equation}
with constants $C_{pd} = C_{pd}\left(\|\bu\|_{\hat{X}},\|\bu^{\ast}\|_{\hat{X}}\right)$, $C_{q} = C_{q}\left(\left\||\res_{\theta^{\ast}}|^{p_y}\right\|_{\underline{Y}}\right)$, and $C_{qd} = C_{qd}\left(\left\||\res_{d,\theta^{\ast}}|^{p_z}\right\|_{\underline{Z}}\right)$ stem from \eqref{eq:assm}, \eqref{eq:qassm} and \eqref{eq:dqassm}, respectively. 
\end{theorem}

\begin{proof}
For notational simplicity, we denote $\res = \res_{\theta^{\ast}}$, the residual \eqref{eq:res1}, corresponding to the trained neural network $\bu^{\ast}$. Similarly, $\res_d = \res_{d,\theta^{\ast}}$ is the data residual \eqref{eq:resd}, corresponding to $\bu^{\ast}$.

As $\bu$ solves the PDE \eqref{eq:pde} and $\res$ is defined by \eqref{eq:res1}, we easily see that, 
\begin{equation}
\label{eq:pf1}
\res = \df(\bu^{\ast}) -\df(\bu).
\end{equation}
Similarly, $\bu$ satisfies the data relation \eqref{eq:dt} and by definition \eqref{eq:resd}, we have,
\begin{equation}
\label{eq:pf2}
\res_d = \map(\bu^{\ast}) -\map(\bu).
\end{equation}
By our assumptions, PINN $\bu^{\ast} \in \hat{X}$, hence, we can directly apply the \emph{conditional stability estimate} \eqref{eq:assm} and use \eqref{eq:pf1}, \eqref{eq:pf2} to obtain,
\begin{equation}
    \label{eq:pf3}
\begin{aligned}
\er_G(E) &=    \|\bu - \bu^{\ast}\|_{L^{p_x}(E)} \quad ({\rm by})~\eqref{eq:egen}, \\
&\leq C_{pd} \left(\|\df(\bu^{\ast}) - \df(\bu) \|^{\tau_p}_{Y} + \|\map(\bu) - \map(\bu^{\ast})\|^{\tau_d}_Z \right)  \quad ({\rm by})~\eqref{eq:assm}, \\
&\leq C_{pd} \left(\|\res\|^{\tau_p}_{Y} + \|\res_d\|^{\tau_d}_{Z}\right) \quad ({\rm by})~\eqref{eq:pf1},\eqref{eq:pf2}.
\end{aligned}
\end{equation}
By the fact that $Y = L^{p_y}(\dom)$, the definition of the training error \eqref{eq:etrn} and the quadrature rule \eqref{eq:quad}, we see that,
\begin{align*}
\|\res\|^{p_y}_{Y} \approx \sum\limits_{i=1}^{N_{int}} w_i |\res_{\theta^{\ast}}|^{p_y}  = \er_{p,T}^{p_y}.
\end{align*}
Hence, the training error component $\er_{p,T}$ is a quadrature for the residual \eqref{eq:res1} and the resulting quadrature error, given by \eqref{eq:qassm} translates to,
\begin{align*}
   \|\res\|^{p_y}_{Y} \leq \er_{p,T}^{p_y} + C_{q}N_{int}^{-\alpha},
\end{align*}
and as $\tau_p \leq 1$
\begin{equation}
    \label{eq:pf4}
   \|\res\|^{\tau_p}_{Y} \leq \er_{p,T}^{\tau_p} + C_{q}^{\frac{\tau_p}{p_y}}N_{int}^{-\frac{\alpha\tau_p}{p_y}}.
\end{equation}
Similarly as $Z = L^{p_z}(\dom^{\prime})$, the definition of the training error \eqref{eq:etrn} and the quadrature rule \eqref{eq:dqd}, we have that,
\begin{align*}
\|\res_d\|^{p_z}_{Z} \approx \sum\limits_{j=1}^{N_d} w^d_j |\res_{d,\theta^{\ast}}|^{p_z}  = \er_{d,T}^{p_z}.
\end{align*}
Hence, the training error component $\er_{d,T}$ is a quadrature for the residual \eqref{eq:resd} and the resulting quadrature error, given by \eqref{eq:dqassm} leads to,
\begin{align*}
   \|\res_d\|^{p_z}_{Z} \leq \er_{d,T}^{p_z} + C_{qd}N_d^{-\alpha_d},
\end{align*}
and as $\tau_d \leq 1$
\begin{equation}
\label{eq:pf5}
   \|\res_d\|^{\tau_d}_{Z} \leq \er_{d,T}^{\tau_d} + C_{qd}^{\frac{\tau_d}{p_z}}N_d^{-\frac{\alpha_d\tau_d}{p_z}}.
\end{equation}
Substituting \eqref{eq:pf4} and \eqref{eq:pf5} into \eqref{eq:pf3} yields the desired bound \eqref{eq:egenb}. 
\end{proof}
We term a PINN, generated by the algorithm \ref{alg:PINN} to be \emph{well-trained} if the following condition hold,
\begin{equation}
    \label{eq:wtn}
    \max\left\{\er_{p,T}^{\tau_p},\er_{d,T}^{\tau_d}\right\} \leq C_{q}^{\frac{\tau_p}{p_y}}N_{int}^{-\frac{\alpha \tau_p}{p_y}} + C_{qd}^{\frac{\tau_d}{p_z}}N_d^{-\frac{\alpha_d \tau_d}{p_z}}.
\end{equation}
Thus, a well-trained PINN is one for which the training errors are smaller than the so-called \emph{generalization gap} (given by the rhs of \eqref{eq:wtn}).

The following remarks about Theorem \ref{thm:1} are in order.
\begin{remark}
The estimate \eqref{eq:egenb} indicates mechanisms that underpin possible efficient approximation of solutions of inverse (unique continuation, data assimilation) problems by PINNs as it breaks down the sources of error into the following parts,
\begin{itemize}
    \item The PINN has to be well-trained i.e, the training error $\er_T$ has to be sufficiently small. Note that we have no a priori control on the training error but can compute it \emph{a posteriori}. 
    \item The class of approximating PINNs has to be sufficiently regular such that the residuals in \eqref{eq:resd}, \eqref{eq:res1} can be approximated to high accuracy by the quadrature rules \eqref{eq:quad},\eqref{eq:dqd}. This regularity of PINNs can be enforced by choosing smooth activation functions such as the sigmoid and hyperbolic tangent in \eqref{eq:ann1}.
    \item Finally, the whole estimate \eqref{eq:egenb} relies on the conditional stability estimate \eqref{eq:assm} for the inverse problem \eqref{eq:pde}, \eqref{eq:dt}. Thus, the generalization error estimate leverages conditional stability for inverse problems of PDEs into efficient approximation by PINNs. This is very similar to the program in a recent paper \cite{MM1} for forward problems of PDEs. 
\end{itemize}
\end{remark}
\begin{remark}
We note that the estimate \eqref{eq:egenb} on the error due to PINNs contains the constants $C_{pd}, C_{q},C_{qd}$. These constants depend on the underlying solution but also on the trained neural network $\bu^{\ast}$ and on the residuals. For any given number of training samples $N_{int},N_d$, these constants are bounded as the underlying neural networks and residuals are smooth functions on bounded domains. However, as $N_{int}, N_d \rightarrow \infty$, these constants might blow-up. As long as these constants blow up at rates that are slower wrt $N_{int}, N_d$ than the decay terms in \eqref{eq:egenb}, one can expect that the overall bound \eqref{eq:egenb} still decays to zero as $N_{int},N_d \rightarrow \infty$. In practice, one has a finite number of training samples and the bounds on the constants can be verified a posteriori from the computed residuals and trained neural networks. 
\end{remark}
\begin{remark}
The estimate \eqref{eq:egenb} was based on determinitic quadrature rules \eqref{eq:quad}, \eqref{eq:dqd}. Following section 2.4.1 of recent paper \cite{MM1}, it can be extended, in a straightforward manner, to the case of randomly chosen training points that stem from Monte Carlo quadrature. 
\end{remark}
\begin{remark}
\label{rmk:nse}
We have considered a noiseless measurement in the data term \eqref{eq:dt} on the observation domain. However, the bound \eqref{eq:egenb} can be readily extended to cover the noisy case in the following manner. We assume that the measurements for the inverse problem are given by,
\begin{equation}
    \label{eq:dtn}
    \map(\bu) = \g + \eta,
\end{equation}
for a noise term $\eta \in Z$. Then, a straightforward modification of the argument in the proof of Theorem \ref{thm:1} yields the bound,
\begin{equation}
    \label{eq:egenbn}
    \er_G \leq C_{pd} \left(\er_{p,T}^{\tau_p} + \er_{d,T}^{\tau_d}+ \|\eta\|^{\tau_d}_{L^{p_z}(\dom^{\prime})} +  C_{q}^{\frac{\tau_p}{p_y}}N_{int}^{-\frac{\alpha \tau_p}{p_y}} + C_{qd}^{\frac{\tau_d}{p_z}}N_d^{-\frac{\alpha_d \tau_d}{p_z}} \right),
\end{equation}
with constants already defined in \eqref{eq:egenb}. Thus, as long as the noise term is small in magnitude, the PINN will still efficiently approximate solution of the inverse problem. However, for $N_{int},N_{d}$ sufficiently small, the noise term will dominate \eqref{eq:egenbn} and one has to use a Bayesian framework to approximate solutions of the inverse problem in this regime.  
\end{remark}
\section{Poisson's equation}
\label{sec:3}
As a first example for the abstract inverse problem \eqref{eq:pde}, \eqref{eq:dt}, we consider the Poisson's equation as a model problem for linear elliptic PDEs.  
\subsection{The underlying inverse problem}
Let $D \subset \R^d$ be an open, bounded, simply connected set with smooth boundary $\bD$\footnote{Further geometric conditions on the boundary might be necessary for obtaining bounds on the quadrature error}. We consider the Poisson's equation on this domain,
\begin{equation}
    \label{eq:ps}
    -\Delta u = f, \quad \forall x \in D,
\end{equation}
with $\Delta$ denoting the Laplace operator and $f \in L^2(D)$ being a source term. We will assume that $u \in H^1(D)$ will satisfy the Poisson's equation \eqref{eq:ps} in the following weak sense, 
\begin{equation}
    \label{eq:wkps}
    \int\limits_{D} \nabla u \cdot \nabla v dx = \int\limits_{D} f v dx,
\end{equation}
for all test functions $v \in H^1_0(D)$. Note that \eqref{eq:wkps} follows as a consequence of multiplying the test function $v \in H^1_0(D)$ and integrating by parts. Moreover, the PDE \eqref{eq:ps} is not well-posed as $u$ is not necessarily in $H^1_0(D)$.

The unique continuation (data assimilation) inverse problem in this case is given by,
\begin{equation}
    \label{eq:dtps}
    u|_{D^{\prime}} = g, \quad {\rm for~some~} D^{\prime} \subset D, 
\end{equation}
with $g \in H^1(D^{\prime})$ and the observation domain $D^{\prime}$ being open, simply connected and with a smooth boundary $\bD^{\prime}$.

Formally, for the unique continuation problem \eqref{eq:ps}, \eqref{eq:dtps}, to have a solution, it is necessary that $g$ satisfies the Poisson's equation \eqref{eq:ps} in $D^{\prime}$. Hence, the unique continuation problem is formally equivalent to,
\begin{equation}
    \label{eq:ecp}
     \begin{aligned}
     -\Delta u &=f, \quad \forall x \in D \setminus D^{\prime}, \\
     u &= g, \quad \forall x \in \bD^{\prime}, \\
     \nabla u\cdot \nl &= \nabla g \cdot \nl, \quad \forall x \in \bD^{\prime}.
     \end{aligned}    
\end{equation}
The problem \eqref{eq:ecp} is termed as the \emph{Elliptic Cauchy problem} and was already studied by Hadamard as an example of ill-posed problems for PDEs.

Well-posedness results for the inverse problem \eqref{eq:ps}, \eqref{eq:dtps} are classical, see \cite{ALE1} for a detailed survey. Here, we follow the slightly simplified presentation due to \cite{BO1} and state the following result, 
\begin{theorem}
\label{thm:p1}
[\cite{BO1}, Theorem 7.1]: Let $f \in L^2(D)$ and let $u \in H^1(D)$ such that \eqref{eq:wkps} holds for all test functions $v \in H^1_0(D)$. Let $g \in H^1(D^{\prime})$ such that \eqref{eq:dtps} holds, then for every open simply connected set $E \subset D$ such that ${\rm dist}(E, \bD) > 0$, there holds,
\begin{equation}
    \label{eq:psst1}
    \|u\|_{H^1(E)} \leq C\left(\|u\|_{H^1(D)}\right) \omega \left(\|f\|_{L^2(D)} + \|g \|_{L^2(D^{\prime})}\right).
\end{equation}
Here, $C(R) = C R^{1-\tau}$ and $\omega(R) = R^{\tau}$, for some absolute constant $C$ and exponent $\tau \in (0,1)$, depending on the set $E$.

Moreover, we have the global stability estimate,
\begin{equation}
    \label{eq:psst2}
    \|u\|_{H^1(D)} \leq C\left(\|u\|_{H^1(D)}\right) \omega \left(\|f\|_{L^2(D)} + \|g \|_{L^2(D^{\prime})}\right),
\end{equation}
with the same $C(R)$ as in \eqref{eq:psst1}, but with modulus of continuity given by $\omega(R) = |\log(R)|^{-\tau}$, with $\tau \in (0,1).$
\end{theorem}
The theorem, as presented in \cite{BO1} (Theorem 7.1) has slightly better estimates than \eqref{eq:psst1}, \eqref{eq:psst2}. In particular, on the rhs, the norm on the source term is the $H^{-1}$ norm. However, the current version of these estimates suffices for our purposes here. Moreover, detailed derivation of the constants is given in Theorem 4.4 of \cite{ALE1} and the proof of \eqref{eq:psst1}, \eqref{eq:psst2} is based on the \emph{three-balls} inequality. 
\begin{remark}
Relating the above formulation of the unique continuation inverse problem \eqref{eq:ps}, \eqref{eq:dtps} to the abstract formalism presented in section \ref{sec:2}, we see that with $X = L^2(D), Y = L^2(D), Z = L^2(D^{\prime}$, the differential operator $\df = -\Delta$, interpreted weakly, and the observable $\map = ID$, it is straightforward to bound the abstract conditional stability estimate \eqref{eq:assm} in the concrete forms \eqref{eq:psst1} or \eqref{eq:psst2}. Thus, this problem falls squarely in the framework considered in section \ref{sec:2} and can be approximated by PINNs. 
\end{remark}
\subsection{PINNs}
In order to complete the algorithm \ref{alg:PINN} to generate a PINN for approximating the inverse problem \eqref{eq:ps}, \eqref{eq:dtps}, we need the following ingredients,
\subsubsection{Training sets}
As the training set $\train_{int}$ in algorithm \ref{alg:PINN}, we take a set of quadrature points $y_i \in D$, for $1 \leq i \leq N_{int}$, corresponding to the quadrature rule \eqref{eq:quad}. These can be quadrature points for a grid based (composite) Gauss quadrature rule or low-discrepancy sequences such as Sobol points. Similarly, the training set $\train_{d} = \{z_j \}$ for $z_j \in D^{\prime}$, with $1 \leq j \leq N_d$, are quadrature points, corresponding to the quadrature rule \eqref{eq:dqd}. 
\subsubsection{Residuals}
We will require that for parameters $\theta \in \Theta$, the neural networks $u_{\theta} \in C^k(D)$, for $k \geq 2$. This can be enforced by choosing a sufficiently regular activation function in \eqref{eq:ann1}. We define the following residuals that are needed in algorithm \ref{alg:PINN}. The PDE residual \eqref{eq:res1} is given by,
\begin{equation}
    \label{eq:resps}
    \res_{\theta} = -\Delta u_{\theta} - f, \quad \forall x \in D,
\end{equation}
and the data residual \eqref{eq:resd} on the observation domain is given by,
\begin{equation}
    \label{eq:resdps}
    \res_{d,\theta} = u_{\theta} - g, \quad \forall x \in D^{\prime}.
\end{equation}
\subsubsection{Loss functions}
In algorithm \ref{alg:PINN} for approximating the inverse problem \eqref{eq:ps}, \eqref{eq:dtps}, we will need the following loss function,
\begin{equation}
    \label{eq:lfps}
    J(\theta) = \sum\limits_{j=1}^{N_d} w^d_j|\res_{d,\theta}(z_j)|^2 + \lambda \sum\limits_{i=1}^{N_{int}} w_i|\res_{\theta}(y_i)|^2,
\end{equation}
with hyperparamter $\lambda$, residuals defined in \eqref{eq:resps}, \eqref{eq:resdps}, training points defined above and weights $w_i$, $w^d_j$, corresponding to quadrature rules \eqref{eq:quad} and \eqref{eq:dqd}, respectively. 
\subsection{Estimates on the generalization error}
We consider any $E \subset D$, with $E$ open, simply connected and such that $dist(E,\bD) > 0$ and define the generalization error with respect to the PINN $u^{\ast} = u_{\theta^{\ast}}$, generated by the algorithm \ref{alg:PINN} with training sets, residuals and loss functions described above, as 
\begin{equation}
    \label{eq:gerps}
    \er_G(E) = \|u - u^{\ast}\|_{H^1(E)}.
\end{equation}
As in theorem \ref{thm:1}, this generalization error will be estimated in terms of the following training errors, 
\begin{equation}
    \label{eq:etrnps}
    \er_{d,T} = \left(\sum\limits_{j=1}^{N_d} w^d_j|\res_{d,\theta^{\ast}}(z_j)|^2\right)^{\frac{1}{2}}, \quad \er_{p,T} = \left(\sum\limits_{i=1}^{N_{int}} w_i|\res_{\theta^{\ast}}(y_i)|^2\right)^{\frac{1}{2}}.
\end{equation}
Note that the training errors $\er_{p,T}$ and $\er_{d,T}$, can be readily computed from the loss functions \eqref{eq:lf2}, \eqref{eq:lfps}, a posteriori. We have the following estimate on the generalization error in terms of the training error,
\begin{lemma}
\label{lem:p1} 
For $f \in C^{k-2}(D)$ and $g \in C^k(D^{\prime})$, with continuous extensions of the functions and derivatives upto the boundaries of the underlying sets and with $k \geq 2$, Let $u \in H^1(D)$ be the solution of the inverse problem corresponding to the Poisson's equation \eqref{eq:ps} i.e, it satisfies \eqref{eq:wkps} for all test functions $v \in H^1_0(D)$ and satisfies the data \eqref{eq:dtps}. Let $u^{\ast} = u_{\theta^{\ast}} \in C^k(D)$ be a PINN generated by the algorithm \ref{alg:PINN}, with loss functions \eqref{eq:lf2}, \eqref{eq:lfps}. Then, the generalization error \eqref{eq:gerps} for any any $E \subset D$, with $E$ open, simply connected and such that $dist(E,\bD) > 0$ is bounded by, 
\begin{equation}
    \label{eq:pbd}
    \|u-u^{\ast}\|_{H^1(E)} \leq C \left(\|u\|^{1-\tau}_{H^1(D)}+\|u^{\ast}\|^{1-\tau}_{H^1(D)}\right)\left(\er_{p,T}^{\tau} + \er_{d,T}^{\tau} + C_{q}^{\frac{\tau}{2}}N_{int}^{-\frac{\alpha\tau}{2}} + C_{qd}^{\frac{\tau}{2}}N_d^{-\frac{\alpha_d \tau}{2}}\right),
\end{equation}
for some $\tau \in (0,1)$ and constant $C$ depending on $E$ and with constants $C_q = C_q\left(\|\res\|_{C^{k-2}(D)}\right)$ and $C_{qd} = C_{qd}\left(\|\res\|_{C^{k}(D^{\prime})}\right)$, given by the quadrature error bounds \eqref{eq:qassm}, \eqref{eq:dqassm}, respectively.

Moreover, we also have the global error bound, 
\begin{equation}
    \label{eq:pbd1}
    \|u-u^{\ast}\|_{H^1(D)} \leq C \left(\|u\|^{1-\tau}_{H^1(D)}+\|u^{\ast}\|^{1-\tau}_{H^1(D)}\right)\left| \log\left(\er_{p,T} + \er_{d,T} + C_{q}^{\frac{1}{2}}N_{int}^{-\frac{\alpha}{2}} + C_{qd}^{\frac{1}{2}}N_d^{-\frac{\alpha_d}{2}}\right)\right|^{-\tau}.
\end{equation}

\end{lemma}
\begin{proof}
For notational simplicity, we denote $\res = \res_{\theta^{\ast}}$ and $\res_d = \res_{d,\theta^{\ast}}$. 

Define $\hat{u} = u^{\ast} - u \in H^1(D)$, by linearity of the differential operator in \eqref{eq:ps} and the data observable in \eqref{eq:dtps} and by definitions \eqref{eq:resps}, \eqref{eq:resdps}, we see that $\hat{u}$ satisfies,
\begin{equation}
    \label{eq:hps}
    \begin{aligned}
    -\Delta \hat{u} &= \res, \quad \forall x \in D, \\
    \hat{u} &= \res_d, \quad \forall x \in D^{\prime},
    \end{aligned}
\end{equation}
with Poisson equation being satisfied in the following weak sense,
\begin{equation}
    \label{eq:wkhps}
    \int\limits_{D} \nabla \hat{u} \cdot \nabla v dx = \int\limits_{D} \res v dx,
\end{equation}
for all test functions $v \in H^1_0(D)$. Hence, we can directly apply the conditional stability estimate \eqref{eq:psst1} to obtain for some $\tau \in (0,1)$,
\begin{equation}
    \label{eq:pl1}
    \begin{aligned}
    \|\hat{u}\|_{H^1(E)} &\leq C( \|\hat{u}\|^{1-\tau}_{H^1(D)}) \left(\|\res_d\|_{L^2(D^{\prime})} + \|\res\|_{L^2(D)} \right)^{\tau} \\
    &\leq C \left(\|u\|^{1-\tau}_{H^1(D)}+\|u^{\ast}\|^{1-\tau}_{H^1(D)}\right)\left(\|\res_d\|_{L^2(D^{\prime})} + \|\res\|_{L^2(D)} \right)^{\tau}.
    \end{aligned}
\end{equation}
Recognizing that the training errors $\er^2_{d,T}$ and $\er_{p,T}^2$ are the quadratures for $\|\res_d\|^2_{L^2(D^{\prime})}$ and $\|\res\|^2_{L^2(D)}$ with respect to the quadrature rules \eqref{eq:dqd} and \eqref{eq:quad}, respectively and using bounds \eqref{eq:dqassm} and \eqref{eq:qassm} yields,
\begin{equation}
    \label{eq:pl2}
    \begin{aligned}
    \|\res_d\|_{L^2(D^{\prime})} &\leq \er_{d,T} + C_{qd}^{\frac{1}{2}}N_d^{\frac{\alpha_d}{2}}, \\
    \|\res\|_{L^2(D)} &\leq \er_{p,T} + C_{q}^{\frac{1}{2}}N_{int}^{\frac{\alpha}{2}},
    \end{aligned}
\end{equation}
with constants $C_q = C_q\left(\|\res\|_{C^{k-2}(D)}\right)$ and $C_{qd} = C_{qd}\left(\|\res_d\|_{C^{k}(D^{\prime})}\right)$.

It is straightforward to obtain the desired inequality \eqref{eq:pbd} by substituting \eqref{eq:pl2} into \eqref{eq:pl1}.

The bound \eqref{eq:pbd1} can be obtained, completely analogously, by replacing \eqref{eq:psst1} in \eqref{eq:pl1} with \eqref{eq:psst2}. 
\end{proof}
\begin{table}[htbp] 
    \centering
    \renewcommand{\arraystretch}{1.1} 
    
    \footnotesize{
        \begin{tabular}{ c c c c c } 
            \toprule
            \bfseries $K-1$  & \bfseries $\tilde{d}$  &\bfseries $q$ & \bfseries $\lambda_{reg}$  &\bfseries $\lambda$\\ 
            \midrule
            \midrule
              4, 8, 10  & 16, 20, 24 & 2& 0, $10^{-6}$& 0.001, 0.01, 0.1, 1\\

            \bottomrule
        \end{tabular}
    \caption{Hyperparamters configurations used for ensemble training inn all the numerical experiments.}
        \label{tab:1}
    }
\end{table}
\begin{figure}[h!]
        \centering
        \includegraphics[width=8cm]{{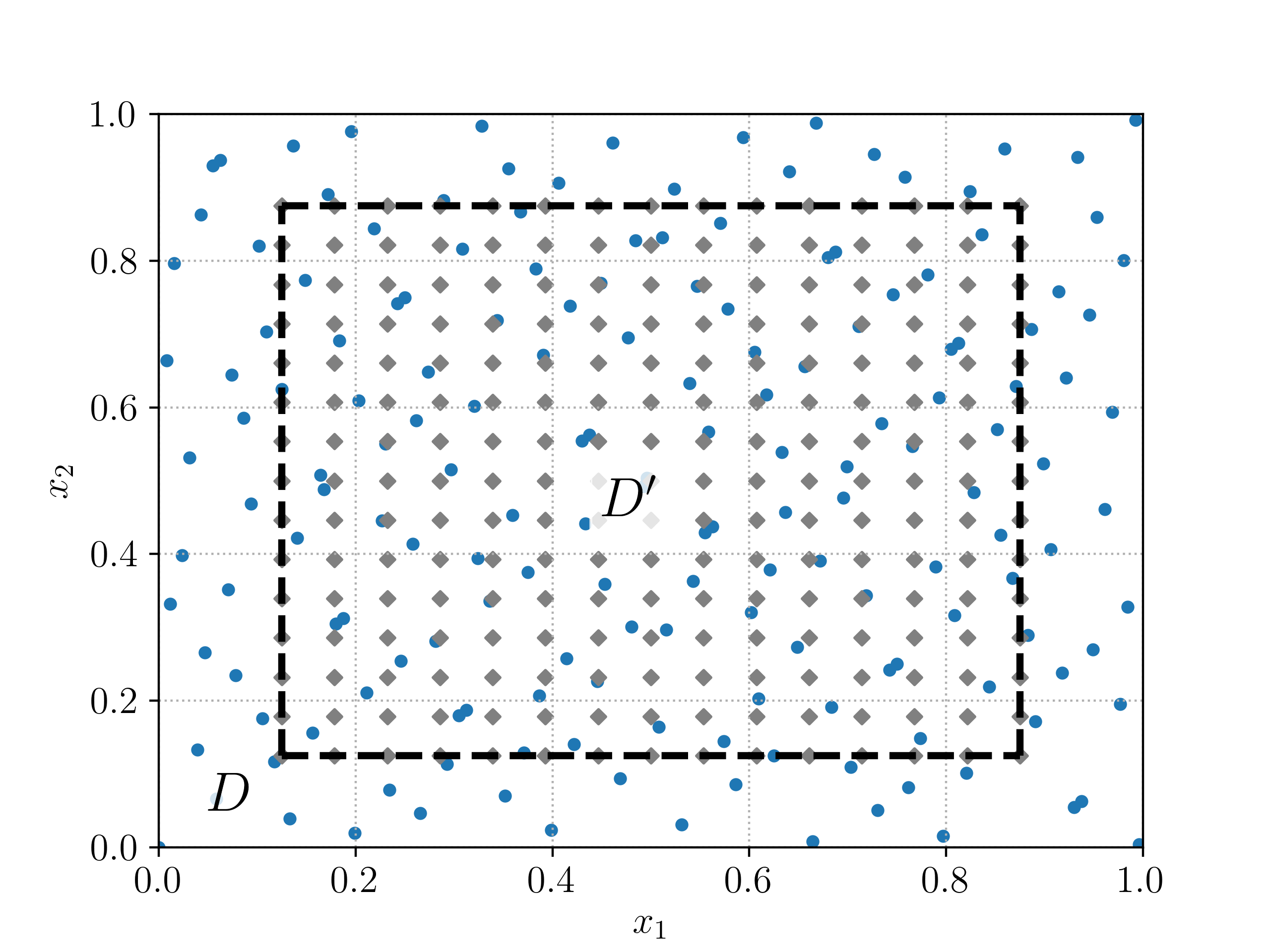}}
    \caption{The domains $D,D^{\prime}$ for the Poission's equation numerical experiment. Training set $\train_{int}$ are Sobol points (blue dots) and training set $\train_d$ are Cartesian grid points (grey squares).}
\label{fig:p1}
\end{figure}

%\begin{figure}[h!]
 %   \begin{subfigure}{.48\textwidth}
  %      \centering
   %     \includegraphics[width=1\linewidth]{{Images/points_poiss.png}}
    %    \caption{ }
     %   \label{fig:dom_poiss}
%    \end{subfigure}
 %    \begin{subfigure}{.48\textwidth}
  %      \centering
   %     \includegraphics[width=1\linewidth]{{Images/Poiss_err.png}}
    %    \caption{}
     %   \label{fig:err_poiss}
%    \end{subfigure}
 %   \caption{Solution domain \ref{fig:dom_poiss} and $L^2$ relative error norm \ref{fig:err_poiss} for the heat equation ($16\times 50$ samples)}
%\label{fig:p2_b}
%\end{figure}
\subsection{Numerical experiments}
We follow \cite{BO1} and repeat their experiment from section 7.4.3. As in \cite{BO1}, the unique continuation problem for the Poisson equation is solved by considering the PDE \eqref{eq:ps} in the domain $D = (0,1)^2$ (see figure \ref{fig:p1}), with source term,
\begin{equation}
    \label{eq:pf}
 f(x_1,x_2) = -60\Big(x_1-x_1^2 + x_2-x_2^2\Big).
\end{equation}
The data \eqref{eq:dtps} will be given in the observation domain (see figure \ref{fig:p1}),
\begin{equation}
    \label{eq:pd}
    D^{\prime} =  \{(x_1,x_2)\in \R^2 : |x_1-0.5|<0.375; ~ |x_2-0.5|<0.375\}.
\end{equation}
In order to generate the data term $g$ in \eqref{eq:dtps}, we find that,
\begin{equation}
    \label{eq:pex}
    u(x_1,x_2) = 30x_1x_2(1-x_1)(1-x_2),
\end{equation}
is an exact solution of the Poisson equation \eqref{eq:ps}, with the source term \eqref{eq:pf} and let $g = u|_{D^{\prime}}$, with $D^{\prime}$ \eqref{eq:pd}, as the data term in \eqref{eq:dtps}. Clearly the exact solution $u$ \eqref{eq:pex} and inputs $f,g$ are as regular as required in Lemma \ref{lem:p1}. 
\begin{figure}[h!]

    \begin{subfigure}{.33\textwidth}
        \centering
        \includegraphics[width=1\linewidth]{{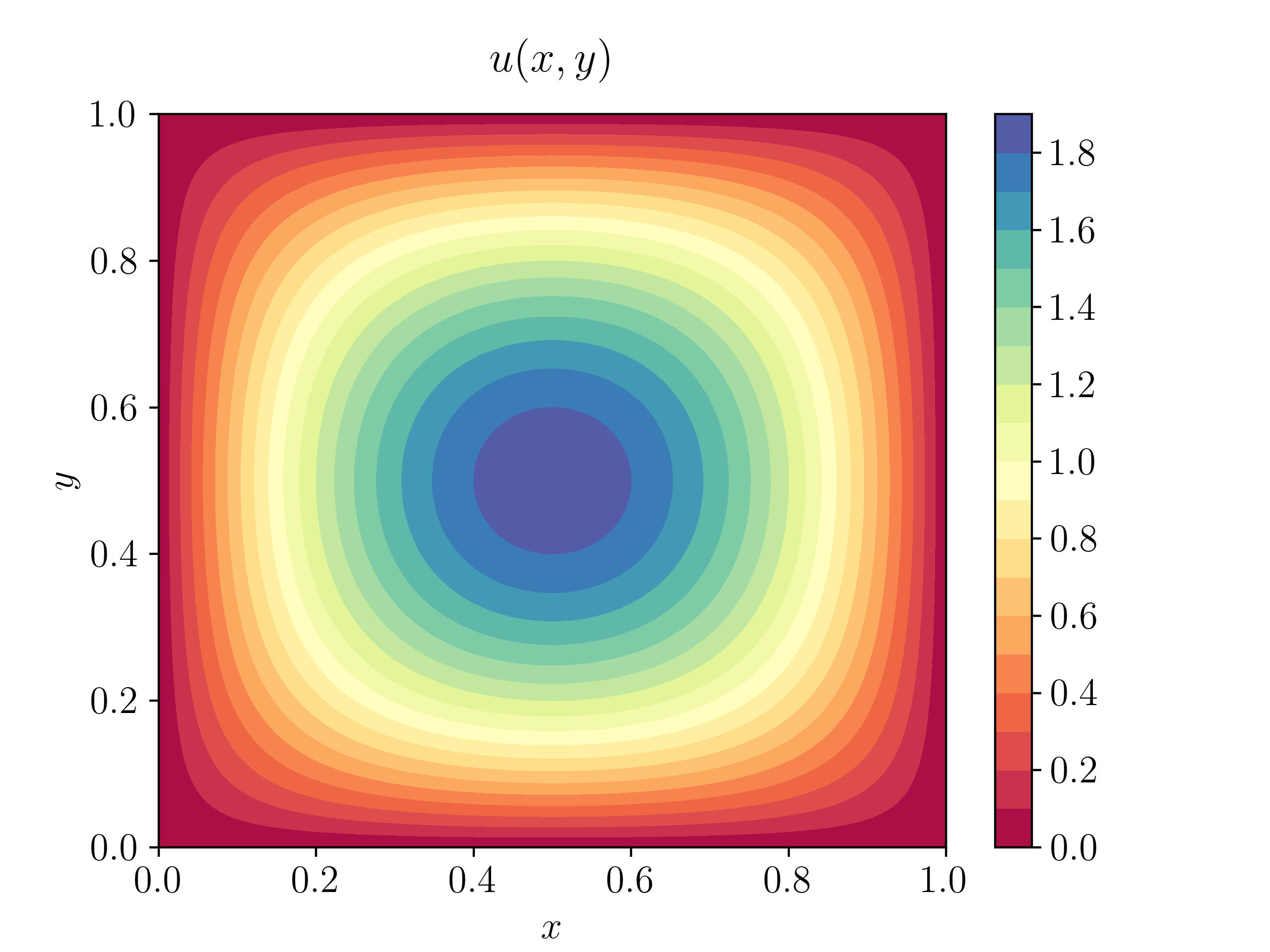}}
        \caption{Exact Solution $u$}
    \end{subfigure}
    \begin{subfigure}{.33\textwidth}
        \centering
        \includegraphics[width=1\linewidth]{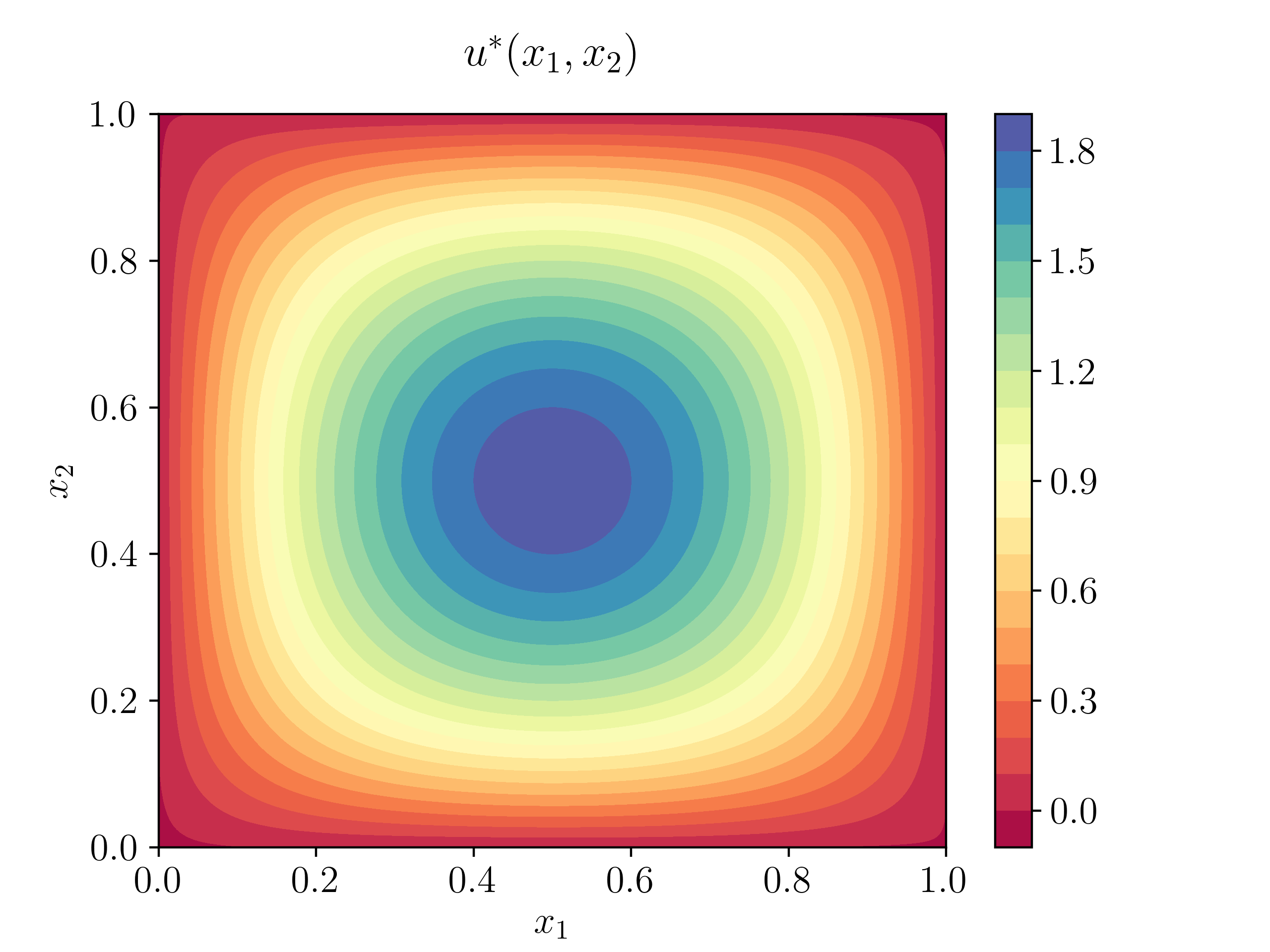}
        \caption{PINN $u^{\ast}$ }
    \end{subfigure}
     \begin{subfigure}{.33\textwidth}
        \centering
        \includegraphics[width=1\linewidth]{{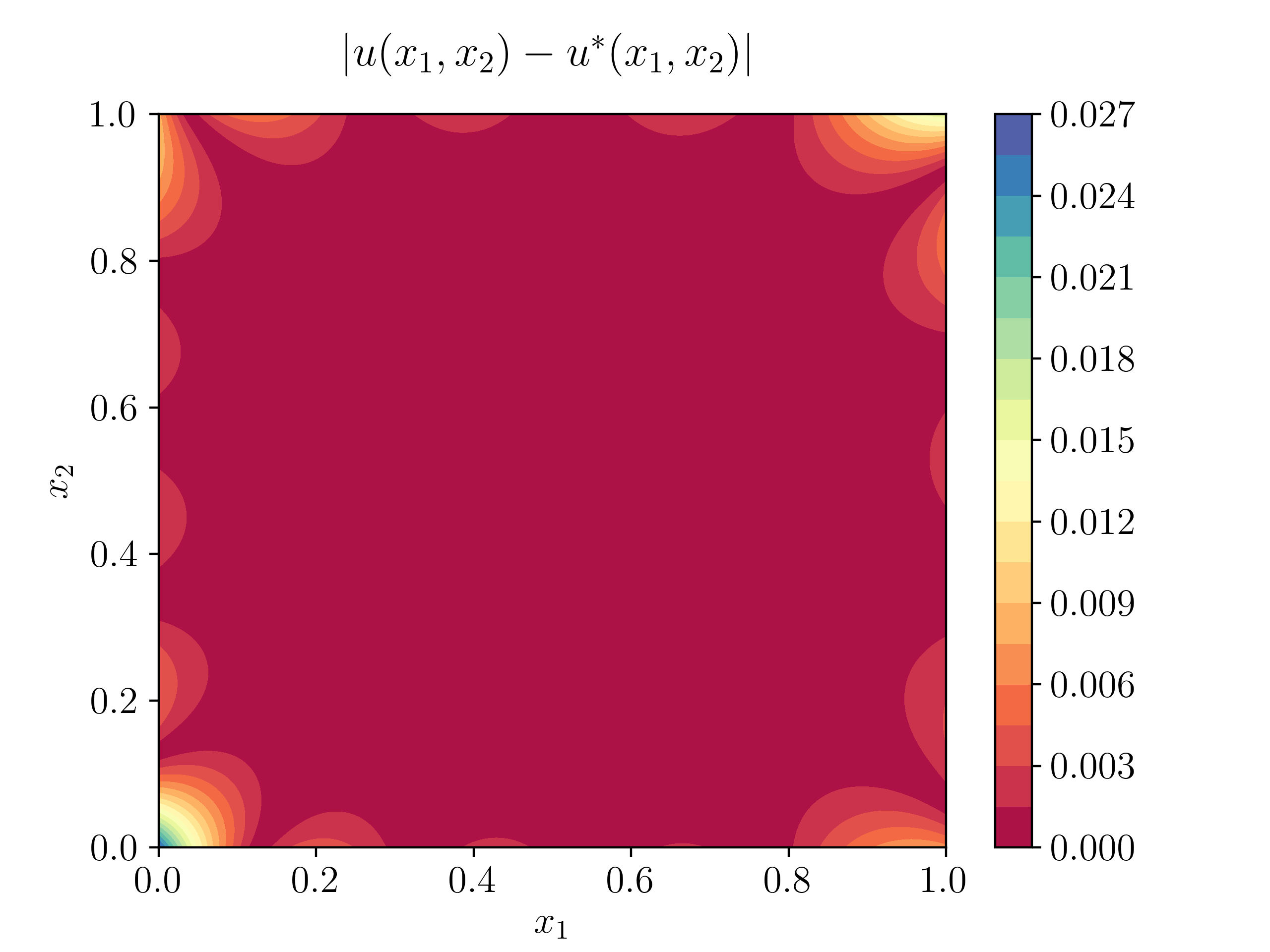}}
        \caption{Relative error $\frac{|\hat{u}|}{\|u\|_{L^2}}$}
    \end{subfigure}
  \caption{Comparison of the Exact solution, the PINN approximation and norm of the resulting error for the unique continuation problem for the Poisson's equation, with $N= 20 \times 20$ training points.}
\label{fig:p2}
\end{figure}
\begin{table}[htbp] 
    \centering
    \renewcommand{\arraystretch}{1.1} 
    
    \footnotesize{
        \begin{tabular}{c  c c c c  c c c c c  } 
            \toprule
             $N$  &\bfseries $K-1$ & \bfseries $\tilde{d}$  &$\lambda_{reg}$&\bfseries $\lambda$&  $\er_T$ & $||u - u^\ast||_{L^2}$   & $||u - u^\ast||_{H^1}$  \\ 
            \midrule
            \midrule
            $20\times20$     &4 &24&0.0& 0.001   & 0.0008& 0.28 $\%$ &1.1 $\%$ \\
        
            \midrule 
            $40\times40$  &4 &24&0.0& 0.001              &0.0006  & 0.25 $\%$ &1.0 $\%$ \\
                \midrule 
            $80\times80 $    &4 &24&0.0& 0.001   &0.00053 & 0.24 $\%$ & 0.9 $\%$ \\
                \midrule 
            $160\times160$    &4 &24&0.0& 0.001   &0.00043 & 0.2 $\%$ &0.8 $\%$ \\
            \bottomrule
        \end{tabular}
        \caption{Poisson's equation: relative percentage generalization errors and training errors for different numbers of training points. }
        \label{tab:p1}
    }
\end{table}

We will generate PINNs by algorithm \ref{alg:PINN}, with loss functions \eqref{eq:lfps}, corresponding to training sets $\train_{int},\train_d$. For $\train_{int}$, we will use Sobol points on the domain $D$ and for $\train_d$, we set up a Cartesian grid on the inner square $D^{\prime}$ \eqref{eq:pd}, see figure \ref{fig:p1} for a representation of both training sets. Note that a simple midpoint rule is used as the quadrature rule \eqref{eq:dqd}, whereas a quasi-Monte Carlo integration is used in \eqref{eq:quad}. 

As the PINNs in algorithm \ref{alg:PINN}, contain several hyperparameters namely, the number of hidden layers $K-1$ (depth) in \eqref{eq:ann1}, number of neurons $d_k\equiv \tilde{d}$ (width) in each hidden layer, the exponent $q$ of the regularization term in loss function \eqref{eq:lf2}, the size $\lambda_{reg}$ of the regularization term and the hyperparameter $\lambda$ in the loss function \eqref{eq:lfps}. We follow \cite{LMR1} and perform an \emph{ensemble training}, with the hyperparameter range specified in Table \ref{tab:1}  to find the best hyperparameters for each experiment. To do so, for each hyperparameter configuration, we will run the BFGS-B optimizer with $30$ randomly selected starting values and select the configuration resulting in the smallest value of the average training loss $\er_T = \sqrt{\er
^2_{d,T} + \lambda\er^2_{p,T}}$ over the retrainings. 

\begin{table}[htbp] 
    \centering
    \renewcommand{\arraystretch}{1.1} 
    
    \footnotesize{
        \begin{tabular}{c  c c c c  c c c c  } 
            \toprule
             $N$  &\bfseries $K-1$ & \bfseries $\tilde{d}$  &$\lambda_{reg}$&\bfseries $\lambda$&  $\er_T$ & $||u - u^\ast||_{L^2}$   & $||u - u^\ast||_{H^1}$  \\ 
            \midrule
            \midrule
            $20\times20$     &4 &24&0.0& 0.001   & 0.0125& 0.70 $\%$ &2.3 $\%$ \\
            \midrule 
            $40\times40$  &4 &24&0.0& 0.001              &0.0127  & 0.53 $\%$ &1.9 $\%$ \\
                \midrule 
            $80\times80 $    &4 &24&0.0& 0.001   &0.0127 & 0.45 $\%$ & 1.7 $\%$ \\
                \midrule 
            $160\times160$    &4 &24&0.0& 0.001   &0.013 & 0.34 $\%$ &1.3 $\%$ \\
            \bottomrule
        \end{tabular}
        \caption{Poisson's equation with noisy measurements: relative percentage generalization errors and training errors for different numbers of training points. }
        \label{tab:p2}
    }
\end{table}

We set $r = {\rm meas}(D^{\prime}) = 9/16$ and for any $N = N_{int} + N_d$, we set $N_d = rN$ and present results for the PINNs in figure \ref{fig:p2}, where we plot the exact solution $u$ \eqref{eq:pex}, the PINN solution $u^{\ast}$ and the error $\hat{u} = u^{\ast} -u$, associated with a PINN generated with $N=20^2 = 400$ training points and with a network with $4$ hidden layers, with $24$ neurons in each layer, $\lambda = 10^{-3}$ in \eqref{eq:lfps} and $\lambda_{reg} = 0$ in \eqref{eq:lf2}. We observe from this figure, that already for this very low number of training points, the PINN generated by algorithm \ref{alg:PINN} is able to approximate the underlying solution of the inverse problem \eqref{eq:ps}, \eqref{eq:dtps}, very accurately. The plot of the error in Figure \ref{fig:p2} (right) shows that the error is mostly concentrated near the boundary and shows a logarithmic behavior, as predicted by the theory and seen in \cite{BO1}. 
%\begin{table}[htbp] 
 %   \centering
  %  \renewcommand{\arraystretch}{1.1} 
    
   % \footnotesize{
    %    \begin{tabular}{c  c c c c  c c c c c  c} 
     %       \toprule
      %       $N$  &\bfseries $K-1$ & \bfseries $d$  &\bfseries $L^1$-reg. &\bfseries $L^2$-reg. &\bfseries $\lambda$&  $\er_T$ & $||u - u^\ast||$   & $||u - u^\ast||_{\omega}$ & Time \\ 
        %    \midrule
         %   \midrule
          %  $20\times20$     &4 &24&0.0&0.0& 0.001   & 0.0008& 0.0040 &0.00038& 23s\\
        %    \midrule 
         %   $30\times30$  &4 &24&0.0&0.0& 0.001              &0.0007  & 0.0032&0.00029&30s\\
          %      \midrule 
          %  $40\times40$  &4 &24&0.0&0.0& 0.001              &0.00072  & 0.0029&0.00025&38s\\
         %       \midrule 
                
          %  $50\times50$  &4 &24&0.0&0.0& 0.001              &0.00066  & 0.0027&0.00023&48s\\
          %      \midrule 
        %    $80\times80 $    &4 &24&0.0&0.0& 0.001   &0.0006 & 0.00257&0.0002&123s\\
         %       \midrule 
          %  $160\times160$    &4 &20&0.0&0.0& 0.001   &0.00052 & 0.0023&0.00018&415s\\
         %   \bottomrule
    %    \end{tabular}
     %   \caption{Poisson's equation: relative $L^2$ norm error $||u(x,y) - u^\ast(x,y)||_{L^2}$ on  $\Omega$ and $\omega$ with \textit{uniformly distributed} noisy-free data for different values of the number of training samples. \textcolor{red}{Updated after ensemble training.}  }
   %     \label{tab:p3}
%    }
%\end{table}
\par The efficiency of PINNs in approximating solutions to this problem is reinforced from Table \ref{tab:p1}, where we present the (percentage relative) generalization errors $\er_G(D)$ \eqref{eq:gerps} and total training errors $\er_T$, for PINNs with $N  = 20^2,40^2,80^2,160^2$, and consequently $N_d = rN$, training points. We also tabulate the (relative percentage) error in $L^2$ between the PINN and the exact solution. From the table, we see that even for very few ($20^2$) training points, the $L^2$ and $H^1$-generalization errors are very low, with $H^1$-error being around $1\%$ and $L^2$-error around $0.3\%$.  This is particularly impressive as it takes approximately $21 s$ to train the PINN for this training set. The error decays slowly as the number of training samples is increased. This is consistent with the logarithmic decay predicted in \eqref{eq:pbd1}. Moreover, the training error is already of the size of $10^{-4}$, even for $20^2$ training points and it is difficult to reduce the training error further. One can possibly further reduce the error by using adaptive activation functions as suggested in \cite{jag3}. However, we did not observe any further reduction in the already low training error by using adaptive activation functions.

Note that we present the generalization errors, averaged over $K=30$ retrainings i.e, different random starting values for the optimizer. The generalization error, corresponding to the starting value with the smallest training loss, is considerably smaller.

Finally, as stated in remark \ref{rmk:nse}, we can also extend the bounds \eqref{eq:pbd}, \eqref{eq:pbd1}, to the case of noisy data, analogous to \eqref{eq:egenbn}. To test the validity of PINNs in this regime, we perturb the right hand side of the data term \eqref{eq:dtps}, with $1\%$ standard normal noise, run algorithm \ref{alg:PINN} to generate PINNs for approximating the inverse problem for the Poisson equation and present the results in Table \ref{tab:p2}. From this table, we see that although slightly higher than in the noise-free case (compare with Table \ref{tab:p1}), the $L^2$ and $H^1$-generalization errors are still very low and decay a bit faster than in the noise-free case. Thus, at least for noisy data with small noise amplitude, we can readily use PINNs for approximating the solutions of the data assimilation problem.  
\section{Heat Equation}
\label{sec:4}
As a model inverse problem for linear parabolic PDEs, we will consider the data assimilation problem for the heat equation.
\subsection{The underlying inverse problem}
With $D \subset \R^d$ being an open, bounded, simply connected set with smooth boundary $\bD$, we consider the heat equation,
\begin{equation}
    \label{eq:ht}
    u_t -\Delta u = f, \quad \forall x \in D, t\in (0,T),
\end{equation}
for some $T \in \R_+$, with $\Delta$ denoting the \emph{spatial} Laplace operator and $f \in L^2(D_T)$ with $D_T = D\times(0,T)$ being the source term. We also assume zero Dirichlet boundary conditions for simplicity.

We will assume that $u \in H^1(((0,T);H^{-1}(D)) \cap L^2((0,T);H^1(D))$ will solve the heat equation \eqref{eq:ht} in a weak sense.
%T \int\limits_{D} \nabla u \cdot \nabla v dx dt = $$\int\limits_0^T \int\limits_{D} f v dx dt,
%\end{equation}
%for all test functions $v \in L^2((0,T);H^1_0(D))$. 

The heat equation \eqref{eq:ht} would have been well-posed if the initial conditions i.e $u_0 = u(x,0)$ had been specified. In fact, the aim of the \emph{data assimilation} problem is to infer the initial conditions and the entire solution field at later times from some measurements of the solution in time. To model this, we consider the following \emph{observables}:
\begin{equation}
    \label{eq:dtht}
    u = g, \quad \forall (x,t) \in D^{\prime} \times (0,T),
\end{equation}
for some open, simply connected observation domain $D^{\prime} \subset D$ and for $g \in L^2(D^{\prime}_T)$ with $D^{\prime}_T = D^{\prime} \times (0,T)$.

Thus solving the data assimilation problem \eqref{eq:ht}, \eqref{eq:dtht}, amounts to finding the solution $u$ of the heat equation \eqref{eq:ht} in the whole space-time domain $D_T$, given data on the observation sub-domain $D_T^{\prime}$. The theory for this data assimilation inverse problem for the heat equation is classical and several well-posedness and stability results are available. Our subsequent error estimates for PINNs rely on the following classical result of Imanuvilov \cite{Im1}, based on the well-known Carleman estimates,
\begin{theorem}
\label{thm:h1}
\cite{Im1}: Let $u \in H^1((0,T);H^{-1}(D)) \cap L^2((0,T);H^1(D))$ solve the heat equation \eqref{eq:ht} with source term $f \in L^2(D_T)$, then for every $0 \leq \bar{T} < T$, there holds,
\begin{equation}
    \label{eq:htst}
    \|u\|_{C([\bar{T},T];L^2(D))} + \|u\|_{L^2((0,T);H^1(D))} \leq C\left(\|u\|_{L^2(D^{\prime}_T)} + \|f\|_{L^2(D_T)} + \|u\|_{L^2(\bD \times (0,T))}\right),
\end{equation}
for some constant $C > 0$.
\end{theorem}
\begin{remark}
Relating the above formulation of the data assimilation inverse problem \eqref{eq:ht}, \eqref{eq:dtht} to the abstract formalism presented in section \ref{sec:2}, we readily see that $X =Y = L^2(D_T), Z = L^2(D^{\prime}_T)$, the differential operator is $\df = \partial_t-\Delta$, interpreted weakly, and the observable is $\map = ID$. Moreover, the abstract conditional stability estimate \eqref{eq:assm} takes the concrete form \eqref{eq:htst}. Thus, this problem falls squarely in the framework considered in section \ref{sec:2} and can be approximated by PINNs. 
\end{remark}
\subsection{PINNs}
We specify the algorithm \ref{alg:PINN} to generate a PINN for approximating the data assimilation inverse problem \eqref{eq:ht}, \eqref{eq:dtht}, in the following steps,
\subsubsection{Training sets}
The training set $\train_{d} = \{z_j \}$ for $z_j = (x,t)_j \in D^{\prime}_T$, with $1 \leq j \leq N_d$, are quadrature points, corresponding to the quadrature rule \eqref{eq:dqd}.
Given the fact that we also consider boundary conditions in \eqref{eq:ht}, we need to slightly modify the training set on which the PDE residual \eqref{eq:res1} is going to be collocated. We take \emph{interior} training set $\train_{int}$ as set of quadrature points $y_i = (x,t)_i \in D_T$, for $1 \leq i \leq N_{int}$, corresponding to the quadrature rule \eqref{eq:quad}. These can be quadrature points for a grid based (composite) Gauss quadrature rule or low-discrepancy sequences such as Sobol points. We also need to introduce \emph{spatial boundary} training set $\train_{sb} = \{\bar{y}_i\}$ for $1 \leq i \leq N_{sb}$, with $\bar{y}_i = (\bar{x},t)_i$, and $t_i \in (0,T)$, $\bar{x}_i \in \bD$, for each $i$. These points can be quadrature points corresponding to a \emph{boundary quadrature rule} i.e for any function $h:\bD \times (0,T) \mapsto \R$, we approximate 
\begin{equation}
    \label{eq:bqd1}
    \int\limits_0^T \int\limits_{\bD} h(x,t) d\sigma(x) dt \approx \sum\limits_{i=1}^{N_{sb}} w^{sb}_i h(\bar{x}_i,t_i),
\end{equation}
and we also assume that this boundary quadrature rule satisfies an error estimate of the form,
\begin{equation}
    \label{eq:bqd}
    \left|\int\limits_0^T \int\limits_{\bD} h(x,t) d\sigma(x) dt - \sum\limits_{i=1}^{N_{sb}} w^{sb}_i h(\bar{x}_i,t_i)\right| \leq C_{bd} N_{sb}^{-\alpha_{sb}},
\end{equation}
with constant $C_{bd} = C\left(\|h\|_{C^\ell(\bD \times (0,T))}\right)$.
\subsubsection{Residuals}
We will require that for parameters $\theta \in \Theta$, the neural networks $(x,t) \mapsto u_{\theta}(x,t) \in C^k(\overline{D_T})$, for $k \geq 2$. We define the following residuals that are needed in algorithm \ref{alg:PINN}. The PDE residual \eqref{eq:res1} is given by,
\begin{equation}
    \label{eq:resht}
    \res_{\theta} = \partial_t u_{\theta} -\Delta u_{\theta} - f, \quad \forall (x,t) \in D_T.
\end{equation}
We need the following residual to account for the boundary data in \eqref{eq:ht},
\begin{equation}
    \label{eq:reshtb}
    \res_{sb,\theta} = u_{\theta}|_{\bD \times (0,T)},
\end{equation}
and the data residual \eqref{eq:resd} is given by,
\begin{equation}
    \label{eq:resdht}
    \res_{d,\theta} = u_{\theta} - g, \quad \forall (x,t) \in D^{\prime}_T.
\end{equation}
\subsubsection{Loss functions}
In algorithm \ref{alg:PINN} for approximating the inverse problem \eqref{eq:ht}, \eqref{eq:dtht}, we will need the following loss function,
\begin{equation}
    \label{eq:lfht}
    J(\theta) = \sum\limits_{j=1}^{N_d} w^d_j|\res_{d,\theta}(z_j)|^2 + \sum\limits_{i=1}^{N_{sb}} w^{sb}_i|\res_{sb,\theta}(\bar{y}_i)|^2 + \lambda \sum\limits_{i=1}^{N_{int}} w_i|\res_{\theta}(y_i)|^2,
\end{equation}
with hyperparamter $\lambda$, residuals defined in \eqref{eq:resht}, \eqref{eq:reshtb}, \eqref{eq:resdht}, training points defined above and weights $w_i$, $w_{i}^{sb}$, $w^d_j$, corresponding to quadrature rules \eqref{eq:quad}, \eqref{eq:bqd1},and \eqref{eq:dqd}, respectively.
\subsection{Estimates on the generalization error}
For any $0 \leq \bar{T} < T$, we define the generalization error for the PINN $u^{\ast} = u_{\theta^{\ast}}$, generated by algorithm \ref{alg:PINN} to approximate the solution $u$ to the data assimilation, inverse problem \eqref{eq:ht}, \eqref{eq:dtht} as,
\begin{equation}
    \label{eq:gerht}
    \er_G(\bar{T}) = \|u - u^{\ast}\|_{C([\bar{T},T];L^2(D))} + \|u - u^{\ast}\|_{L^2((0,T);H^1(D))}.
\end{equation}
We estimate this generalization error in terms of the following training errors, 
\begin{equation}
    \label{eq:etrnht}
    \begin{aligned}
    \er_{d,T} &= \left(\sum\limits_{j=1}^{N_d} w^d_j|\res_{d,\theta^{\ast}}(z_j)|^2\right)^{\frac{1}{2}}, \quad
    \er_{int,T} = \left(\sum\limits_{i=1}^{N_{int}} w_i|\res_{\theta^{\ast}}(y_i)|^2\right)^{\frac{1}{2}}, \quad 
    \er_{sb,T} = \left(\sum\limits_{i=1}^{N_{sb}} w^{sb}_i|\res_{\theta^{\ast}}(\bar{y}_i)|^2\right)^{\frac{1}{2}}.
    \end{aligned}
\end{equation}
Note that the training errors $\er_{int,T},\er_{sb,T}$ and $\er_{d,T}$, can be readily computed from the loss functions \eqref{eq:lf2}, \eqref{eq:lfht}, a posteriori. We have the following estimate on the generalization error in terms of the training error,
\begin{lemma}
\label{lem:h1} 
For $f \in C^{k-2}(D_T)$ and $g \in C^k(D^{\prime}_T)$, with continuous extensions of the functions and derivatives upto the boundaries of the underlying sets and with $k \geq 2$, let $u \in H^1((0,T);H^{-1}(D)) \cap L^2((0,T);H^1(D))$ be the solution of the inverse problem corresponding to the heat equation \eqref{eq:ht} and satisfies the data \eqref{eq:dtht}. Let $u^{\ast} = u_{\theta^{\ast}} \in C^k(D_T)$ be a PINN generated by the algorithm \ref{alg:PINN}, with loss functions \eqref{eq:lf2}, \eqref{eq:lfht}. Then, the generalization error \eqref{eq:gerht} for any $0 \leq \bar{T} < T$ is bounded by, 
\begin{equation}
    \label{eq:hbd}
    \er_G(\bar{T}) \leq C \left(\er_{d,T}+ \er_{int,T}+\er_{sb,T} +  C_{q}^{\frac{1}{2}}N_{int}^{-\frac{\alpha}{2}} + C_{bd}^{\frac{1}{2}}N_{sb}^{-\frac{\alpha_{sb}}{2}}+ C_{qd}^{\frac{1}{2}}N_d^{-\frac{\alpha_d}{2}}\right),
\end{equation}
for some constant $C$ depending on $\bar{T},u,u^{\ast}$ and with constants $C_q= C_q\left(\|\res_{\theta^{\ast}}\|_{C^{k-2}(D_T)}\right)$, $C_{bd} = C_{bd}\left(\|\res_{sb,\theta^{\ast}}\|_{C^{k}(\bD\times(0,T))}\right)$ and $C_{qd} = C_{qd}\left(\|\res_{d,\theta^{\ast}}\|_{C^{k}(D^{\prime}_T)}\right)$, given by the quadrature error bounds \eqref{eq:qassm}, \eqref{eq:bqd} and \eqref{eq:dqassm}, respectively.
\end{lemma}
\begin{proof}
For notational simplicity, we denote $\res = \res_{\theta^{\ast}}, \res_{sb} = \res_{sb,\theta^{\ast}}$ and $\res_d = \res_{d,\theta^{\ast}}$. 

Define $\hat{u} = u^{\ast} - u  \in H^1((0,T);H^{-1}(D)) \cap L^2((0,T);H^1(D))$, by linearity of the differential operator in \eqref{eq:ht} and the data observable in \eqref{eq:dtht} and by definitions \eqref{eq:resht}, \eqref{eq:reshtb} and \eqref{eq:resdht}, we see that $\hat{u}$ satisfies,
\begin{equation}
    \label{eq:hht}
    \begin{aligned}
   \hat{u}_t -\Delta \hat{u} &= \res, \quad \forall x,t \in D_T, \\
   \hat{u} &= \res_{sb}, \quad \forall x \in \bD, t \in (0,T), \\
    \hat{u} &= \res_d, \quad \forall x,t \in D^{\prime}_T.
    \end{aligned}
\end{equation}
Hence, we can directly apply the conditional stability estimate \eqref{eq:htst} to obtain,
\begin{equation}
    \label{eq:hl1}
    \begin{aligned}
    \|\hat{u}\|_{C([\bar{T},T];L^2(D))} &\leq C \left(\|\res_d\|_{L^2(D^{\prime})_T} + \|\res\|_{L^2(D_T)} + \|\res_{sb}\|_{L^2(\bD \times (0,T))} \right), \\
    \end{aligned}
\end{equation}
with the constant $C$ from \eqref{eq:htst}. 

Recognizing that the training errors $\er^2_{d,T},\er_{sb,T}^2$ and $\er_{int,T}^2$ are the quadrature approximations for $\|\res_d\|^2_{L^2(D^{\prime})_T}$, $\|\res_{sb}\|_{L^2(\bD \times (0,T))} $ and $\|\res\|^2_{L^2(D_T)}$ with respect to the quadrature rules \eqref{eq:dqd}, \eqref{eq:bqd1} and \eqref{eq:quad}, respectively and using bounds \eqref{eq:dqassm},\eqref{eq:bqd} and \eqref{eq:qassm} yields the inequality \eqref{eq:pl2}. Substituting it into \eqref{eq:hl1} leads to the desired bound \eqref{eq:hbd} in a straightforward manner. 

\end{proof}

%\begin{figure}[h!]
%        \centering
%        \includegraphics[width=8cm]{{Images/points_heat.png}}
%    \caption{}
%\label{fig:p1}
%\end{figure}
\begin{figure}[h!]
        \centering
        \includegraphics[width=8cm]{{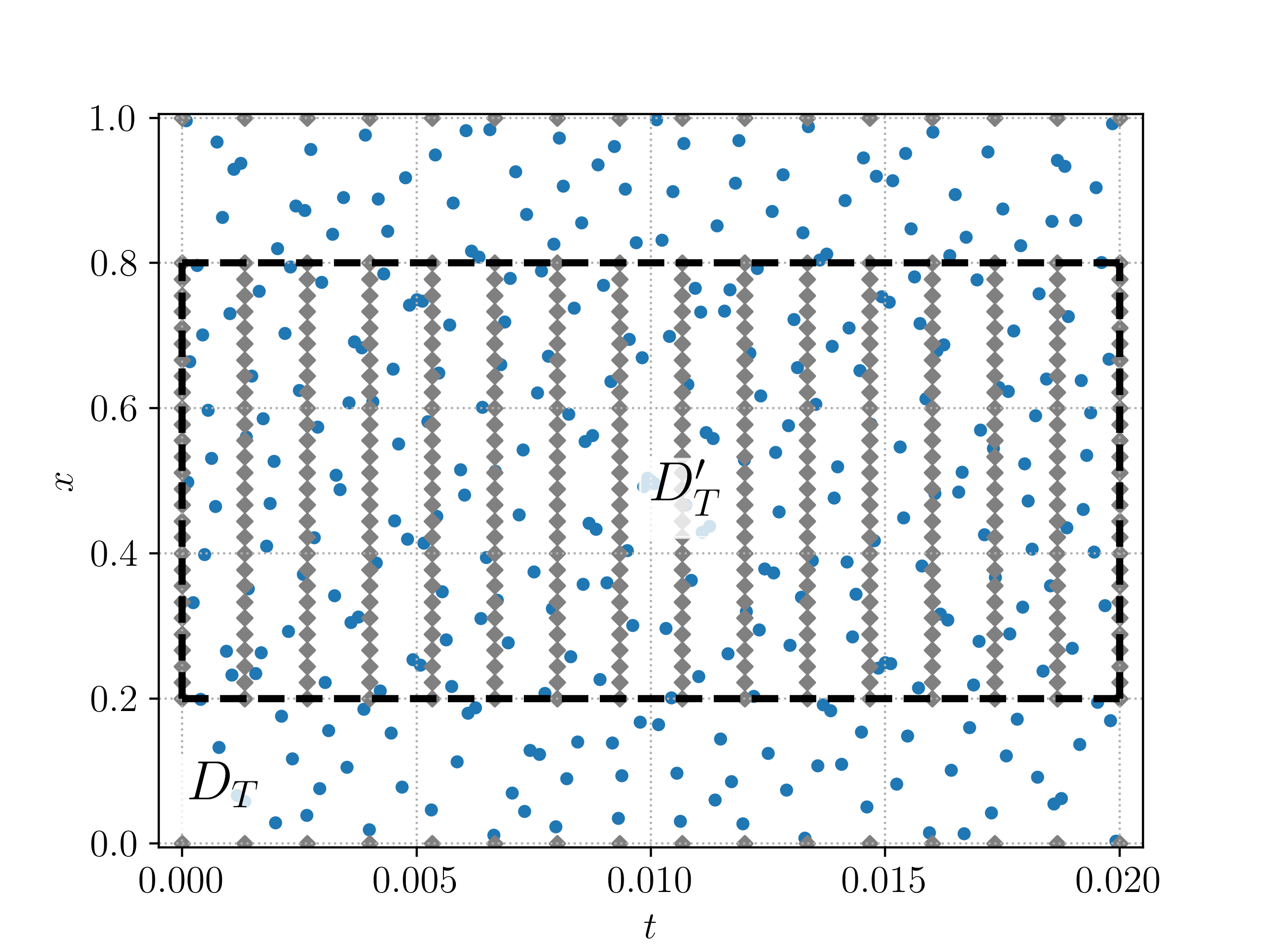}}
    \caption{The domains $D_T,D_T^{\prime}$ for the heat equation numerical experiment. Training set $\train_{int}$ are Sobol points (blue dots) and training set $\train_d  \cup \train_{sb}$ are Cartesian grid points (grey squares).}
\label{fig:h1}
\end{figure}
\begin{figure}[h!]

    \begin{subfigure}{.33\textwidth}
        \centering
        \includegraphics[width=1\linewidth]{{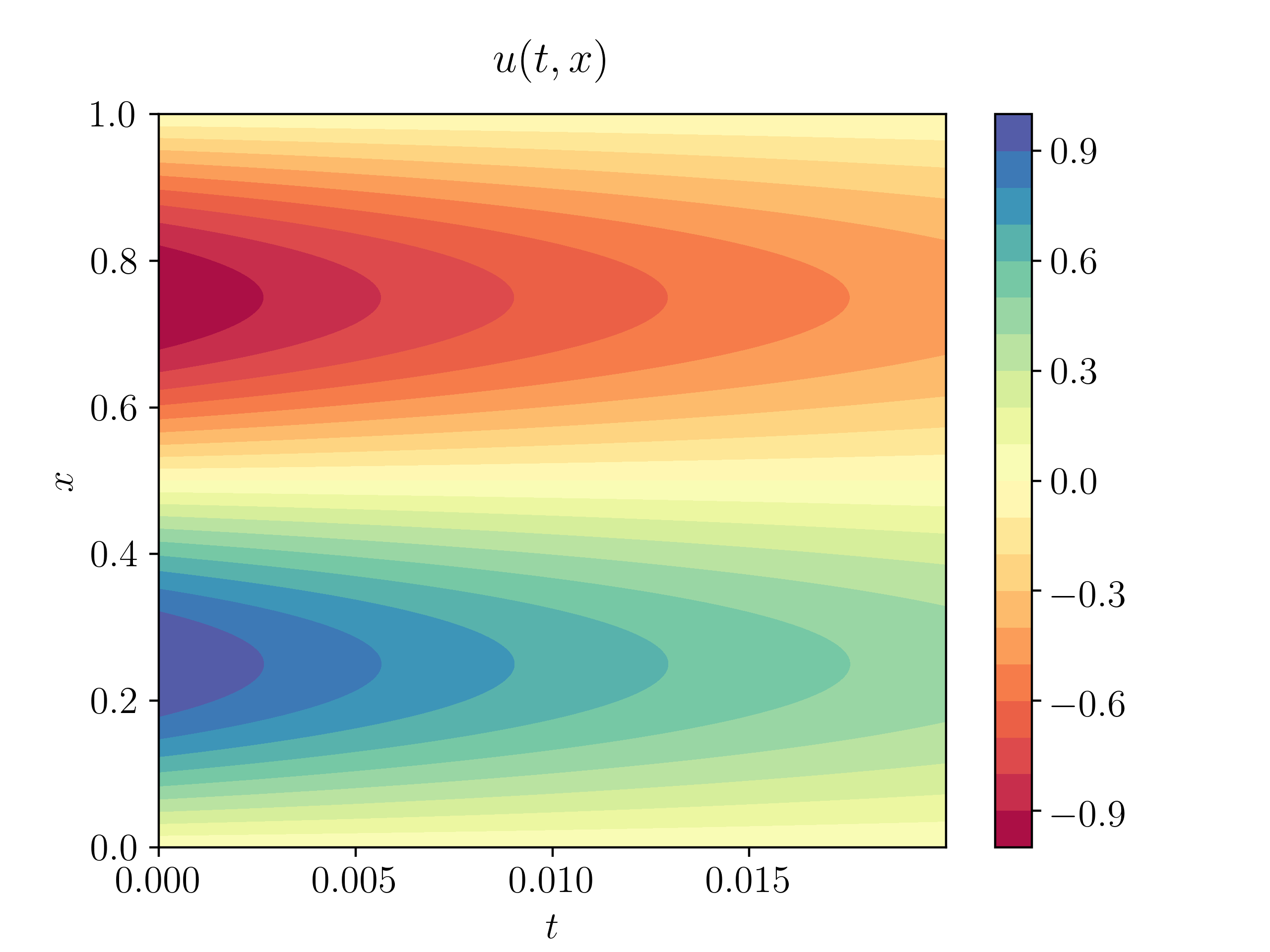}}
        \caption{Exact Solution $u$}
    \end{subfigure}
    \begin{subfigure}{.33\textwidth}
        \centering
        \includegraphics[width=1\linewidth]{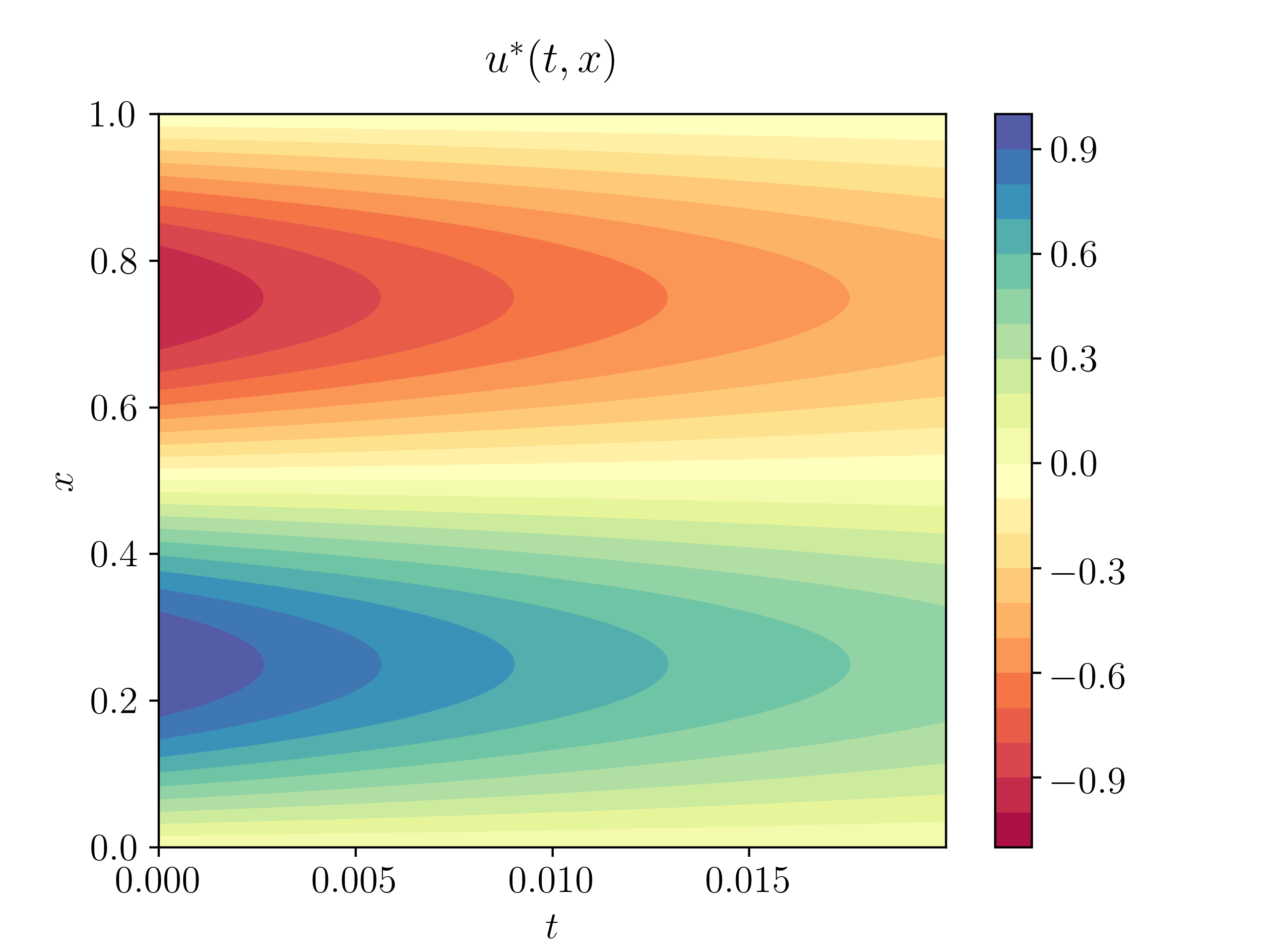}
        \caption{PINN $u^{\ast}$ }
    \end{subfigure}
     \begin{subfigure}{.33\textwidth}
        \centering
        \includegraphics[width=1\linewidth]{{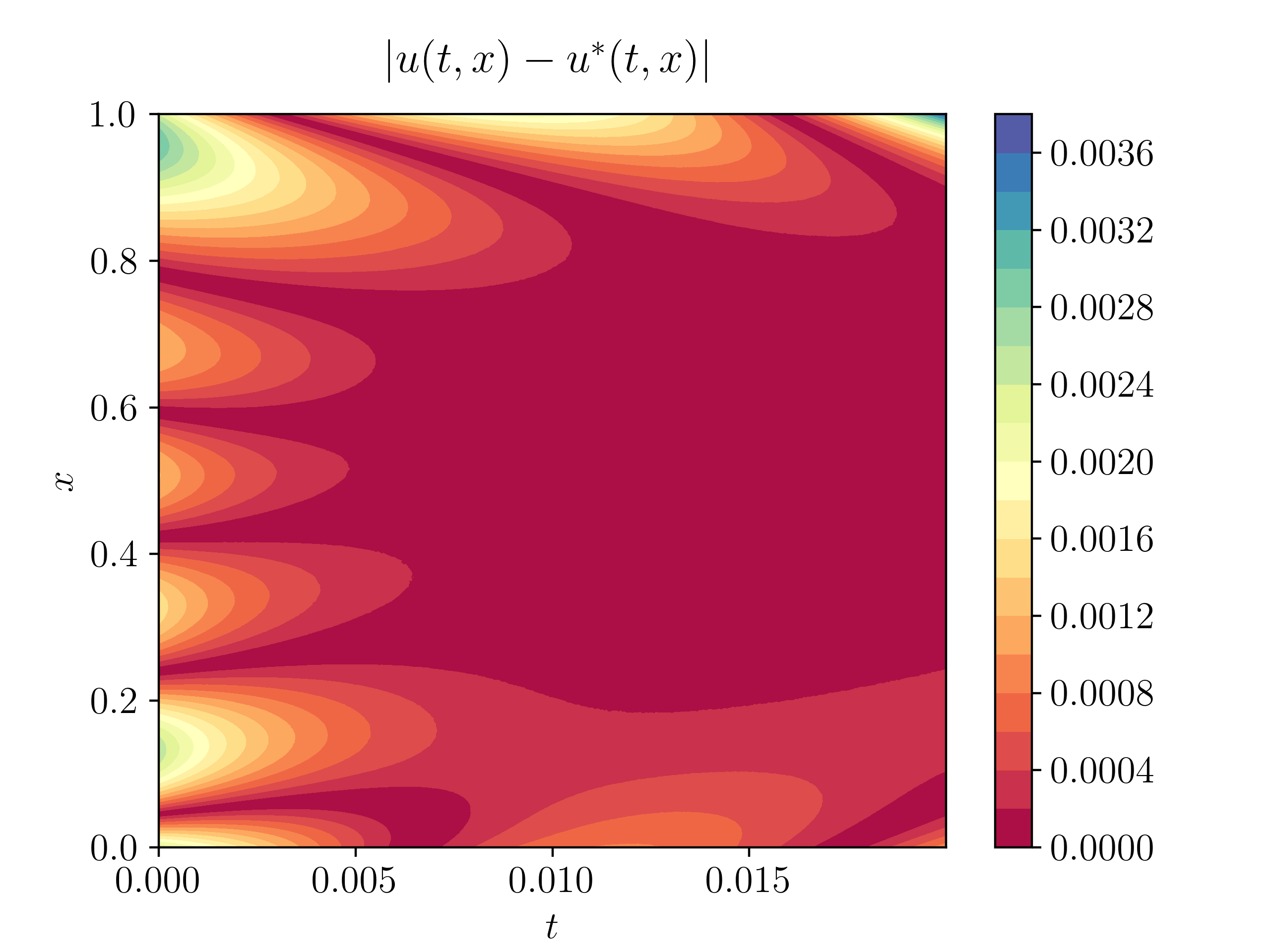}}
        \caption{Relative error $\frac{|\hat{u}|}{\|u\|_{L^2}}$}
    \end{subfigure}
  \caption{Comparison of the Exact solution, the PINN approximation and  relative error for the data assimilation problem for the heat equation, with $N = 16 \times 50$ training points.}
\label{fig:h2}
\end{figure}

\subsection{Numerical experiments}
\subsubsection{Heat equation in one space dimension}
We follow \cite{BO2} and start by considering the heat equation \eqref{eq:ht} in one space dimension (with zero Dirichlet boundary conditions) and setting $D = (0,1)$. The final time is set as $T=0.02$ and the data term in \eqref{eq:dtht} is specified in the observation domain $D^{\prime}_T$, with $D^{\prime} = (a,1-a)$, and as in \cite{BO2}, we set $a = 0.2$. The data term $g$ and the source term $f$ are generated from the exact solution,
\begin{equation}
    \label{eq:hex1}
    u(x,t) = e^{-4\pi^2t^2}\sin(2\pi x).
\end{equation}

In the first numerical experiment, we use Sobol points as training points $\train_{int}$ in the domain $D_T$ and
%$\train_{sb}$ on the spatial boundary $\bD \times (0,T)$ (see figure \ref{fig:h1} for an illustration).
Cartesian grid points are chosen as the training set $\train_d  \subset D_T^{\prime}$ and $\train_{sb} \subset \bD \times (0,T)$ (see figure \ref{fig:h1} for an illustration).
We run the algorithm \ref{alg:PINN} on these training sets to generate the PINN for approximating the data assimilation problem for the heat equation \eqref{eq:ht}, \eqref{eq:dtht}. 

In figure \ref{fig:h2}, we plot the exact solution and the solution field, generated by a PINN, and the relative error (in norm), with a total of $N= N_{int} +N_{sb} +N_d= 16 \times 50$ training points, with $N_{int} = (1-r)N$, and $r = 0.6$ training points and with hyperparameters, shown in Table \ref{tab:h1}. We see from this figure that the PINN is able to approximate the exact solution (in both space and time) to very high accuracy. From the plot of the error (figure \ref{fig:h2} (right)), we see that bulk of the error is concentrated near the initial time i.e $T\approx 0$. This is not surprising as the solution of the heat equation is damped very quickly in time and initial errors are dissipated. 

To further quantity this high accuracy of PINNs, we present the percentage relative errors, $\|u - u^{\ast}\|_{L^2(D_T)}$ and $\|u - u^{\ast}\|_{L^2((0,T);H^1(D))}$ , for a sequence of PINNs, with increasing number of training points. Note that this quantity is a slight perturbation of the generalization error in \eqref{eq:hbd} and can be readily bounded above by the lhs of \eqref{eq:hbd}. For each configuration the model is retrained 20 times and the configuration realizing the lowest value of the average training loss over the retrainings is selected. We see from this table that the errors in both $L^2$ and $H^1$ are very small (significantly less than $1\%$), even for very few $16 \times 50 = 800$ training points. However, there is some saturation effect and the errors do not really decrease on increasing the number of training samples. This can be due to the fact that training error $\er_T$ is already very small, even for the smallest number of training points, and it is difficult to reduce them further with the LBFGS optimizer for the loss function \eqref{eq:lfht}. 
\begin{table}[htbp] 
    \centering
    \renewcommand{\arraystretch}{1.1} 
    
    \footnotesize{
        \begin{tabular}{c  c c c c  c c  c  c} 
            \toprule
             $N$  &\bfseries $K-1$ & \bfseries $\tilde{d}$  &$\lambda_{reg}$ &\bfseries $\lambda$&  $\er_T$ &  $||u - u^\ast||_{L^2}$   & $||u - u^\ast||_{H^1}$  \\ 
            \midrule
            \midrule
            $16\times50$           &  8 &20&0.0& 0.001   &0.001  & 0.18  \% &0.42 \%\\
                \midrule 
            $16\times100    $          &8 &20&0.0& 0.001   & 0.00096& 0.25 \% &0.52 \% \\
            \midrule 
            $16\times200$     &8 &20&0.0& 0.001    &0.00078 & 0.23 \%&0.55 \%\\
            \bottomrule
        \end{tabular}
        \caption{$1$-D Heat equation: relative percentages errors for different values of the number of training samples.}
        \label{tab:h1}
    }
\end{table}
\par Finally, in order to investigate the sensitivity of the PINN errors with respect to the type of training points (underlying quadrature rule), we repeat the numerical experiment for data assimilation with the heat equation (with exact solution \eqref{eq:hex1}), but with randomly chosen training points in the training sets $\train_{int,sb,d}$. All points are chosen with respect to underlying uniform distributions and the corresponding PINNs are generated by running algorithm \ref{alg:PINN}. Note that we can readily prove a version of the generalization error estimate \eqref{eq:hbd} in this setting by adapting the arguments in \cite{MM1} (see Lemma 2.10) and also the recent paper \cite{LMM1}. The resulting (averaged over $K=30$ different randomly chosen training sets) generalization errors in $L^2$ and $H^1$ are shown in Table \ref{tab:h2}. From this table, we observe that the generalization errors are as small as the ones in the case of Sobol and Cartesian training points. Moreover, there is a slight but consistent decay in the error with increasing number of training points. These results indicate considerable robustness of the PINNs, with respect to the choice of training points. 

\begin{table}[htbp] 
    \centering
    \renewcommand{\arraystretch}{1.1} 
    
    \footnotesize{
        \begin{tabular}{c  c c c c  c c c c  c} 
            \toprule
             $N$  &\bfseries $K-1$ & \bfseries $\tilde{d}$  &$\lambda_{reg}$ &\bfseries $\lambda$&  $\er_T$ &  $||u - u^\ast||_{L^2}$   & $||u - u^\ast||_{H^1}$  \\ 
            \midrule
            \midrule
            $16\times50$           &  4 &24&0.0& 0.001   &0.001  & 0.27 \% &0.55 \%\\
                \midrule 
            $16\times100    $          &4 &24&0.0& 0.001   & 0.00098& 0.25 \% &0.52 \% \\
            \midrule 
            $16\times200$     &4 &24&0.0& 0.000788    &0.00076 & 0.22 \%&0.47 \%\\
            \bottomrule
        \end{tabular}
        \caption{$1$-D Heat equation with randomly chosen training points: relative percentages errors for different values of the number of training samples. }
        \label{tab:h2}
    }
\end{table}

\subsubsection{Heat equation in several space dimensions}
For this experiment, we consider the linear heat equation \eqref{eq:ht} in the domain $D_T = [0,1] \times[0,1]^n$  for different space dimensions $n$. We follow the example in \cite{MM1} and consider the heat equation with initial data $\bar{u}(x) = \frac{\|x\|^2}{n}$ and the exact solution given by 
\begin{equation}
    \label{eq:hex2}
    u(x,t)=\frac{\|x\|^2}{n} + 2t. 
\end{equation}
The observation domain is set as $D_T^\prime = [a,1-a] \times[a,1-a]^n$ with $a=0.4$ and the data term \eqref{eq:dtht} is defined by restricting the exact solution \eqref{eq:hex2} to this observation doman. 

All the training sets $\train_{int},\train_{sb}$ and $\train_d$ consists of points, chosen randomly and independently with the underlying uniform distribution. For all dimensions $n$ considered here, we let $N_{int}=8192$, $N_{d}=6144$ and $N_{sb}=2048$, resulting in a total of $N = N_{int} + N_{sb} + N_d = 16384$ training points.

On these training sets, the algorithm \ref{alg:PINN} is run with loss function \eqref{eq:lfht} and the best performing hyperparameters are identified after ensemble training and presented in Table \ref{tab:h_n}. The resulting $L^2$-error $\|u-u^{\ast}\|_{L^2(D_T)}$ (in relative percentages and computed on a test set of $10^5$ randomly chosen samples) is shown in Table \ref{tab:h_n}. We observe from this table that the PINN error is very small, less than $0.1\%$ till $n=10$ space dimensions. From then on, the error increases apparently linearly, with still very low errors of $2\%$ for $n=100$ space dimensions. These results are \emph{striking} on account of the following factors,
\begin{itemize}
    \item The linear increase of error with respect to dimension appears to overcome the well-known \emph{curse of dimensionality}, where one would expect an exponential increase of the error with respect to dimension. 
    \item The relative size of the observation domain $D^{\prime}_T$ with respect to the whole domain $D_T$, shrinks exponentially with dimension. Yet, the PINNs algorithm is able to reconstruct the entire solution field with high accuracy, from observations in this very small domain. 
    \item This experiment should be compared to experiment $3.4.3$ in \cite{MM1}, where the forward problem for the heat equation in very high dimensions was approximated with PINNs. Compared to the results presented there (see Table $6$ in \cite{MM1}), we observe that the PINNs are able to approximate the very high-dimensional inverse problem for the heat equation with \emph{greater accuracy for even smaller number of training samples}, than for the forward problem. This surprising observation merits further investigation and highlights the potential of PINNs in solving inverse problems.

\end{itemize}

\begin{table}[htbp] 
    \centering
    \renewcommand{\arraystretch}{1.1} 
    
    \footnotesize{
        \begin{tabular}{ c  c c c c c c c c} 
            \toprule
            \bfseries $n$  &\bfseries $N$   &\bfseries $K-1$ & \bfseries $\tilde{d}$  &$\lambda_{reg}$&\bfseries $\lambda$ &\bfseries $\er_T$ &\bfseries $||u - u^\ast||_{L^2}$   \\ 
            \midrule
            \midrule
            1           & 16384& 4&20&0.0 & 0.01 &5.1$\cdot10^{-5}$& 0.012\% \\
            \midrule 
              5            & 16384& 4&20& 0.0& 0.01 &0.0003& 0.044\% \\
                \midrule 
             10            & 16384& 4&20& 0.0& 0.001 &0.0003& 0.07\% \\
                \midrule 
             20            & 16384& 4&20& $10^{-6}$& 0.001 &0.0035& 0.62\% \\
               \midrule 
             50            & 16384& 4&20& $10^{-6}$& 0.001 &0.0038& 1.68\% \\
               \midrule 
              100            & 16384& 4&20& $10^{-6}$& 0.001 &0.003& 2.0\% \\

            \bottomrule
        \end{tabular}
    \caption{Data assimilation problem for the Multi ($n$-D) dimensional Heat equation with randomly chosen training points: relative percentages errors for different number of dimensions. }
    \label{tab:h_n}
    }
\end{table}

%We see from the table that the relative $L^2$ norms are very low up to 10 spatial dimensions and slowly increase with dimension, resulting in a low errors of $2.0\%$, even for $100$ space dimensions. This shows the ability of PINNs to possibly overcome the curse of dimensionality, at least for the heat equation. Moreover, the results are consistent with the forward problem presented in \cite{}, despite the considerably smaller value of training samples used.
\section{The Wave equation}
\label{sec:5}
Next, we will consider the wave equation as a model problem for linear hyperbolic PDEs. 
\subsection{The underlying inverse problem}
With $D \subset \R^d$ being an open, bounded, simply connected set with smooth boundary $\bD$, we consider the wave equation with zero Dirichlet boundary conditions,
\begin{equation}
    \label{eq:wv}
    \begin{aligned}
    u_{tt} -\Delta u &= f, \quad \forall (x,t) \in D\times (0,T), \\
    u &\equiv 0, \quad \forall (x,t )\in \bD\times(0,T),
    \end{aligned}
\end{equation}
for some $T \in \R_+$, with $\Delta$ denoting the \emph{spatial} Laplace operator and $f \in L^2(D_T)$ with $D_T = D\times(0,T))$ being the source term.

The forward problem for the wave equation is only well-posed when we know the initial conditions,
\begin{equation}
    \label{eq:wvin} 
    \begin{aligned}
    u(x,0) &= u_0(x), \quad \forall x \in D, \\
    u_t(x,0) &= u_1(x), \quad \forall x \in D,
    \end{aligned}
\end{equation}
for some $u_0 \in L^2(D)$ and $u_1 \in H^{-1}(D)$. However, in many problems of interest, the initial data $u_{0,1}$ are not known and have to be inferred from measurements of the form,
\begin{equation}
\label{eq:dtwv}
u(x,t) = g, \quad (x,t) \in D^{\prime}\times (0,T),
\end{equation}
for some subset $D^{\prime} \subset \bar{D}$. We also denote the observation domain as  $D^{\prime}_T = D^{\prime} \times (0,T)$ .

The resulting \emph{data assimilation} inverse problem consists of finding the initial data $u_0,u_1$ and consequently the entire solution field $u$ of the wave equation \eqref{eq:wv}, given data $f,g$. A slight variant of this problem stems from photoacoustic tomography (PAT) \cite{PAT} and it also arises in control theory, in the form of so-called Luenberger observers \cite{LO}.

The data assimilation problem for the wave equation \eqref{eq:wv}, \eqref{eq:dtwv} has received considerable amount of attention in the mathematical literature and can be solved as long as the following \emph{geometric control condition}, introduced in the seminal paper \cite{BLR}, is satified,
\begin{definition}
\emph{Geometric Control Condition}(GCC) (see \cite{BLR,TRE}): The domain $D^{\prime}_T \subset \bar{D}_T$, is said to satisfy the geometric control condition (gcc) in $D_T$ if every compressed generalized bicharacteristic $(x(s),t(s),\tau(s),\xi(s)$ intersects the set $D^{\prime}_T$ for some $s \in \R$ 
\end{definition}
We refer the reader to \cite{BLR,TRE} for the rather technical definition of generalized bi-characteristics. It roughly states that all light rays in $D_T$ must intersect $D^{\prime}_T$, taking reflections at the boundary into account. An even stronger, sufficient condition that implies the GCC is the so-called $\Gamma$-condition that \cite{GAMM} roughly requires that the final time $T$ and the set $\partial D^{\prime} \cap \bD$ are relatively large. 

Under the GCC, one has follow \cite{BO3} and prove the following well-posedness result for the data assimilation problem for the wave equation \eqref{eq:wv},\eqref{eq:dtwv},
\begin{theorem}
\label{thm:w1}
[\cite{BO3},Theorem 2.2]: Let $D_T = D \times (0,T)$, such that $D \subset \R^d$ is a domain with smooth boundary $\bD$. Let $D^{\prime}_T = D^{\prime} \times (0,T)$, with $D^{\prime} \subset \bar{D}$, satisfy the geometric control condition. If $u \in L^2(D_T)$ be such that $u(\cdot,0) \in L^2(D)$, $\partial_t u(\cdot,0) \in H^{-1}(D)$, $u|_{\bD \times (0,T)} \in L^2(\bD \times (0,T))$, then the following estimate holds,
\begin{equation}
    \label{eq:wvst}
    \sup\limits_{t \in [0,T]} \left(\|u(\cdot,t)\|_{L^2(D)} + \|\partial_t u(\cdot,t)\|_{H^{-1}(D)} \right) \leq \kappa \left(\|u\|_{L^2(D^{\prime}_T)} + \|u_{tt}-\Delta u\|_{L^2(D_T)} + \|u\|_{L^2(\bD \times (0,T))}\right),
\end{equation}
\end{theorem}
with observability constant $\kappa$ that depends on the underlying domain geometry and final time $T$. The proof of this theorem relies heavily on the so-called observability estimates of \cite{BLR}, which are derived by using micro-local propagation of singularities for the wave equation. Alternative proofs use Carleman estimates \cite{CWV}. 

Again, we can recast the data assimilation inverse problem for the wave equation in the abstract formulation of section \ref{sec:2}, with \eqref{eq:wvst} playing the role of the conditional stability estimate \eqref{eq:assm}. Thus, this inverse problem is amenable to efficient approximation by PINNs. 
\subsection{PINNs}
We specify the algorithm \ref{alg:PINN} to generate a PINN for approximating the data assimilation inverse problem \eqref{eq:wv}, \eqref{eq:dtwv}, in the following steps,
\subsubsection{Training sets}
We consider exactly the same training sets as for the heat equationn i.e, $\train_{d} = \{z_j \}$ for $z_j = (x,t)_j \in D^{\prime}_T$, with $1 \leq j \leq N_d$, are quadrature points, corresponding to the quadrature rule \eqref{eq:dqd}, $\train_{int}$ as set of quadrature points $y_i = (x,t)_i \in D_T$, for $1 \leq i \leq N_{int}$, corresponding to the quadrature rule \eqref{eq:quad} and  \emph{spatial boundary} training set $\train_{sb} = \{\bar{y}_i\}$ for $1 \leq i \leq N_{sb}$, with $\bar{y}_i = (\bar{x},t)_i$, and $t_i \in (0,T)$, $\bar{x}_i \in \bD$, for each $i$. These points can be quadrature points corresponding to the \emph{boundary quadrature rule} \eqref{eq:bqd1}. 
\subsubsection{Residuals}
We will require that for parameters $\theta \in \Theta$, the neural networks $(x,t) \mapsto u_{\theta}(x,t) \in C^k(\overline{D_T})$, for $k \geq 2$. We define the following residuals that are needed in algorithm \ref{alg:PINN}. The PDE residual \eqref{eq:res1} is given by,
\begin{equation}
    \label{eq:reswv}
    \res_{\theta} = \partial_{tt} u_{\theta} -\Delta u_{\theta} - f, \quad \forall (x,t) \in D_T.
\end{equation}
We need the following residual to account for the boundary data in \eqref{eq:wv},
\begin{equation}
    \label{eq:reswvb}
    \res_{sb,\theta} = u_{\theta}|_{\bD \times (0,T)}.
\end{equation}
The data residual \eqref{eq:resd} is given by,
\begin{equation}
    \label{eq:resdwv}
    \res_{d,\theta} = u_{\theta} - g, \quad \forall (x,t) \in D^{\prime}_T.
\end{equation}
\subsubsection{Loss functions}
In algorithm \ref{alg:PINN} for approximating the inverse problem \eqref{eq:wv}, \eqref{eq:dtwv}, we will need the following loss function (which is identically the same as the loss function \eqref{eq:lfht} for the heat equation),
\begin{equation}
    \label{eq:lfwv}
    J(\theta) = \sum\limits_{j=1}^{N_d} w^d_j|\res_{d,\theta}(z_j)|^2 +\sum\limits_{i=1}^{N_{sb}} w^{sb}_i|\res_{sb,\theta}(\bar{y}_i)|^2 + \lambda \sum\limits_{i=1}^{N_{int}} w_i|\res_{\theta}(y_i)|^2 ,
\end{equation}
with hyperparamter $\lambda$, residuals defined in \eqref{eq:resdwv}, \eqref{eq:reswv} and \eqref{eq:reswvb}, training points defined above and weights $w^d_j$, $w_i$ and $w^{sb}_i$, corresponding to quadrature rules \eqref{eq:dqd}, \eqref{eq:quad} and \eqref{eq:bqd1}, respectively.
\begin{figure}[h!]
    \begin{subfigure}{.48\textwidth}
        \centering
        \includegraphics[width=1\linewidth]{{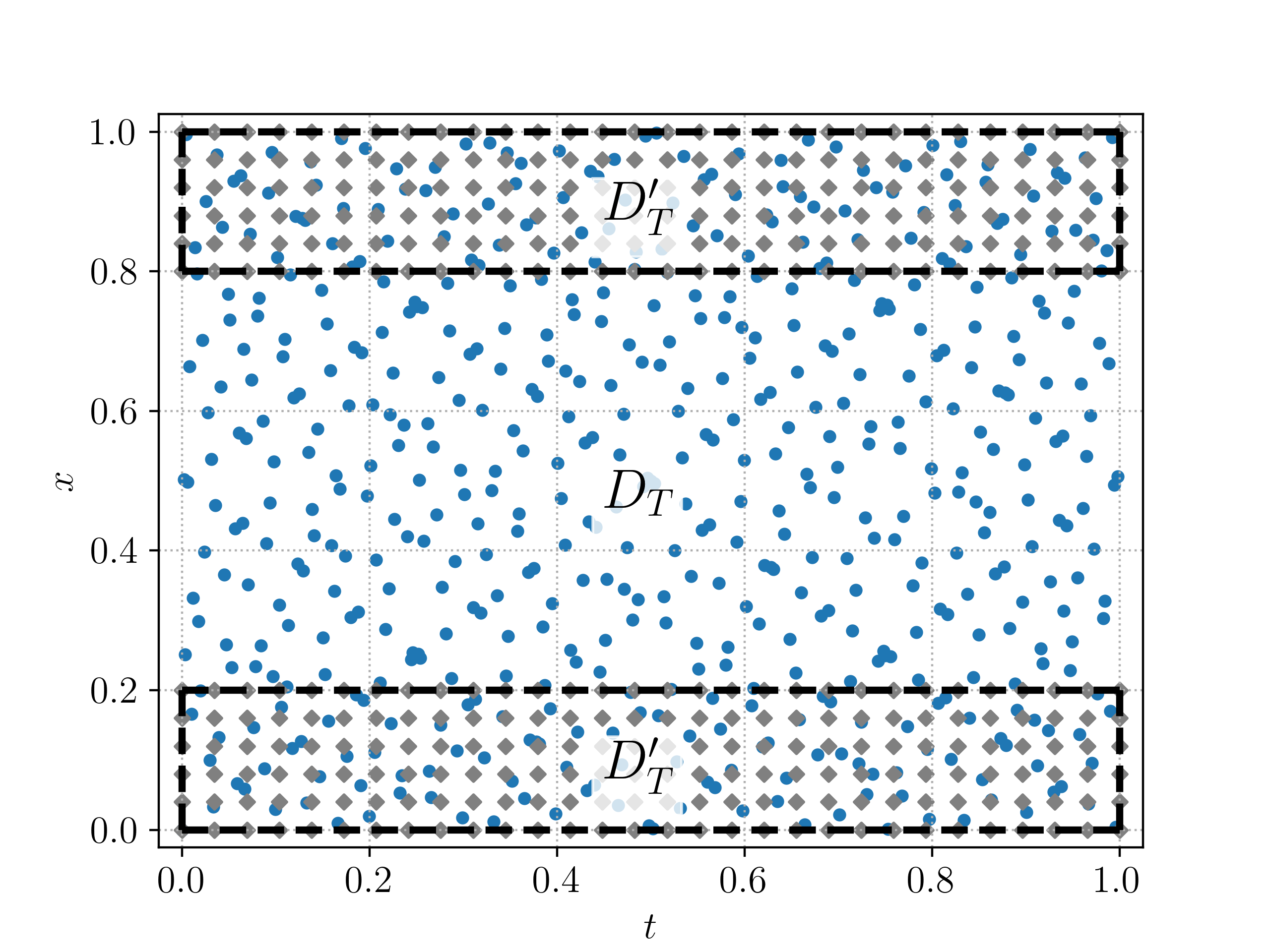}}
        \caption{Domain $D_T^{\prime}$ satisfying geometric control condition}
    \end{subfigure}
     \begin{subfigure}{.48\textwidth}
        \centering
        \includegraphics[width=1\linewidth]{{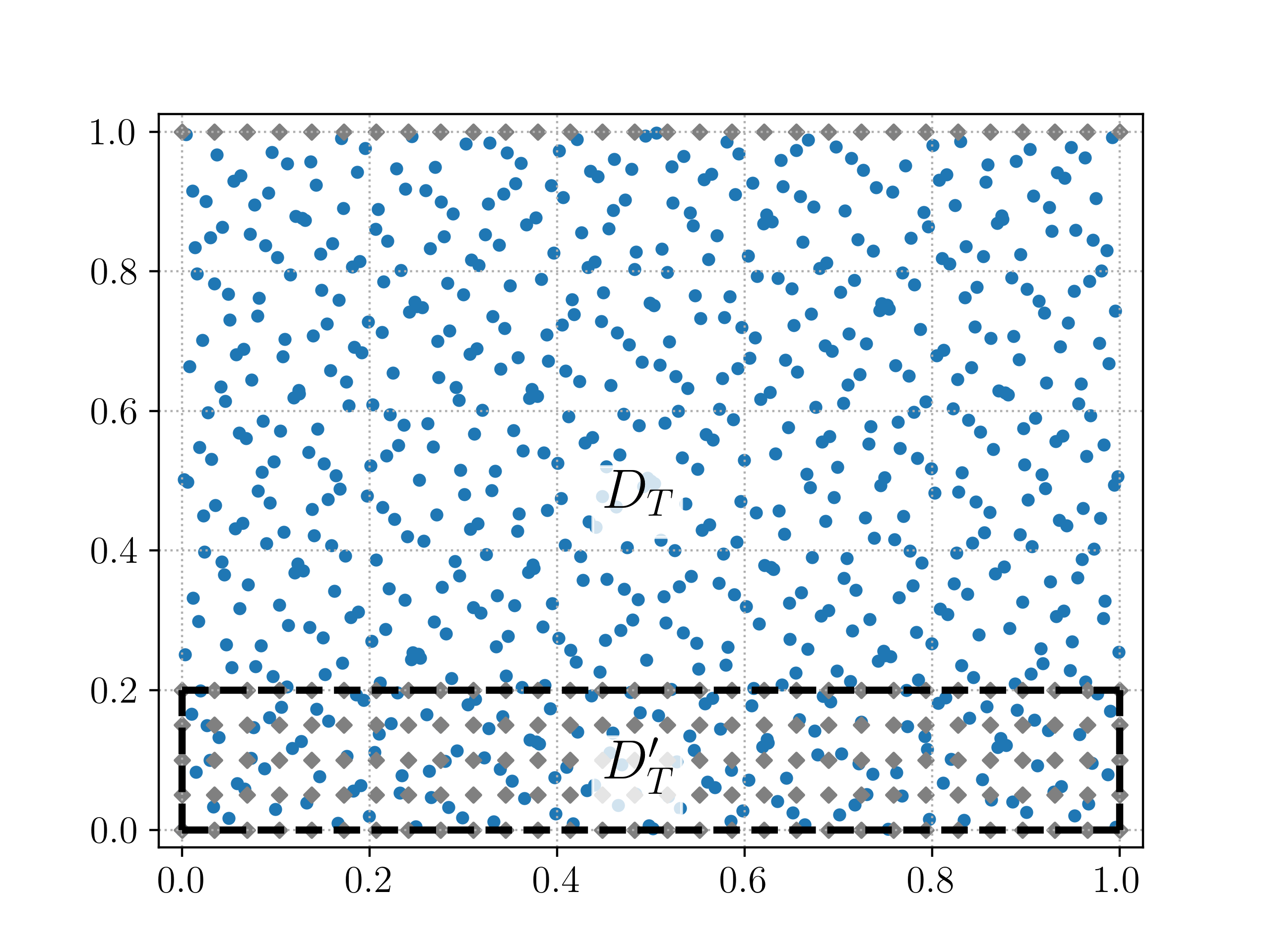}}
        \caption{Domain $D_T^{\prime}$ not satisfying geometric control condition}
    \end{subfigure}
    \caption{The domains $D_T,D_T^{\prime}$ for the numerical experiment for the wave equation. Left: Domain satisfying the geometric control condition (GCC), Right: Domain not satisfying the GCC. Training sets $\train_{int}$ are Sobol points (blue dots) and training set $\train_d\cup\train_{sb}$ are Cartesian grid points (grey squares).}
\label{fig:w1}
\end{figure}
\subsection{Estimates on the generalization error}
Denoting $u^{\ast} = u_{\theta^{\ast}}$ as the PINN, generated by algorithm \ref{alg:PINN} to approximate the data assimilation problem \eqref{eq:wv}, \eqref{eq:dtwv}, for the wave equation, we define the following generalization error,
\begin{equation}
    \label{eq:gerwv}
    \er_G := \|u-u^{\ast}\|_{C^1([0,T];H^{-1}(D))\cap C([0,T];L^{2}(D))}  = \sup\limits_{t \in [0,T]} \left(\|u(\cdot,t)- u^{\ast}(\cdot,t)\|_{L^2(D)} + \|\partial_t u(\cdot,t)- u^{\ast}(\cdot,t)\|_{H^{-1}(D)} \right).
\end{equation}
As in theorem \ref{thm:1} and the previous examples, we will bound the generalization error in terms of the following training errors,
\begin{equation}
    \label{eq:etrnwv}
    \begin{aligned}
    \er_{d,T} = \left(\sum\limits_{j=1}^{N_d} w^d_j|\res_{d,\theta^{\ast}}(z_j)|^2\right)^{\frac{1}{2}}, \quad 
    \er_{int,T} &= \left(\sum\limits_{i=1}^{N_{int}} w_i|\res_{\theta^{\ast}}(y_i)|^2\right)^{\frac{1}{2}}, \quad 
    \er_{sb,T} &= \left(\sum\limits_{i=1}^{N_{sb}} w^{sb}_i|\res_{\theta^{\ast}}(\bar{y}_i)|^2\right)^{\frac{1}{2}}.
    \end{aligned}
\end{equation}
The training errors $\er_{int,T},\er_{sb,T}$ and $\er_{d,T}$, can be readily computed from the loss functions \eqref{eq:lf2}, \eqref{eq:lfwv}, a posteriori. We have the following estimate on the generalization error in terms of the training error,
\begin{lemma}
\label{lem:w1}
Let the domain $D^{\prime}_T$ satisfy the geometric control condition in $D_T$. For $f \in C^{k-2}(D_T)$ and $g \in C^k(D^{\prime}_T)$, with $k \geq 2$, let $u \in C([0,T)];L^2(D)) \cap C^1([0,T];H^{-1}(D))$ be the solution of the data assimilation corresponding to the wave equation \eqref{eq:wv} and satisfies the data \eqref{eq:dtwv}. Let $u^{\ast} = u_{\theta^{\ast}} \in C^k(D_T)$ be a PINN generated by the algorithm \ref{alg:PINN}, with loss functions \eqref{eq:lf2}, \eqref{eq:lfwv}. Then, the generalization error \eqref{eq:gerwv} is bounded by, 
\begin{equation}
    \label{eq:wbd}
    \er_G \leq \kappa \left(\er_{d,T}+ \er_{int,T}+ \er_{sb,T} + C_{q}^{\frac{1}{2}}N_{int}^{-\frac{\alpha}{2}} + C_{bd}^{\frac{1}{2}}N_{sb}^{-\frac{\alpha_{sb}}{2}}+ C_{qd}^{\frac{1}{2}}N_d^{-\frac{\alpha_d}{2}}\right),
\end{equation}
for observability constant $\kappa$ depending on the underlying geometry, $T,u,u^{\ast}$ and with constants $C_q = C_q\left(\|\res_{\theta^{\ast}}\|_{C^{k-2}(D_T)}\right),C_{bd} = C_{bd}\left(\|\res_{sb,\theta^{\ast}}\|_{C^{k}(\bD\times(0,T))}\right)$ and $C_{qd} = C_{qd}\left(\|\res_{d,\theta^{\ast}}\|_{C^{k}(D^{\prime}_T)}\right)$, given by the quadrature error bounds \eqref{eq:qassm}, \eqref{eq:bqd} and \eqref{eq:dqassm}, respectively.
\end{lemma}
\begin{proof}
For notational simplicity, we denote $\res = \res_{\theta^{\ast}}, \res_{sb} = \res_{sb,\theta^{\ast}}$ and $\res_d = \res_{d,\theta^{\ast}}$. 

Define $\hat{u} = u^{\ast} - u  \in C^1([0,T];H^{-1}(D)) \cap C([0,T];L^2(D))$, by linearity of the differential operator and boundary conditions in \eqref{eq:wv} and the data observable in \eqref{eq:dtwv} and by definitions \eqref{eq:reswv},\eqref{eq:reswvb},\eqref{eq:resdht}, we see that $\hat{u}$ satisfies,
\begin{equation}
    \label{eq:hwv}
    \begin{aligned}
   \hat{u}_{tt} -\Delta \hat{u} &= \res, \quad \forall (x,t) \in D_T, \\
   \hat{u}|_{\bD \times (0,T)} &= \res_{sb}, \\
    \hat{u} &= \res_d, \quad \forall (x,t) \in D^{\prime}_T.
    \end{aligned}
\end{equation}
Therefore, we can apply the observability estimate \eqref{eq:wvst} to obtain,
\begin{equation}
    \label{eq:wl1}
    \sup\limits_{t \in [0,T]} \left(||\hat{u}(\cdot,t)\|_{L^2(D)} + \|\partial_t \hat{u}(\cdot,t)\|_{H^{-1}(D)} \right) \leq \kappa \left(\|\res_d\|_{L^2(D^{\prime}_T)} + \|\res_{int}\|_{L^2(D_T)} + \|\res_{sb}\|_{L^2(\bD \times (0,T))}\right),
\end{equation}
with the observability constant $\kappa$ from \eqref{eq:wvst}. 

Realizing that the training errors $\er^2_{d,T}$, $\er_{int,T}^2$ and $\er_{sb,T}^2$ are the quadrature approximations for $\|\res_d\|^2_{L^2(D^{\prime})_T}$,  $\|\res_{int}\|^2_{L^2(D_T)}$and $\|\res_{sb}\|^2_{L^2(\bD \times (0,T))}$ with respect to the quadrature rules \eqref{eq:dqd}, \eqref{eq:quad} and \eqref{eq:bqd1}, respectively and using bounds \eqref{eq:dqassm}, \eqref{eq:qassm} and \eqref{eq:bqd} and substituting the result into \eqref{eq:wl1} yields the desired bound \eqref{eq:wbd}. 
\end{proof}

\begin{figure}[h!]

    \begin{subfigure}{.33\textwidth}
        \centering
        \includegraphics[width=1\linewidth]{{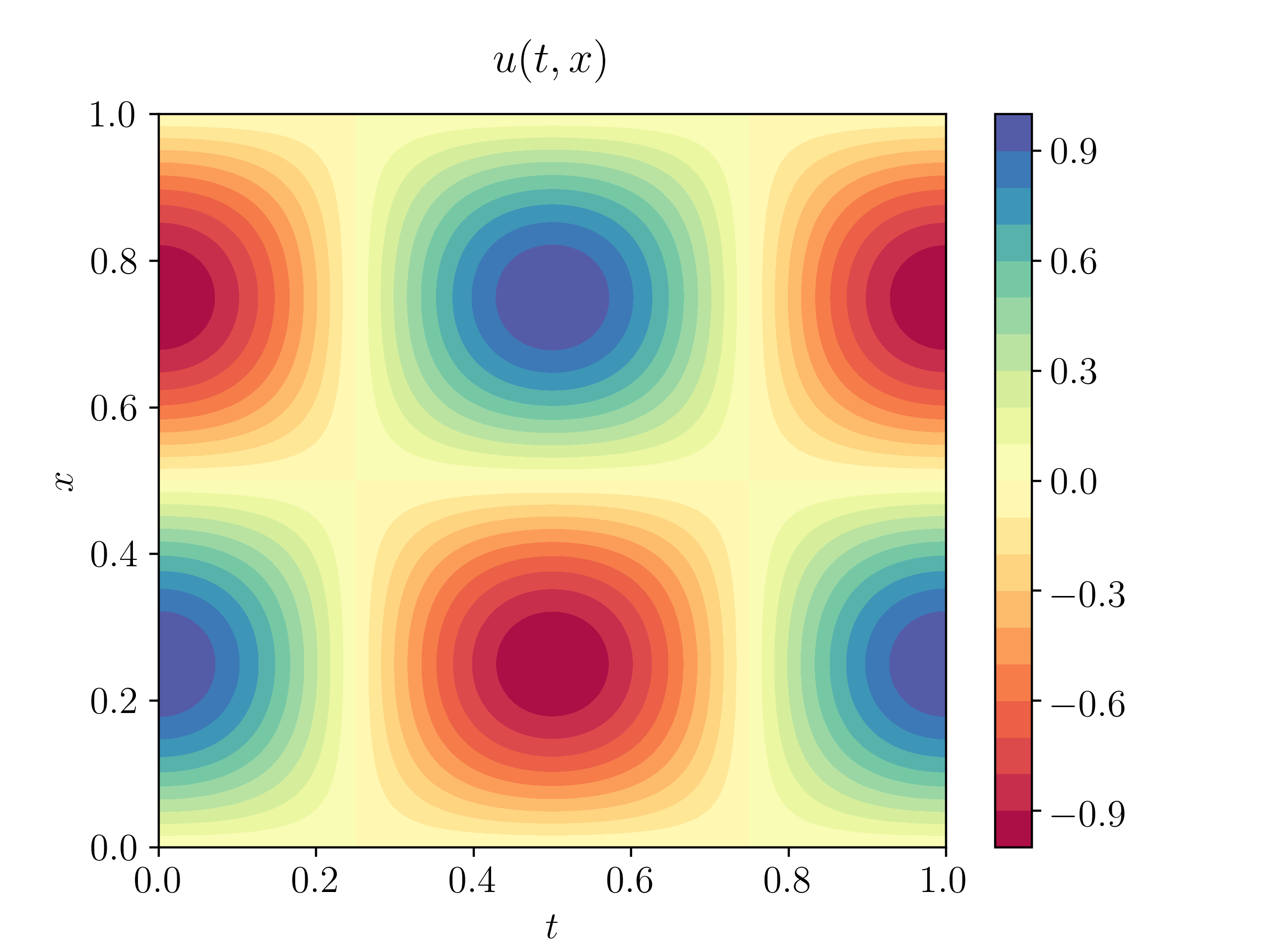}}
        \caption{Exact Solution $u$}
    \end{subfigure}
    \begin{subfigure}{.33\textwidth}
        \centering
        \includegraphics[width=1\linewidth]{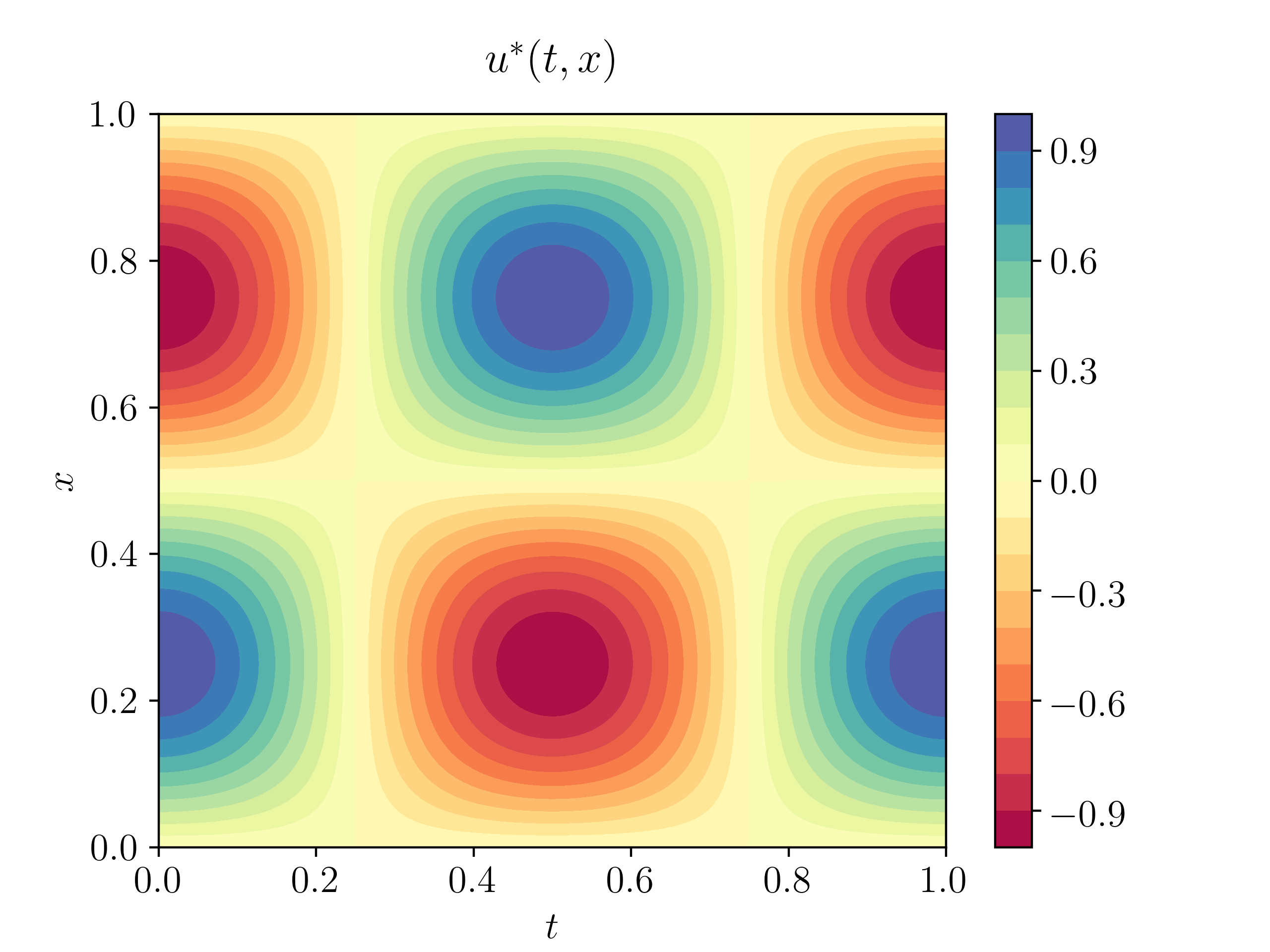}
        \caption{PINN $u^{\ast}$ }
    \end{subfigure}
     \begin{subfigure}{.33\textwidth}
        \centering
        \includegraphics[width=1\linewidth]{{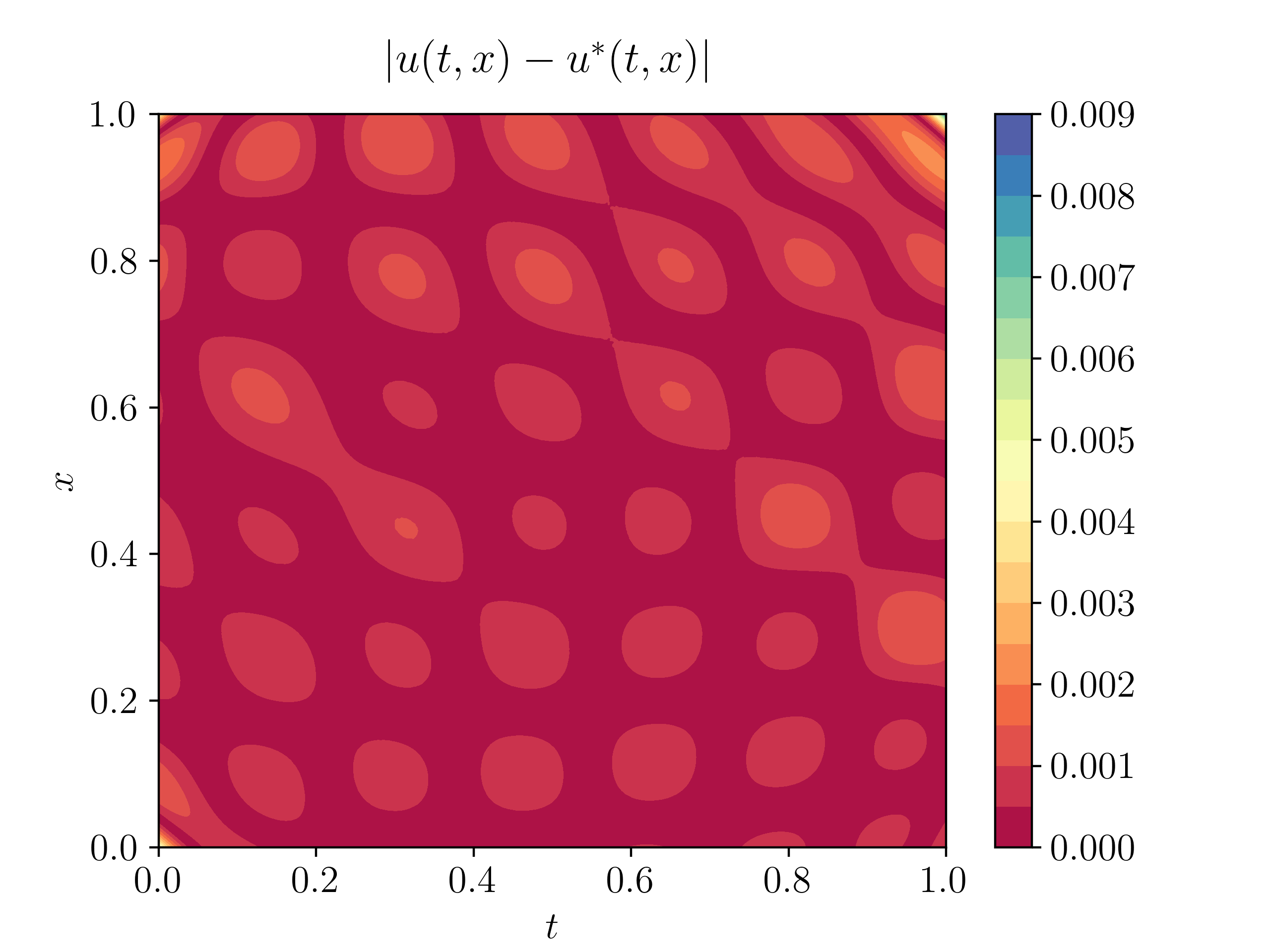}}
        \caption{Relative error $\frac{|\hat{u}|}{\|u\|_{L^2}}$}
    \end{subfigure}
  \caption{Data assimilation problem for the Wave equation in domains shown in figure \ref{fig:w1} (left) satisfying the geometric control condition. Exact solution, PINN with $N = 60 \times 60$ training points and error.}
\label{fig:w2}
\end{figure}
\begin{table}[htbp] 
    \centering
    \renewcommand{\arraystretch}{1.1} 
    
    \footnotesize{
        \begin{tabular}{ c c c c  c c c} 
            \toprule
           $N$  &\bfseries $K-1$ & \bfseries $\tilde{d}$  &$\lambda_{reg}$ &\bfseries $\lambda$&  $\er_T$ &  $||u - u^\ast||_{L^2}$    \\ 
            \midrule
            \midrule 
            $60\times60$ &4 & 24&0.0&0.001      &0.00125  & 0.29 $\%$ \\
                \midrule 
            $90\times90$ &4 & 20&0.0&0.001   &0.0011 & 0.28 $\%$ \\
            \midrule 
            $120\times120$   &4 & 24&0.0&0.001  &0.00085 & 0.21 $\%$ \\
            \bottomrule
        \end{tabular}
        \caption{$1$-D Wave equation with observation domain satisfying the geometric control condition and shown in figure \ref{fig:w1}(left): relative percentage $L^2$ errors for different values of the number of training samples.}
        \label{tab:w1}
    }
\end{table}

\subsection{Numerical experiments}
We present the following experiment proposed in \cite{BO3}, where the authors considered the wave equation in one space dimension, in the domain $D = [0,1]$ and with final time $T=1$. The source term $f$ and data term $g$ in \eqref{eq:dtwv} are generated from the exact solution 
\begin{equation}
    \label{eq:wex}
    u(t,x) = \sin(2\pi t)\sin(2\pi x),
\end{equation}
resulting in $f=0$ and also satisfying the zero Dirichlet boundary conditions of \eqref{eq:wv}. 

For the first numerical experiment, we choose $D^{\prime} = (0,0.2) \cup (0.8,1)$ as the domain on which data $g$ is specified, see figure \ref{fig:w1} for an illustration of the domains. Note that the resulting observation domain $D^{\prime}_T$ satisfies the geometric control condition \cite{BO3}. For these domains, we choose training sets as follows: the training set $\train_{int}$ consists of Sobol points and the training sets $\train_{sb}$  and $\train_d$ are Cartesian grid points.

In figure \ref{fig:w2}, we plot the exact solution, the PINN and the resulting error, corresponding to $N = N_{int} + N_{sb} + N_{d}= 60 \times 60$ and $N_{int} = 0.6N$ training points. We see from this figure that the PINN approximates the underlying solution very well, with errors being small and distributed throughout the space-time domain. The accuracy of the PINNs is further confirmed in Table \ref{tab:w1}, where we present the errors for a sequence of PINNs (with increasing number of training points). We focus on the $ \sup\limits_{t \in [0,T]}\Big(\|u(\cdot,t)-u(\cdot,t)^{\ast}\|_{L^2(D)}\Big)$ error, which can be readily bounded above by the generalization error bound \eqref{eq:wbd}. We see from this table that the approximation errors are very low, with less than $0.3\%$ relative error, already with $60 \times 60$ training points. This error decays further but saturates around a value of $0.2\%$ for more training points. 
\begin{figure}[h!]
     \begin{subfigure}{.33\textwidth}
        \centering
        \includegraphics[width=1\linewidth]{{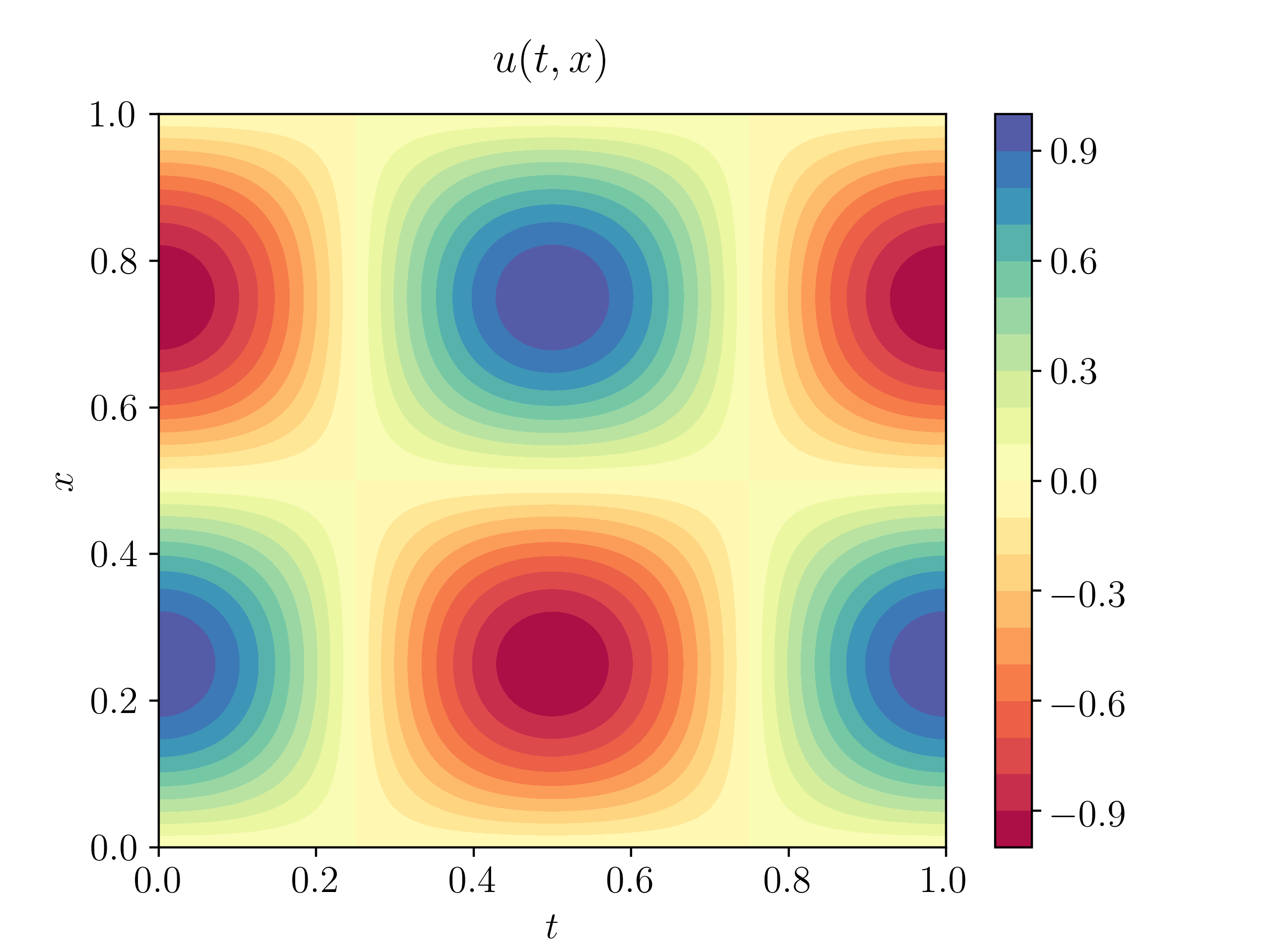}}
        \caption{Exact Solution $u$}
    \end{subfigure}
    \begin{subfigure}{.33\textwidth}
        \centering
        \includegraphics[width=1\linewidth]{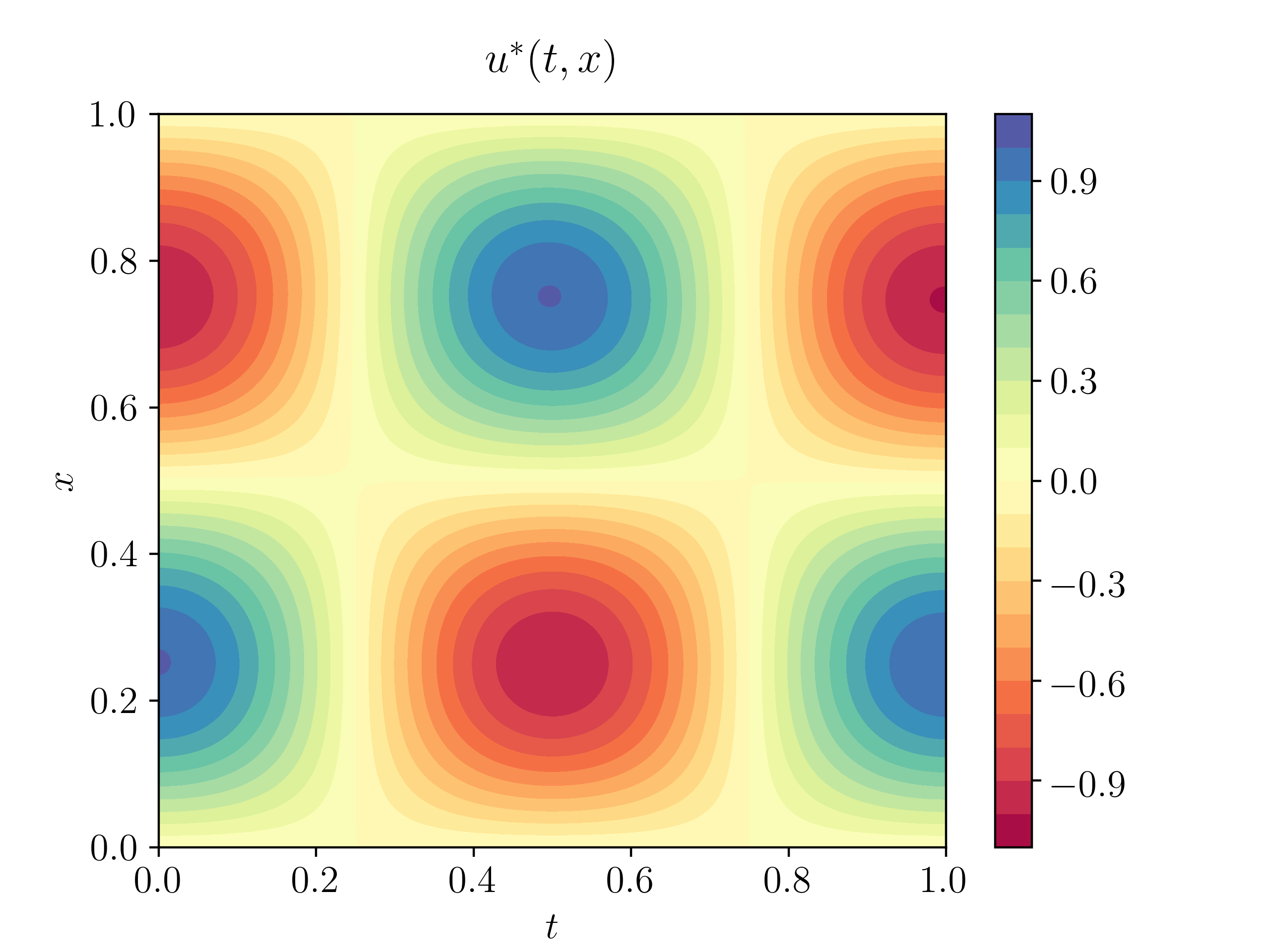}
        \caption{PINN $u^{\ast}$ }
    \end{subfigure}
     \begin{subfigure}{.33\textwidth}
        \centering
        \includegraphics[width=1\linewidth]{{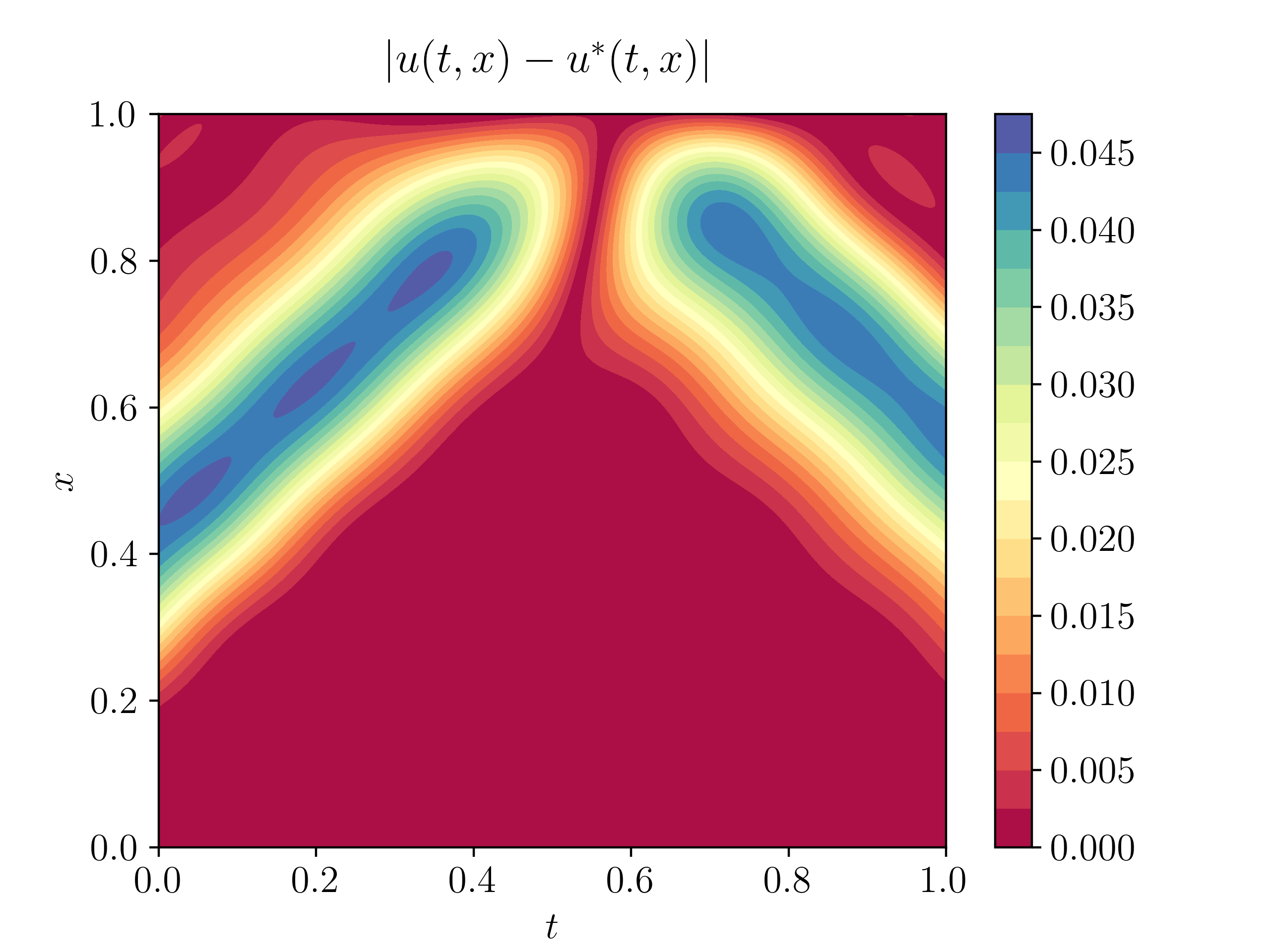}}
        \caption{Relative error $\frac{|\hat{u}|}{\|u\|_{L^2}}$}
    \end{subfigure}
  \caption{Data assimilation problem for the Wave equation in domains shown in figure \ref{fig:w1} (right) satisfying the geometric control condition. Exact solution, PINN with $N = 60 \times 60$ training points and error.}
\label{fig:w3}
\end{figure}

For the second numerical experiment with the wave equation, we choose $D^{\prime} = (0,0.2)$ as the domain on which data $g$ is specified. Note that this domain does not satisfy the geometric control condition \cite{BO3}. The interior, spatial boundary and data training points are chosen similarly to the previous numerical experiment and are illustrated in figure \ref{fig:w1} (right). In figure \ref{fig:w3}, we plot the exact solution and the PINN (and the error), corresponding to $N = N_{int}+ N_{sb} + N_{d} = 60 \times 60$, and $N_{int} = 0.8N$ training points. We see from this figure that although the geometric control condition is not satisfied, the PINN seems to approximate the underlying exact solution rather well. However, a close inspection of the error, plotted in figure \ref{fig:w3} (right) reveals that the error is significantly greater than in the previous numerical experiment where the domain $D^{\prime}_T$ satisfied the geometric control condition. In particular, we see that the error seems to be transported along rays that do not belong to the observable part of the boundary. Nevertheless and as observed in \cite{BO3}, the errors are of small amplitude, even for this example. This is further verified in Table \ref{tab:w2} where we present the generalization errors for this example for a sequence of PINNs, with increasing numbers of training points. From this table, we observe decay of the error with increasing number of training points. The overall generalization error is quite low, around $1.4\%$ for largest number of training points considered here. Note that this error is still an order of magnitude larger than in the previous numerical experiment, where the geometric control condition was satisfied. This experiment nicely illustrates the role of the geometric control condition of \cite{BLR} in this context. 
\begin{table}[htbp] 
    \centering
    \renewcommand{\arraystretch}{1.1} 
    
    \footnotesize{
        \begin{tabular}{ c c c c  c c c} 
            \toprule
            $N$  &\bfseries $K-1$ & \bfseries $\tilde{d}$  &$\lambda_{reg}$ &\bfseries $\lambda$&  $\er_T$ &  $||u - u^\ast||_{L^2}$    \\ 
            \midrule
            \midrule 
            $60\times60$ &4 & 24&0.0&0.001      &0.0011  & 2.2 $\%$ \\
                \midrule 
            $90\times90$ &4 & 24&0.0&0.001   &0.00087 & 1.5 $\%$ \\
            \midrule 
            $120\times120$   &4 & 24&0.0&0.001  &0.00081 & 1.4 $\%$ \\
            \bottomrule
        \end{tabular}
        \caption{$1$-D Wave equation with observation domain not satisfying the geometric control condition and shown in figure \ref{fig:w1} (right): relative percentage $L^2$ errors for different values of the number of training samples.}
        \label{tab:w2}
    }
\end{table}

\section{The Stokes equation}
\label{sec:6}
The effectiveness of PINNs in approximating inverse problems was brilliantly showcased in the recent paper \cite{KAR4}, where the authors proposed PINNs for the data assimilation problem with the Navier-Stokes equation. As a first step towards rigorously analyzing this, we will focus on the much simpler model of the stationary Stokes equation below.
\subsection{The underlying inverse problem}
Let $D \subset \R^d$ be an open, bounded, simply connected set with smooth boundary. We consider the Stokes' equations as a model of stationary, highly viscous fluid:
\begin{equation}
    \label{eq:st}
    \begin{aligned}
    \Delta \bu + \nabla p &=\f, \quad \forall x \in D, \\
    \div(\bu) &= f_d, \quad \forall x \in D.
    \end{aligned}
\end{equation}
Here, $\bu: D \mapsto \R^d$ is the velocity field, $p: D \mapsto \R$ is the pressure and $\f: D\mapsto \R^d$, $f_d: D \mapsto \R$ are source terms. 

Note that the Stokes equation \eqref{eq:st} is not well-posed as we are not providing any boundary conditions. In the corresponding data assimilation problem \cite{BH1} and references therein, one provides the following \emph{data},
\begin{equation}
    \label{eq:dtst}
    \bu = \g, \quad \forall x \in D^{\prime},
\end{equation}
for some open, simply connected set $D^{\prime} \subset D$.
Thus, the data assimilation inverse problem for the Stokes equation amounts to inferring the velocity field $\bu$ (and the pressure $p)$), given $\f,f_d$ and $\g$. In particular, we wish to find solutions $\bu \in H^1(D;\R^d)$ and $p \in L_0^2(D)$ (i.e, square integrable functions with zero mean), such that the following holds, 
\begin{equation}
    \label{eq:wkst}
    \begin{aligned}
    \int\limits_D \nabla \bu\cdot \nabla \bv dx + \int\limits_D p \div(\bv) dx &= \int\limits_D \f \bv dx, \\
    \int\limits_D \div(\bu) w dx &= \int\limits_D f_d w dx, \end{aligned}
\end{equation}
for all test functions $\bv \in H^1_0(D;\R^d)$ and $w \in L^2(D)$.

The well-posedness and conditional stability estimates for the data assimilation problem for the Stokes equation \eqref{eq:st}, \eqref{eq:dtst} has been extensively investigated in \cite{Uhl} and references therein. In particular, we have the following stability estimate,
\begin{theorem}
\label{thm:st1}
Let $D^{\prime} \subset D$ and let $B_{R_1}(x_0)$ be the largest ball, $B_{R_1}(x_0) \subset D^{\prime}$. Let $\bu \in H^1(D;\R^d)$ and $p \in L_0^2(D)$ satisfy \eqref{eq:wkst} for all test functions $\bv \in H^1_0(D;\R^d)$ and $w \in L^2(D)$. Then for any $\f \in L^2(D;\R^d)$, $f_d \in L^2(D)$ and for any $R_2 > R_1$ such that $B_{R_2}(x_0) \subset D$, the following bound holds,
\begin{equation}
    \label{eq:stst}
\|\bu\|^2_{L^2(B_{R_2}(x_0))} \leq C\left( \cost_{\f,f_d} + \cost_{\bu}^{\tau}\cost_{\f,f_d}^{1-\tau} + \left(\cost_{\f,f_d}^{1-\tau}\cost_{\bu}^{1-\tau}\right)\|\bu\|_{L^2(D^{\prime})}^{2\tau}\right),
\end{equation}
with constants defined as,
\begin{equation}
    \label{eq:stc}
    \cost_{\f,f_d}:= \|\f|^2_{L^2(D)} + \|f_d\|^2_{L^2(D)}, \quad \cost_{\bu}:= \|\bu\|_{L^2(D)}^2.
\end{equation}

\end{theorem}
\begin{proof}
Let $\bu_1,p_1$ be solutions of the following boundary-value problem for the Stokes equations,
\begin{equation}
    \label{eq:sth}
    \begin{aligned}
    \Delta \bu_1 + \nabla p_1 &= \f, \quad \forall x \in D, \\
    \div(\bu_1) &= f_d, \quad \forall x \in D, \\
      \bu_1 &\equiv 0, \quad \forall x \in \bD.
    \end{aligned}
\end{equation}
We know from classical theory for the Stokes' equation \cite{SER}, \cite{BH1}(estimate 2.2) that,
\begin{equation}
    \label{eq:pst1}
    \|\bu_1\|_{H^1(D)} + \|p_1\|_{L^2(D)} \leq C \left(\|\f\|_{L^2(D)} + \|f_d\|_{L^2(D)}\right),
\end{equation}
for some constant $C$ that depends only on the domain $D$.

Define $\bu_2,p_2$ as the solutions of the following homogeneous Stokes' equations,
\begin{equation}
    \label{eq:sth1}
    \begin{aligned}
    \Delta \bu_2 + \nabla p_2 &=0, \quad \forall x \in D, \\
    \div(\bu_2) &= 0, \quad \forall x \in D.
    \end{aligned}
\end{equation}
We know that for any $R_1 < R_2 < R_3$, such that $B_{R_2}(x_0) \subset D$, the following \emph{three balls} inequality \cite{Uhl} holds for the solutions $\bu_2 \in H^1(D;\R^d)$ and $p \in H^1(D)$of the homogeneous Stokes equations \eqref{eq:sth1},
\begin{equation}
    \label{eq:3bls}
    \int_{B_{R_2}(x_0)} |\bu_2|^2 dx \leq C\left(\int_{B_{R_1}(x_0)} |\bu_2|^2 dx \right)^{\tau}\left(\int_{B_{R_3}(x_0)} |\bu_2|^2 dx \right)^{1- \tau},
\end{equation}
with a constant $C$ and $\tau \in (0,1)$ that depends on $\frac{R_1}{R_3}$, $\frac{R_2}{R_3}$ and $d$ (see \cite{Uhl} for more on the constants). 

It is easy to check that $\bu = \bu_1 + \bu_2$ and $p = p_1 + p_2$, satisfy the Stokes equation \eqref{eq:st} in the weak sense, i.e they satisfy \eqref{eq:wkst}. By identifying all constants in the following calculation with a generic constant $C$, we have,
\begin{align*}
    \int_{B_{R_2}(x_0)} |\bu|^2 dx &= \int_{B_{R_2}(x_0)} |\bu_1 + \bu_2|^2 dx \\
    &\leq C \left(\int_{B_{R_2}(x_0)} |\bu_1|^2 dx  + \int_{B_{R_2}(x_0)} |\bu_2|^2 dx \right)
    \leq C \left(\int_{D} |\bu_1|^2 dx  + \int_{B_{R_2}(x_0)} |\bu_2|^2 dx \right) \\
    &\leq C\left(\|\f\|_{L^2(D)}^2 + \|f_d\|_{L^2(D)}^2 +\left(\int_{B_{R_1}(x_0)} |\bu_2|^2 dx \right)^{\tau}\left(\int_{B_{R_3}(x_0)} |\bu_2|^2 dx \right)^{1- \tau} \right) \quad ({\rm by}~\eqref{eq:pst1},\eqref{eq:3bls}) \\
    &\leq C\left(\|\f\|_{L^2(D)}^2 + \|f_d\|_{L^2(D)}^2 +\left(\int_{B_{R_1}(x_0)} |\bu-\bu_1|^2 dx \right)^{\tau}\left(\int_{B_{R_3}(x_0)} |\bu-\bu_1|^2 dx \right)^{1- \tau} \right) \\
    &\leq  C\left(\|\f\|_{L^2(D)}^2 + \|f_d\|_{L^2(D)}^2    + \|\bu\|_{L^2(B_{R_1}(x_0))}^{2\tau}\|\bu\|_{L^2(D)}^{2(1-\tau)}\right) \\ 
    &+ C\left(\|\bu\|_{L^2(B_{R_1}(x_0))}^{2\tau}\|\bu_1\|_{L^2(D)}^{2(1-\tau)} + \|\bu_1\|_{L^2(D)}^{2\tau}\|\bu\|_{L^2(D)}^{2(1-\tau)} + \|\bu_1\|^2_{L^2(D)}\right).
\end{align*}
Substituting \eqref{eq:pst1} and identifying the constants yields the desired bound \eqref{eq:stst}.

\end{proof}
As in the previous examples, this inverse problem can be recast in terms of the abstract formalism of section \ref{sec:2}, with the conditional stability estimate \eqref{eq:stst} occupying the role of \eqref{eq:assm} and making this inverse problem amenable to efficient approximation by PINNs. 
\subsection{PINNs}
We specify the algorithm \ref{alg:PINN} to generate a PINN for approximating the inverse problem \eqref{eq:st}, \eqref{eq:dtst} in the following steps,
\subsubsection{Training sets}
As the training set $\train_{int}$ in algorithm \ref{alg:PINN}, we take a set of quadrature points $y_i \in D$, for $1 \leq i \leq N_{int}$, corresponding to the quadrature rule \eqref{eq:quad}. These can be quadrature points for a grid based (composite) Gauss quadrature rule or low-discrepancy sequences such as Sobol points. Similarly, the training set $\train_{d} = \{z_j \}$ for $z_j \in D^{\prime}$, with $1 \leq j \leq N_d$, are quadrature points, corresponding to the quadrature rule \eqref{eq:dqd}.
\subsubsection{Residuals}
We will require that for parameters $\theta \in \Theta$, the neural networks $\bu_{\theta} \in C^k(D;\R^d),p_{\theta} \in C^k(D;\R)$, for $k \geq 2$.  We define the following residuals that are needed in algorithm \ref{alg:PINN}. The PDE residual \eqref{eq:res1}, consists of two parts in this case (see also \cite{MM1}, section 5), given by,
\begin{equation}
    \label{eq:resst1}
    \res_{\theta,\bu} = \Delta \bu_{\theta} +\nabla p_{\theta} - \f, \quad \forall x \in D,
\end{equation}
and
\begin{equation}
    \label{eq:resst2}
    \res_{\theta,div} = \div(\bu_{\theta})- \g, \quad \forall x \in D,
\end{equation}
The data residual is given by,
\begin{equation}
    \label{eq:resdst}
    \res_{d,\theta} = \bu_{\theta} - \g, \quad \forall x \in D^{\prime}.
\end{equation}
\subsubsection{Loss functions}
In algorithm \ref{alg:PINN} for approximating the inverse problem \eqref{eq:st}, \eqref{eq:dtst}, we will need the following loss function,
\begin{equation}
    \label{eq:lfst}
    J(\theta) = \sum\limits_{j=1}^{N_d} w^d_j|\res_{d,\theta}(z_j)|^2 + \lambda \bigg(\sum\limits_{i=1}^{N_{int}} w_i|\res_{\bu,\theta}(y_i)|^2 + \sum\limits_{i=1}^{N_{int}} w_i|\res_{div,\theta}(y_i)|^2\bigg),
\end{equation}
with hyperparamter $\lambda$, residuals defined in \eqref{eq:resps}, \eqref{eq:resdps}, training points defined above and weights $w_i$, $w^d_j$, corresponding to quadrature rules \eqref{eq:quad} and \eqref{eq:dqd}, respectively. 
\begin{figure}[h!]
        \centering
        \includegraphics[width=8cm]{{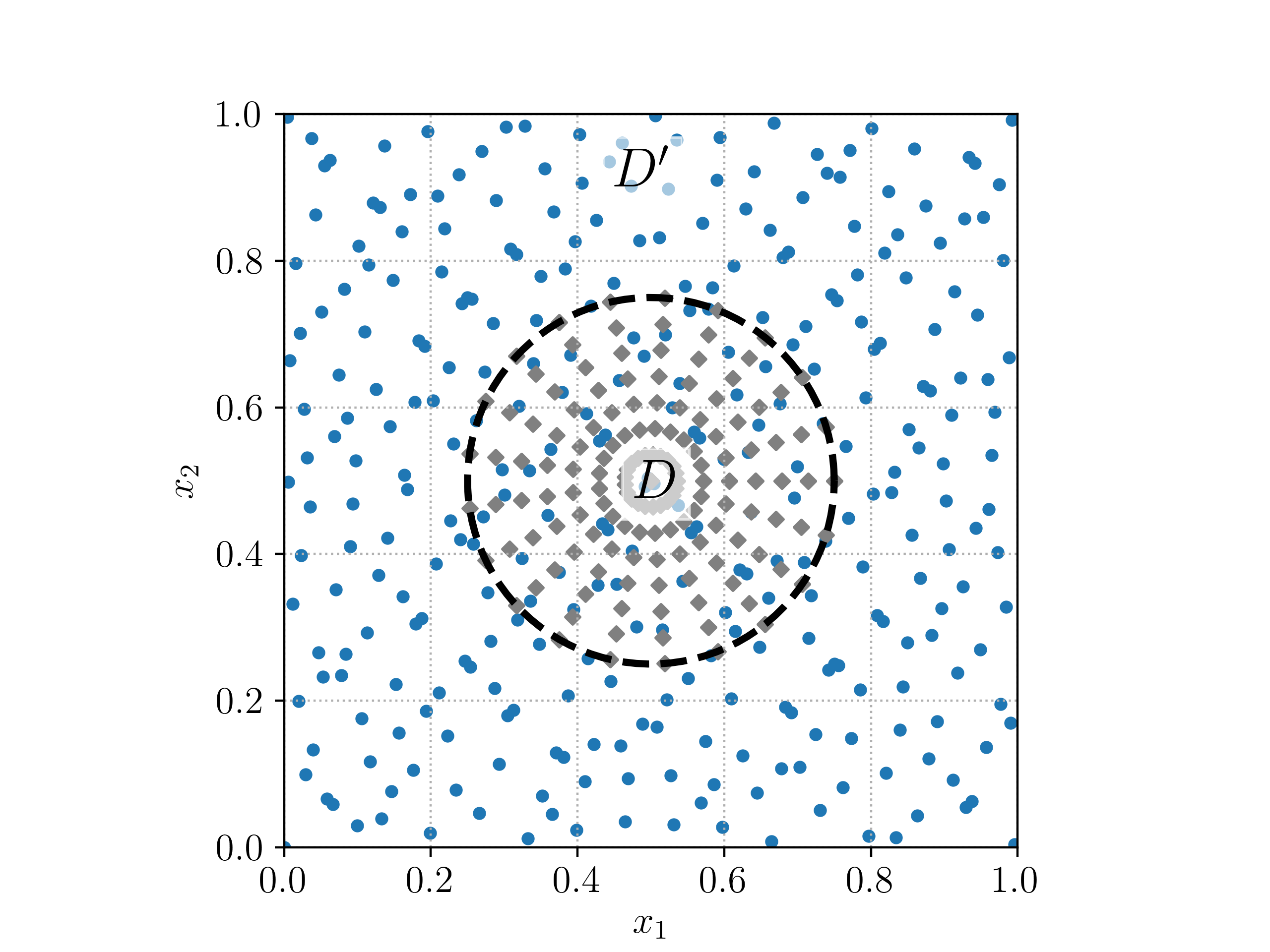}}
    \caption{The domains $D,D^{\prime}$ for the numerical experiment for the Stokes equation. Training set $\train_{int}$ are Sobol points (blue dots) and training set $\train_d$ are Radial (Cartesian) grid points (grey squares).}
\label{fig:st1}
\end{figure}

\subsection{Estimates on the generalization error}
Let $B_{R_1}(x_0)$ be the largest ball inside the observation domain $D^{\prime} \subset D$. We will consider balls $B_{R}(x_0) \in D$
such that $R > R_1$ and estimate the generalization error,
\begin{equation}
    \label{eq:gerst}
    \er_G(B_R) := \|\bu-\bu^{\ast}\|_{L^2(B_R(x_0))},
\end{equation}
with $\bu^{\ast} = \bu_{\theta^{\ast}}$ is the PINN generated by algorithm \ref{alg:PINN}. We will estimate this generalization error in terms of the training errors defined by,
\begin{equation}
    \label{eq:etrnst}
    \er_{d,T} = \left(\sum\limits_{j=1}^{N_d} w^d_j|\res_{d,\theta^{\ast}}(z_j)|^2\right)^{\frac{1}{2}}, \quad \er_{p,T} = \left(\sum\limits_{i=1}^{N_{int}} w_i|\res_{\bu,\theta^{\ast}}(y_i)|^2 + \sum\limits_{i=1}^{N_{int}} w_i|\res_{div,\theta}(y_i)|^2\right)^{\frac{1}{2}},
\end{equation}
in the following lemma,
\begin{lemma}
\label{lem:st1} 
For $\f \in C^{k-2}(D;\R^d), f_d \in C^{k-1}(D)$ and $\g \in C^k(D^{\prime})$, with $k \geq 2$. Let $\bu \in H^1(D;\R^d)$ and $p \in H^1(D)$ be the solution of the inverse problem corresponding to the Stokes' equations \eqref{eq:st} i.e, they satisfy \eqref{eq:wkst} for all test functions $\bv \in H^1_0(D;\R^d),w\in L^2(D)$ and satisfies the data \eqref{eq:dtst}. Let $\bu^{\ast} = \bu_{\theta^{\ast}} \in C^k(D;\R^d), p^{\ast} = p_{\theta^{\ast}} \in C^k(D)$ be a PINN generated by the algorithm \ref{alg:PINN}, with loss functions \eqref{eq:lf2}, \eqref{eq:lfst}. Let $B_{R_1}(x_0)$ be the largest ball inside $D^{\prime} \subset D$.
Then, the generalization error \eqref{eq:gerst} for balls $B_{R}(x_0) \in D$ with $R > R_1$ is bounded by,
\begin{equation}
    \label{eq:stbd}
    \begin{aligned}
    &\left(\er_G(B_R)\right)^2 \leq \\
    &C \left[ \er_{p,T}^2 + \cost_{\hat{\bu}}^\tau \er_{p,T}^{2(1-\tau)} + \cost_{\hat{\bu}}^{1-\tau} \er_{p,T}^{2(1-\tau)} \er_{d,T}^{2\tau} + C_q N^{-\alpha} + \cost^{\tau}_{\hat{\bu}} C_q^{(1-\tau)}N^{-(1-\tau)\alpha} + \cost^{1-\tau}_{\hat{\bu}} C_q^{(1-\tau)}C_{qd}^{\tau}N^{-(1-\tau)\alpha}N_d^{-\alpha_d \tau} \right]
    \end{aligned}
\end{equation}
with constants $\cost_{\hat{\bu}} = \|\bu\|^2_{L^2(D)} + \|\bu^{\ast}\|^2_{L^2(D)}$, $C_q = C_q\left(\|\res_{\bu,\theta^{\ast}}\|_{C^{k-2}(D)} + \|\res_{div,\theta^{\ast}}\|_{C^{k-1}(D)}\right)$ and $C_{qd} = C_{qd}\left(\|\res_{d,\theta^{\ast}}\|_{C^{k}(D^{\prime})}\right)$, given by the quadrature error bounds \eqref{eq:qassm} and \eqref{eq:dqassm} and $\tau$ given in \eqref{eq:stst}.
\end{lemma}
\begin{proof}
For notational simplicity, we denote $\res_{\bu} = \res_{\bu,\theta^{\ast}}, \res_{div} = \res_{div,\theta^{\ast}}$ and $\res_d = \res_{d,\theta^{\ast}}$. 

Define $\hat{\bu} = \bu^{\ast} - \bu \in H^1(D;\R^d)$ and $\hat{p} = p^{\ast} - p \in H^1(D)$, by linearity of the differential operator in \eqref{eq:st} and the data observable in \eqref{eq:dtst} and by definitions \eqref{eq:resst1}, \eqref{eq:resst2}, \eqref{eq:resdst}, we see that $\hat{\bu},\hat{p}$ satisfy,
\begin{equation}
    \label{eq:hst}
    \begin{aligned}
    \Delta \hat{\bu} + \nabla \hat{p} &= \res_{\bu}, \quad \forall x \in D, \\
    \div(\hat{\bu}) &= \res_{div}, \quad \forall x \in D, \\
    \hat{\bu} &= \res_d, \quad \forall x \in D^{\prime},
    \end{aligned}
\end{equation}
with Stokes' equation being satisfied in the sense of  \eqref{eq:wkst}. 

Hence, we can directly apply the conditional stability estimate \eqref{eq:stst} to obtain,
\begin{equation}
    \label{eq:plst1}
    \er^2_G(B_R) \leq C\left(\cost_{\res} + \cost_{\hat{\bu}}^\tau \cost_{\res}^{1-\tau} + \left(\cost_{\res}^{1-\tau}\cost_{\hat{\bu}}^{1-\tau}\right)\|\res_{d}\|_{L^2(D^{\prime})}^{2\tau}\right),
\end{equation}
with constants,
\begin{equation}
\label{eq:plst2}
\cost_{\res} = \|\res_{\bu}\|^2_{L^2(D)} + \|\res_{div}\|^2_{L^2(D)}, \quad \cost_{\hat{\bu}} = \|\bu\|^2_{L^2(D)} + \|\bu^{\ast}\|_{L^2(D)}.
\end{equation}
By observing that $\er_{d,T}^2$ is the quadrature approximation of $\|\res_{d}\|_{L^2(D^{\prime})}^{2}$ with the quadrature rule \eqref{eq:dqd} and $\er_{p,T}^2$ is the quadrature approximation of $\cost_{\res}$ with the quadrature ruel \eqref{eq:quad}, we can directly apply \eqref{eq:qassm} and \eqref{eq:dqassm} to obtain the desired inequality \eqref{eq:stbd}.
\end{proof}
\begin{remark}
The bound \eqref{eq:stbd} is more complicated than the corresponding generalization error estimates for the Poisson, Heat and Wave equations. Nevertheless, the same general structure can be observed. There are two types of terms, one based on the training errors $\er_{p,T},\er_{d,T}$, which can be computed a posteriori and should be made small during the training process. The other set of terms are decreasing in the number $N_{int},N_d$ of training samples. Thus the total error can be reduced by increasing the number of training points, while keeping the resulting training errors low.
\end{remark}
\begin{remark}
We observe that the generalization error in \eqref{eq:stbd} is given in terms of balls. Arbitrary sets can be readily included by considering the estimate in a union of balls containing that set. Moreover, the generalization error \eqref{eq:gerst} only considers errors in the velocity field. In principle, this estimate can be used to bound pressure errors by working with the corresponding pressure Poisson equation. 
\end{remark}

\begin{figure}[h!]
\centering

    \begin{subfigure}{.3\textwidth}
        \centering
        \includegraphics[width=1\linewidth]{{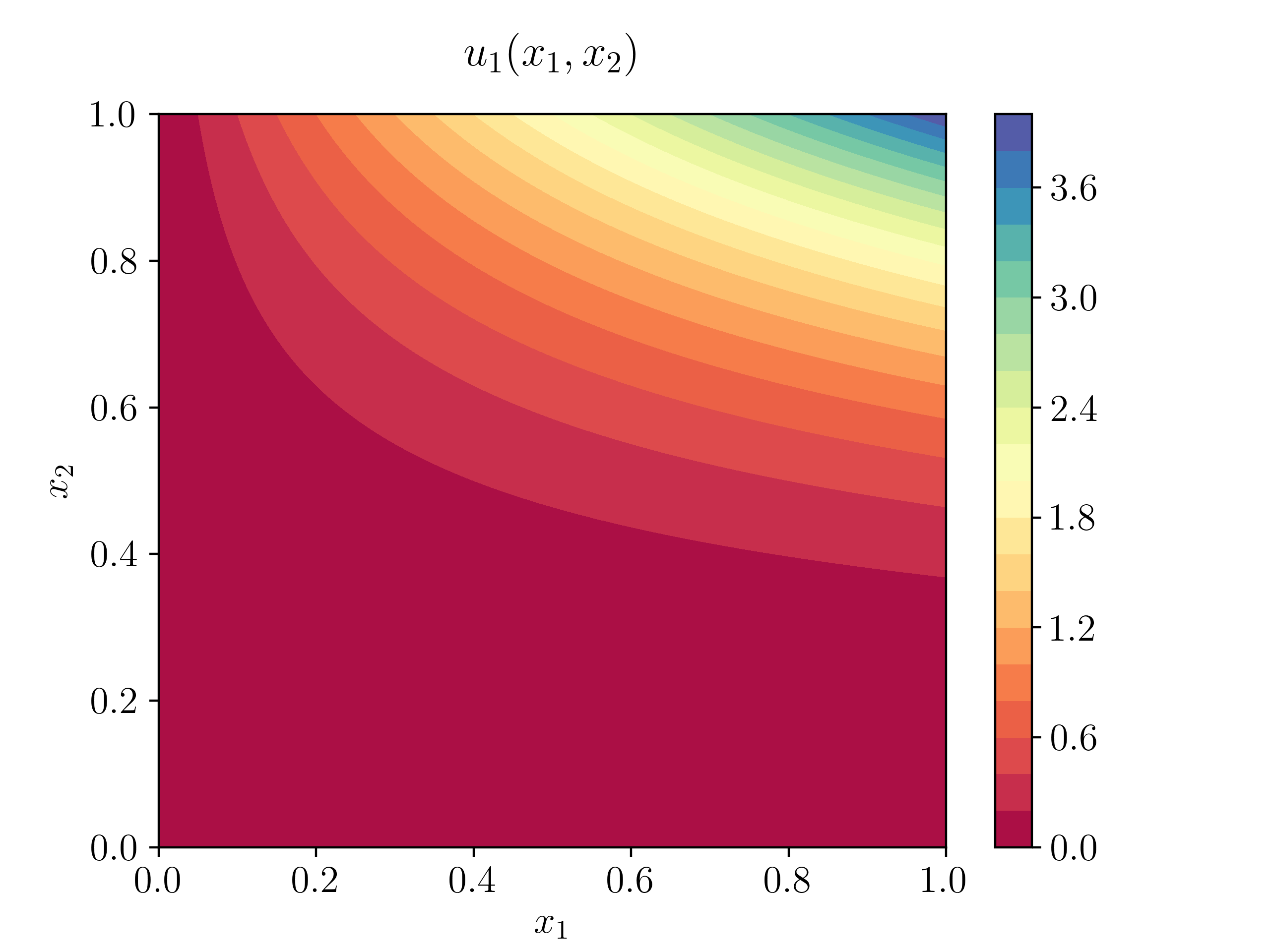}}
        \caption{Exact solution $\bu_1$}
    \end{subfigure}
      \begin{subfigure}{.3\textwidth}
        \centering
        \includegraphics[width=1\linewidth]{{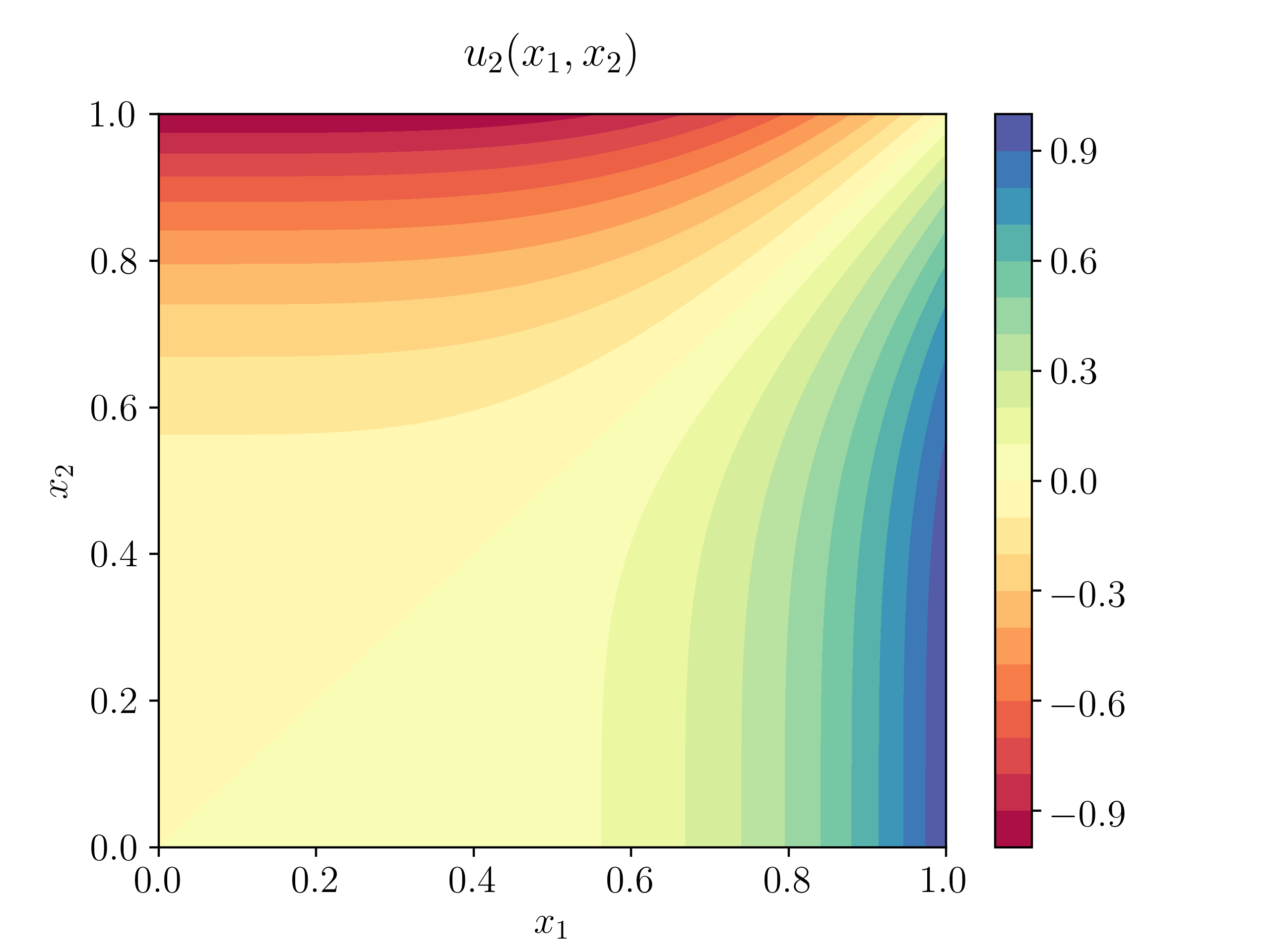}}
        \caption{Exact solution $\bu_2$}
    \end{subfigure}
    \begin{subfigure}{.3\textwidth}
        \centering
        \includegraphics[width=1\linewidth]{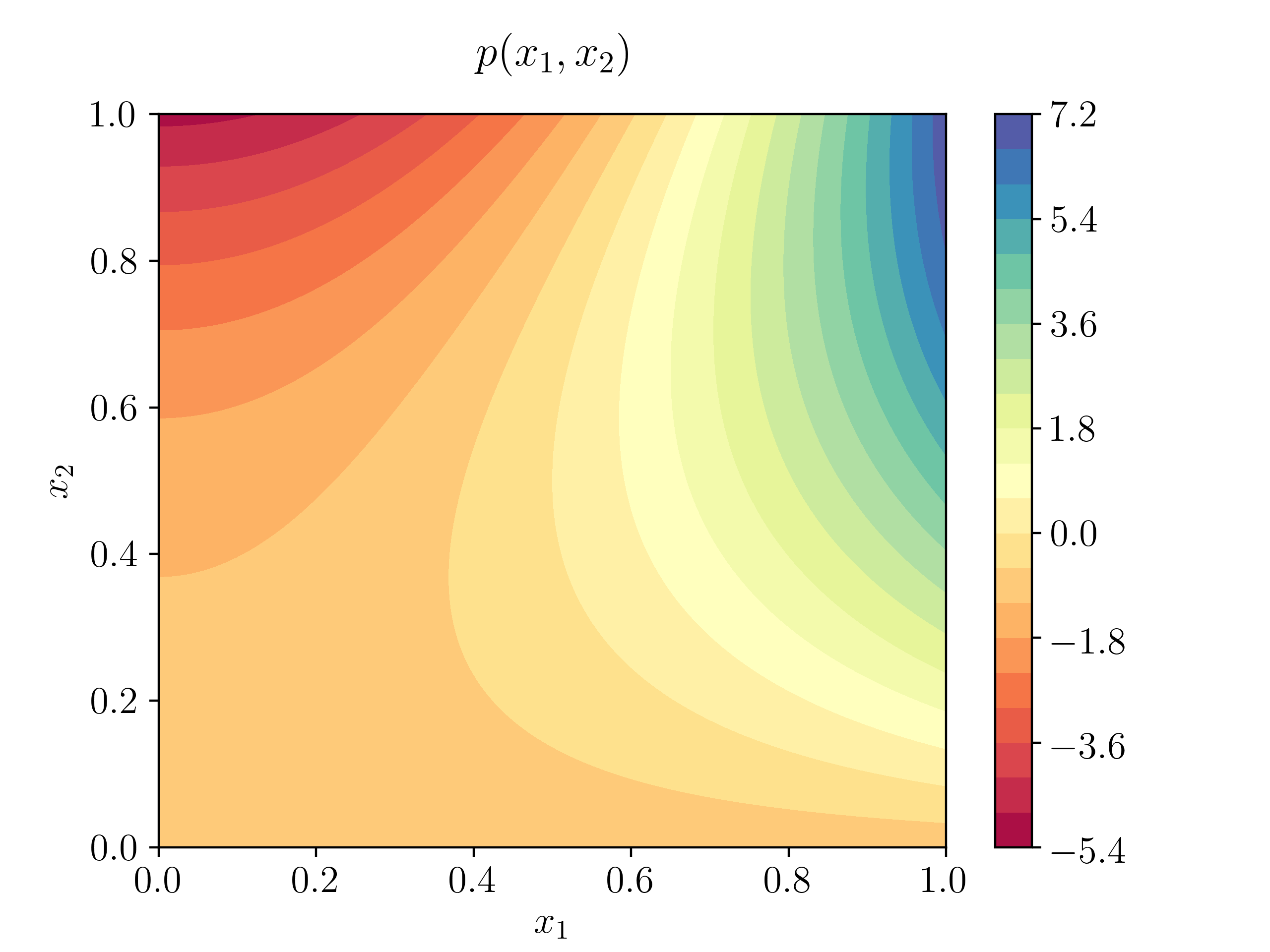}
        \caption{Exact solution $p$}
    \end{subfigure}
    
     \begin{subfigure}{.3\textwidth}
        \centering
        \includegraphics[width=1\linewidth]{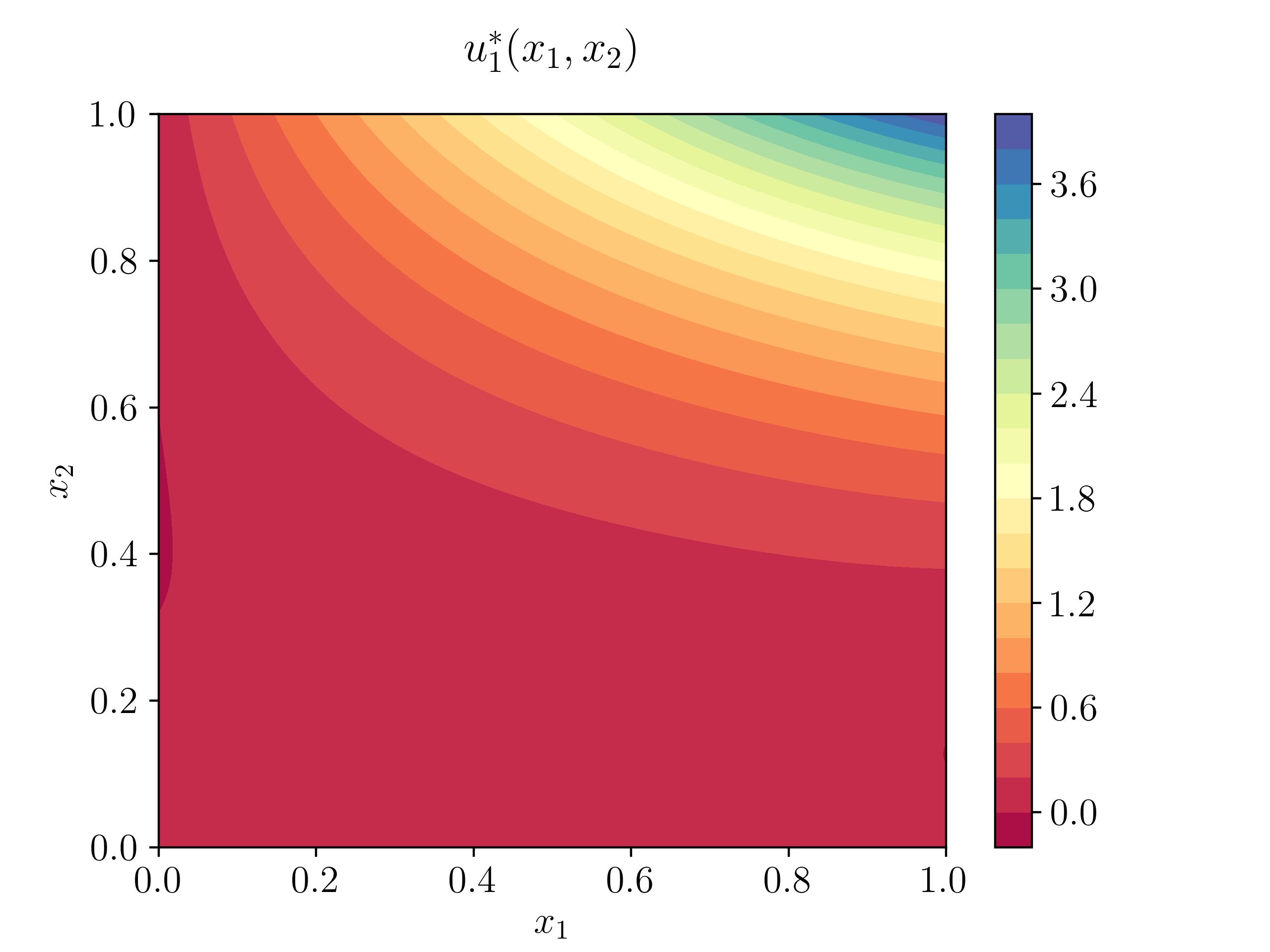}
        \caption{PINN $\bu_1^{\ast}$}
    \end{subfigure}
    \begin{subfigure}{.3\textwidth}
        \centering
        \includegraphics[width=1\linewidth]{{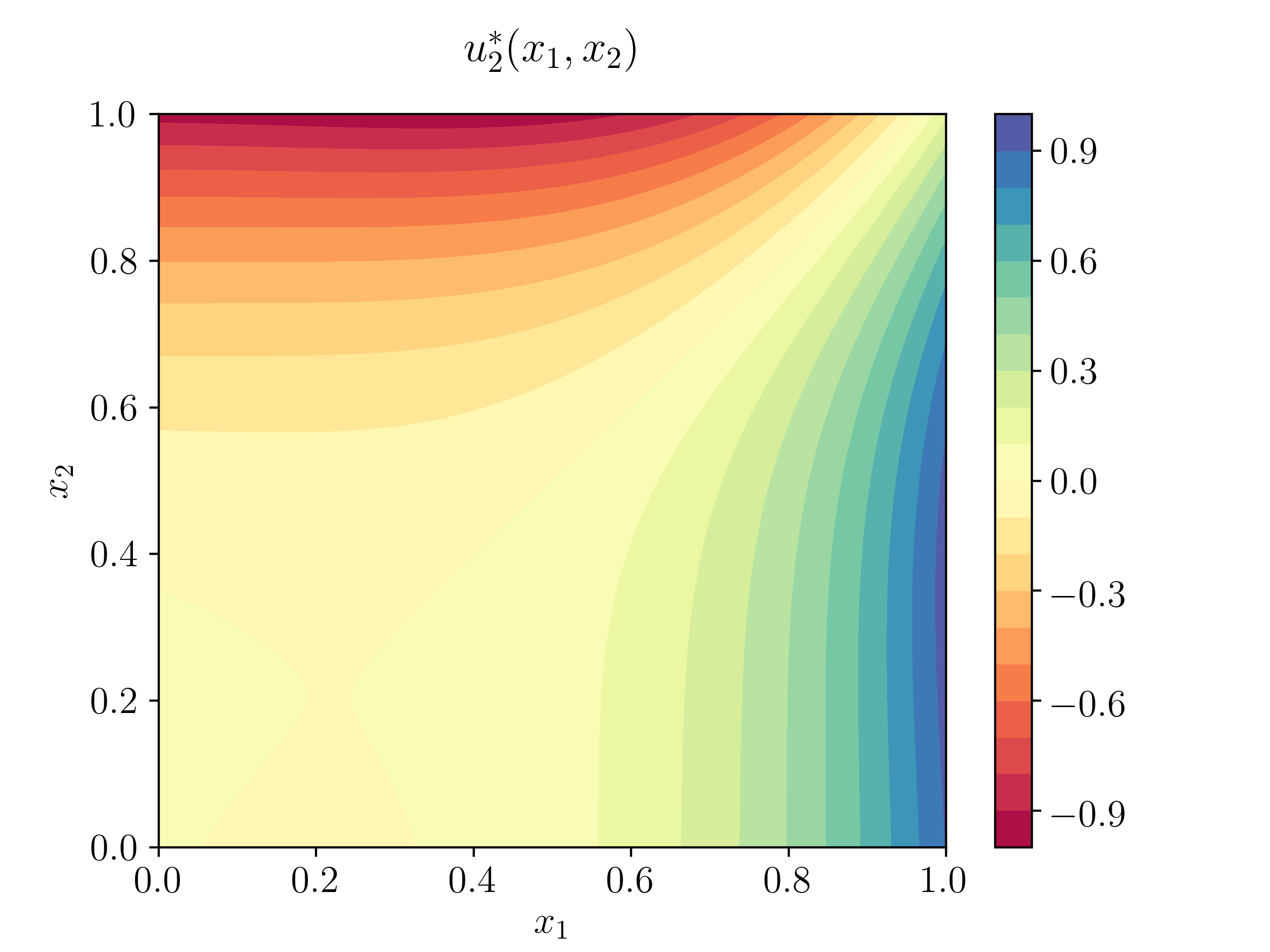}}
        \caption{PINN $\bu_2^{\ast}$}
    \end{subfigure}
    \begin{subfigure}{.3\textwidth}
        \centering
        \includegraphics[width=1\linewidth]{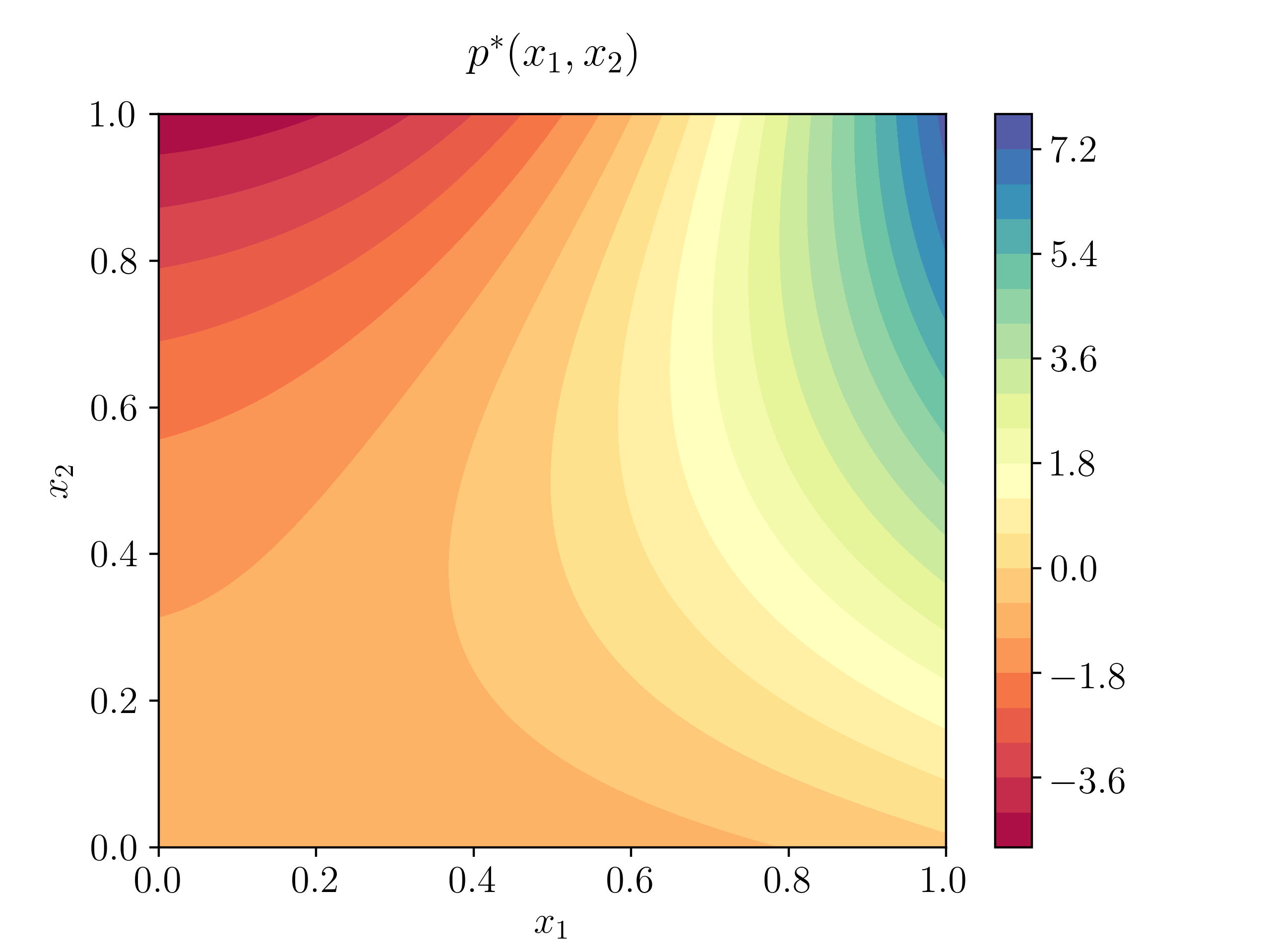}
        \caption{PINN $p^{\ast}$}
    \end{subfigure}
    
\caption{Comparison of the Exact solution and the PINN approximations  for the inverse problem for the Stokes equation, with $N = 20 \times 20$ training points.}
\label{fig:st2}
\end{figure}
\subsection{Numerical experiments}
We follow \cite{BH1} and consider the homogeneous version of Stokes equation i.e $\f,f_d \equiv 0$, in the two-dimensional domain $(0,1)^2$. We consider the exact solutions,
\begin{equation}
\label{eq:stex}
    \bu(x_1,x_2)=(4x_1x_2^3, x_1^4-x_2^4), \quad p(x_1,x_2)=12x_1^2x_2 - 4x_2^3 - 1,
\end{equation}
and define the data term $\g = \bu|_{D^{\prime}}$ on the sub-domain,
\begin{equation}
\label{eq:stdom}
D^{\prime} = \{(x_1,x_2)\in \R^2 : \sqrt{(x_1-0.5)^2 + (x_2-0.5)^2} <0.25\}.
\end{equation}
The domain $D$ and the sub-domain $D^{\prime}$ are illustrated in figure \ref{fig:st1}. We chose Sobol points as the training set $\train_{int}$ and  equally spaced (Cartesian) radial points as the training set $\train_{d} \subset D^{\prime}$. See figure \ref{fig:st1} for an illustration of these training sets. 

With these training sets, we run algorithm \ref{alg:PINN} with loss function \eqref{eq:lf2}, \eqref{eq:lfst} to obtain PINNs approximating the data assimilation problem for the Stokes equation \eqref{eq:st}, \eqref{eq:dtst}. The results with a PINN, trained with $N=40\times40$ and $N_d = rN $ points, with $r = {\rm meas}(D^{\prime})/{\rm meas}(D)$, are shown in figure \ref{fig:st2}, where we plot the exact and PINNs solutions for the velocity components $\bu_1$, $\bu_2$ and pressure $p$. We see from this figure that the PINN is approximating the exact solution quite well, at least to the eye. 

In order to quantify the approximation abilities of PINN for this example, in Table \ref{tab:st}, we present the training errors and relative percentage errors of $\|\bu-\bu^{\ast}\|_{L^2(D)}$ and $\|p-p^{\ast}\|_{L^2(D)}$. From this table, we observe a slow, yet consistent, decay of the $L^2$-errors of the velocity field. This is consistent with the estimate \eqref{eq:stbd} on the generalization error. More surprisingly, we also observe from the Table \ref{tab:st} that the pressure errors are reasonably small and decay with increasing numbers of training samples. Note that we did not provide any estimate on pressure errors, nor was pressure observed in the data in \eqref{eq:dtst}. Yet, given that pressure is a Lagrange multiplier for the Stokes equations and we can control pressure in terms of the velocity through the pressure Poission equations, we observe good approximation for the pressure by PINNs. If the pressure was also specified in the observation domain $D^{\prime}$, the data residual \eqref{eq:resdst}, and the loss function \eqref{eq:lfst} can de readily modified and the resulting PINN is likely to lead to a better approximation of the pressure.

However, the amplitude of the error, for both the velocity field and the pressure,  is higher than the other three equations that we considered earlier in the paper. This is not unexpected as the observation domain $D^{\prime}$ is smaller in this case (see figure \ref{fig:st1}) and has approximately $20\%$ of the area of the whole domain $D$. To be able to reconstruct the velocity and pressure fields with reasonable accuracy from such a small observation domain is significant. Note that we have used the best retraining i.e, the PINN with the smallest training loss among the different random initializations of the optimizer. Choosing the average over all the retrainings led to slightly greater generalization error. 
\begin{table}[htbp] 
    \centering
    \renewcommand{\arraystretch}{1.1} 
    
    \footnotesize{
       \begin{tabular}{ c c c c  c c c c } 
            \toprule
            $N$  &\bfseries $K-1$ & \bfseries $\tilde{d}$  &$\lambda_{reg}$ &\bfseries $\lambda$&  $\er_T$ &  $||\bu - \bu^\ast||_{L^2}$   &   $||p - p^\ast||_{L^2}$ \\ 
            \midrule
             \midrule 
            $20\times20$   & 4 & 24&0.0&0.001 & 0.0007 & 2.3 \% & 5.6 \%  \\
            \midrule
            $40\times40$  &4 & 24&0.0&0.001 & 0.0004& 1.7 \%   &4.0 \%\\
            \midrule 
            $80\times80$   & 4 & 20&0.0&0.01 &0.00046  &  1.6 \%  &3.5 \% \\
         
            \bottomrule
        \end{tabular}
        \caption{Stokes equation: Relative percentages errors for different numbers of training points.}
        \label{tab:st}
    }
\end{table}

\section*{Code} The building and the training of PINNs, together with the ensemble training for the selection of the model hyperparameters, are performed with a collection of Python scripts, realized with the support of PyTorch \url{https://pytorch.org/}. The scripts can be downloaded from \url{https://github.com/mroberto166/PinnsSub}.

\section{Discussion}
\label{sec:7}
Inverse problems for PDEs are very challenging computationally on account of the difficulties of finding stable regularizations and possibly high computational cost. Physics informed neural networks have recently been very successfully applied to efficiently simulate the solutions of several inverse problems for (nonlinear) PDEs, \cite{KAR2,KAR4,KAR6} and references therein. However, rigorous justifications are currently lacking. 

Our main aim in this paper was to provide a rigorous justification for using PINNs to approximate the solutions of a large class of inverse problems for PDEs. To this end, we focussed on the so-called \emph{data assimilation} or \emph{unique continuation} inverse problems in this paper. For these problems, the inputs necessary for the forward problem of the underlying PDE (for instance the abstract pde \eqref{eq:pde}), such as initial and boundary conditions are unknown. However, one has access to (possibly noisy) observables of the solution (\eqref{eq:dt} for the abstract pde \eqref{eq:pde}), on the so-called \emph{observation domain}, which a subset of the underlying domain. Such data assimilation problems arise in many applications, particularly in geophysics and meteorology. 

In section \ref{sec:2}, we posed the data assimilation inverse problem for the abstract PDE \eqref{eq:pde}, \eqref{eq:dt}. The corresponding algorithm \ref{alg:PINN} to generate a PINN for approximating this inverse problem was proposed. Key ingredients in this algorithm included using quadrature points, corresponding to underlying quadrature rules \eqref{eq:quad} as \emph{training points}, and the (strong form of) PDE residual \eqref{eq:res1}, collocated at these training points as the loss function \eqref{eq:lf1}, for training the neural network. Under the assumption that solutions to the underlying inverse problem \eqref{eq:pde}, \eqref{eq:dt}, satisfy a \emph{conditional stability estimate} \eqref{eq:assm}, we were able to prove a rigorous upper bound on the generalization (approximation) error \eqref{eq:egen} of the PINN, in terms of the (computable) training error \eqref{eq:etrn} and the number of training samples (with a rate prescribed by the quadrature errors \eqref{eq:qassm},\eqref{eq:dqassm}). Thus, we are able to leverage the conditional stability estimate on the inverse problem for the underlying PDE into a bound on the generalization error for the approximating PINN. 

We illustrate this abstract framework for concrete examples for linear PDEs that arise in a wide variety of models. We consider the Poisson, heat, wave and Stokes equations in this paper. All these PDEs possess conditional stability estimates for the underlying data assimilation inverse problem, proved either by the well-known \emph{three balls inequalities} or \emph{Carleman estimates}. We adapt the abstract formalism to each example and provide concrete estimates on the generalization error. Numerical experiments are presented and validate the proposed theory for each of the linear PDEs considered here. We observe from the numerical experiments that PINNs are very efficient at approximating the underlying inverse problem. The resulting errors are very small and are less than $1\%$, even for a few training samples.

It is instructive at this point to compare the results obtained with PINNs to those with other methods. We select stabilized non-confirming finite element methods, such as the ones presented in \cite{BO1,BO2,BO3,BH1} and references therein. In these papers, the authors also employed conditional stability estimates to obtain rigorous error estimates for the finite element method. Comparing our numerical results to these papers, we notice that the PINNs were comparable in terms of the amplitude of the error, when compared to finite element methods on fine meshes. On the other hand, PINNs are extremely simple to program, with the same level of code complexity for both the forward and the inverse problem \cite{KAR2}. Moreover, PINNs are able to obtain very small errors with a training time of less than $1-2$ minutes for most examples presented here. Given their simplicity and efficiency, we believe PINNs constitute a promising alternative to available methods for the data assimilation problem. Another interesting finding was the ability of PINNs to solve inverse problems in very high dimensions. We illustrated this performance for the heat equation in domains upto $100$ space dimensions and observed that PINN, based on random training points, were able to approximate the solution of the inverse problem to high accuracy and without incurring the \emph{curse of dimensionality}. 

Summarizing, we are able to provide a unified framework and to the best of our knowledge, \emph{the first rigorous justification} for using PINNs in the context of data assimilation inverse problems for PDEs. 

The current paper, as well as the general methodology of PINNs, does have a few shortcomings that we point out below,
\begin{itemize}
    \item The bound \eqref{eq:egenb} on the generalization error \eqref{eq:egen} is in terms of the training error. We do not estimate this error here. However, this is standard in machine learning \cite{MLbook2}, where the training error is computed \emph{a posteriori} (see the numerical results for training errors in respective experiments). The training error stems from the solution of a non-convex optimization problem in very high dimensions and there are no robust estimates on this problem. Once such estimates become available, they can be readily incorporated in the rhs of \eqref{eq:egenb}. For the time being, the estimate \eqref{eq:egenb} should be interpreted as \emph{as long as the PINN is trained well, it generalizes well}. This is indeed borne out by the results of the numerical experiments.
    \item A particular manifestation of this issue with the training error was encountered in some of the numerical experiments, where the training error reached a very low value, even for a few training samples. Adding further samples led to a saturation of the already low generalization error as the training error could not be reduced further. If this happens, one needs to either change the architecture of the neural network, for instance by using residual or convolutional neural networks or change the optimizer or use standard tricks in machine learning such as batch normalization and dropout to further reduce the training and generalization errors.
    \item Although our abstract formalism was very general and covered nonlinear PDEs, we only considered concrete examples of linear PDEs in this paper. This is largely on account of simplicity and the ready availability of conditional error estimates. On the other hand, some of the striking success of PINNs in the context of inverse problems, has been for nonlinear PDEs such as the Navier-Stokes equations \cite{KAR4} and even the compressible Euler equations \cite{KAR6}. Although some Carleman estimates, leading to conditional stability bounds, are available for the Navier-Stokes equations \cite{Im2}, applying them to derive bounds on the generalization error is very challenging and will be considered in a forthcoming paper. However, extending our results to other linear PDEs such as Maxwell's equation, Helmholtz equation, advection-difusion equations and even the linearized Navier-Stokes equations, where conditional stability estimates of the form \eqref{eq:assm} are available, is relatively straightforward. 
    \item We have focused on the \emph{deterministic data assimilation problem} here. In practice, data in the observation domain \eqref{eq:dt} is \emph{noisy}. For small amplitude noise, the estimates can be readily extended to bound the generalization error in \eqref{eq:egenbn} and the efficiency of PINNs in this regime is validated in a numerical experiment for the Poisson's equation with noisy data. However, dealing with data, perturbed by larger amplitudes of noise, necessitates a Bayesian framework, such as the one considered in \cite{KAR9}. Deriving rigorous estimates on the generalization error in a Bayesian framework will be considered in a future paper.

\end{itemize}

\section*{Acknowledgements.} The research of SM and RM was partially supported by European Research Council Consolidator grant ERCCoG 770880: COMANFLO.

\bibliographystyle{abbrv}
\bibliography{MM2refs}

\end{document}